\documentclass[11pt,reqno]{amsart}
\usepackage{graphicx,amsmath,amsthm,verbatim,amssymb,lineno,titletoc}
\usepackage{mathrsfs}
\usepackage{adjustbox}
\usepackage{ytableau}
\usepackage{bbm}
\usepackage{hyperref}
\hypersetup{colorlinks}
\usepackage[numbers]{natbib}
\usepackage{array}
\usepackage{tikz}
\usetikzlibrary{shapes.multipart,patterns,arrows}

\usepackage[font=scriptsize]{caption}

\usepackage{amssymb}
\usepackage[T1]{fontenc} 
\usepackage{fourier} 
\usepackage[english]{babel} 
\usepackage{amsmath,amsfonts,amsthm,bbm} 
\usepackage{appendix}

\usepackage{amssymb}
\usepackage[all]{xy}
\usepackage{enumerate}
\usepackage{tikz}
\usepackage{mathrsfs}
\usepackage[english]{babel}
\usepackage{multirow}
\usepackage{float}
\usepackage{enumitem}
\usepackage{graphicx,accents}

\usepackage{algorithm}
\usepackage[noend]{algpseudocode}

\newcolumntype{M}[1]{>{\centering\arraybackslash}m{#1}}
\newcolumntype{N}{@{}m{0pt}@{}}

\pdfoutput=1

\newtheorem{theorem}{Theorem}[section]
\newtheorem{proposition}[theorem]{Proposition}

\newtheorem{lemma}[theorem]{Lemma}
\newtheorem{corollary}[theorem]{Corollary}

\newcommand{\edit}[1]{{\bf \textcolor{blue}
		{[#1\marginpar{\textcolor{red}{***}}]}}}

\theoremstyle{definition}

\newtheorem{innercustomremark}{Remark}
\newenvironment{customremark}[1]
{\renewcommand\theinnercustomremark{#1}\innercustomremark}
{\endinnercustomremark}

\allowdisplaybreaks

\begin{document}
	
	\title[Synchronization of PCOs on trees]{Global synchronization of pulse-coupled oscillators on trees}

	\author{Hanbaek Lyu}
	\address{Hanbaek Lyu, Department of Mathematics, The Ohio State University, Columbus, OH 43210.}
	\email{\texttt{yu.1242@osu.edu}}

	\date{\today}
	

	\begin{abstract}
	     Consider a distributed network on a finite simple graph $G=(V,E)$ with diameter $d$ and maximum degree $\Delta$, where each node has a phase oscillator revolving on $S^{1}=\mathbb{R}/\mathbb{Z}$ with unit speed. Pulse-coupling is a class of distributed time evolution rule for such networked phase oscillators inspired by biological oscillators, which depends only upon event-triggered local pulse communications. In this paper, we propose a novel inhibitory pulse-coupling and prove that arbitrary phase configuration on $G$ synchronizes by time $51d$ if $G$ is a tree and $\Delta \le 3$. We extend this pulse-coupling by letting each oscillator throttle the input according to an auxiliary state variable. We show that the resulting adaptive pulse-coupling synchronizes arbitrary initial configuration on $G$ by time $83d$ if $G$ is a tree. As an application, we obtain a universal randomized distributed clock synchronization algorithm, which uses $O(\log \Delta)$ memory per node and converges on any $G$ with expected worst case running time of $O(|V|+(d^{5}+\Delta^{2})\log |V|)$. 
	\end{abstract}

	\maketitle

	\section{Introduction}

Complex systems consist of a large number of locally interacting agents and exhibit a wide range of collective emergent behaviors without any centralized control mechanism. A prime example of such is a population of coupled oscillators, where mutual efforts to synchronize their phases or frequency aggregate and the whole population eventually reaches a synchronized state \cite{strogatz2001exploring}. Examples from nature include blinking fireflies \cite{buck1938synchronous}, neurons in the brain \cite{tateno2007phase}, and the circadian pacemaker cells \cite{enright1980temporal}. These biological oscillators, commonly modeled as pulse-coupled oscillators (PCOs) in dynamical systems literature, communicate in an extremely efficient way. Namely, they send pulses to all of their neighbors on a particular event (e.g., firing of neurons, blinking of fireflies, etc.), irrespective of any kind of global information, and oscillators receiving such pulses adjust their phase accordingly.

The study of emergence of synchrony from a system of coupled oscillators has a natural and strong connection to the field of distributed algorithms in theoretical computer science. A distributed algorithm is a time evolution rule for a system of processors by which the population collectively performs higher-level tasks, which may go beyond the capacity of an individual agent. A fundamental building block of such distributed algorithm is to synchronize local clocks in each processors so that a shared notion of time is available across the network. Especially in modern wireless sensor networks, clock synchronization is a fundamental issue \cite{potdar2009wireless, akyildiz2002wireless} and poses challenging theoretical questions \cite{sundararaman2005clock}. Namely, a large number of low-cost and low-power wireless sensors are deployed over a large region in an ad-hoc manner, and they are supposed to collectively perform tasks such as monitoring and measuring large-scale phenomena \cite{pousttchi2009applicat}. Due to the limited resource in each sensor compared to the large size of the network, scalable algorithms for clock synchronization are becoming more valuable. Pulse couplings have their inherent efficiency and scalability, which inspired many efficient algorithms for clock synchronization algorithms for wireless sensor networks \cite{hong2005scalable, pagliari2011scalable, wang2012energy, wang2013increasing}.

The main contribution of this work is a rigorous derivation of the emergence of synchrony in a population of inhibitory PCOs on finite trees. A fundamental issue in inhibitory PCOs on trees is that nodes with large degree may receive input pulses so often that its phase is constantly inhibited and it may never send pulses to its neighbors. This divides the tree into non-communicating components so global synchrony may not emerge. Our key innovation is to equip each oscillator with an auxiliary state variable which may throttle the input pulse. This effectively breaks symmetries in local configurations on finite trees and reduces the relevant diameter to consider in constant time. Our main result shows that on an arbitrary finite tree with diameter $d$, \textit{arbitrary} initial configuration synchronizes by time $83d$. We believe that this is the first time that synchronization of PCOs from every configuration on tree networks is derived rigorously, especially with an explicit bound on the convergence time that is optimal up to a constant factor. Our result parallels the well-known results of Mirollo and Strogatz \cite{mirollo1990synchronization} and Klingmayr and Bettsetter \cite{klinglmayr2012self} on PCOs on all-to-all networks.  

We also implement our pulse-coupling as a distributed clock synchronization algorithm via a suitable discretization procedure. Consider a distributed network of processors over a communication network $G=(V,E)$ with diameter $d$ and maximum degree $\Delta$. Our algorithm is extremely efficient in that each node uses only $O(1)$ bits of memory and $O(1)$ bits of information exchange during a unit time. Moreover, it converges to synchrony from arbitrary initial configuration in $O(d)$ times if $G$ is a finite tree. In order to overcome the restriction on trees, we make it a multi-layer algorithm by composing with a distance $\le 2$ coloring and spanning tree algorithms (both of which use randomization). Running the three algorithms simultaneously, this gives us a memory-efficient ($O(\log\Delta)$ memory per node), self-stabilizing (global convergence from any initial configuraiton), with asymptotically linear worst case expected running time. (See the simulation in Figure \ref{fig:simulation}.) The composite algorithm is scalable and can be applied to dynamic and growing networks, as long as the maximum degree is bounded. To the author's knowledge, our algorithm is the first truly scalable and self-stabilizing clock synchronization algorithm with asymptotically linear expected running time.

\begin{figure*}[h]
	\centering
	\includegraphics[width = 1 \linewidth]{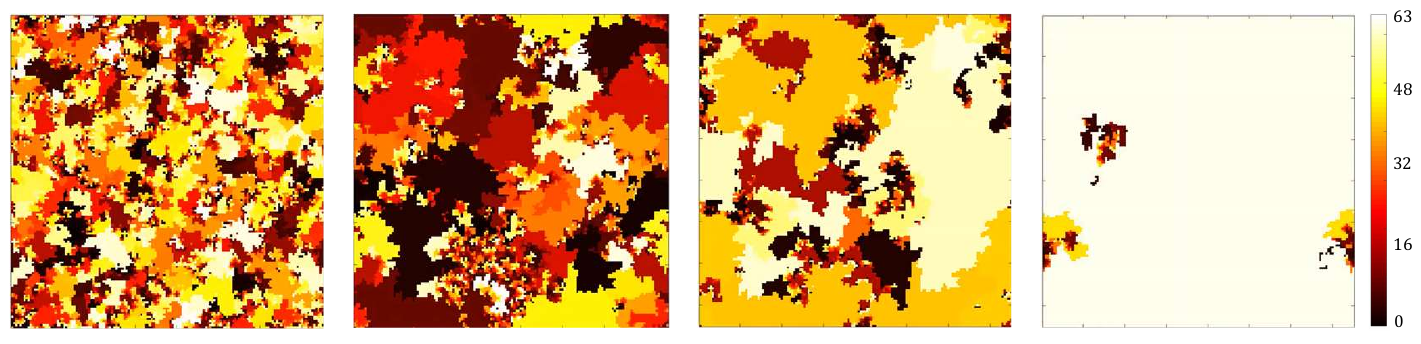}
	\caption{ Snapshots of the adaptive 4-coupling modulo $M=64$ (definition given in Section \ref{section:application}) on a uniform spanning tree of the $150 \times 150$ torus with Moore neighborhood, where each lattice point has 8 neighbors. Colors indicate phases (in $\mathbb{Z}_{64}$) of each node (see color bar on the right). Initial configuration is generated from the uniform product measure, and the snapshots are taken at times $t=$ 32, 64, 250, and 470 seconds, from left to right; The time scale is so that a full oscillation cycle on $\mathbb{Z}_{64}$ of an isolated node takes exactly 1 second. The system reaches global synchrony by time $t=563$ seconds. The readers are referred to the author's \href{http://www.hanbaeklyu.com/publications/}{website} for a supplementary movie of the full compelling phase dynamics.
	}
	\label{fig:simulation}
\end{figure*}

This paper is organized as follows. We give definitions of our models and statements of the main results in Section \ref{section:Definitions and statement of results}. A brief survey on related works is given in Section \ref{Section:Related works}. We introduce an alternative representation of our models in Section \ref{Section:Relative circular representation}, which will be the basis of rigorous analysis. In Section \ref{Section:The width lemma and the branch width lemma}, we prove some preliminary observations and establish two fundamental statements (Lemma \ref{lemma:branchwidth} and Proposition \ref{prop:branch_excitation}). Then we prove our main results (Theorem \ref{4treethm} and \ref{A4Ctreethm}) in Section \ref{Section:Proof of main results} assuming three lemmas concerning local limit cycles (Lemmas \ref{lemma:branch_attraction}, \ref{lemma:branchorbit_a}, and \ref{lemma:branchorbit_b}); These lemmas will be proved in the following section, Section \ref{section:locallemmas}. We devote Section \ref{section:application} to an application of our results to distributed clock synchronization algorithms. In the section, we implement our pulse-coupling as a distributed algorithm (with pseudocodes provided), and show the resulting algorithm has desired properties in Theorem \ref{thm:implementation}. Moreover, by composing with a randomized distance $\le 2$ coloring and spanning tree algorithms, we obtain a universal clock synchronization algorithm, whose performance is analyzed in Corollary \ref{thm:clocksync}. Finally, some concluding remarks are given in Section \ref{Section:Concluding remarks}.

\vspace{0.3cm}
\section{Definitions and  statement of results}
\label{section:Definitions and statement of results}

First we introduce some general terminology. Fix a graph $G=(V,E)$. For any subgraph $H\subseteq G$, we denote by $V(H)$ and $E(H)$ its node set and edge set, respectively. For two subgraphs $S,H\subseteq G$, we denote by $S-H=S-V(H)$ the graph obtained by deleting all nodes (along with all incident edges) of $H$ from $S$. For each node $v\in V$, denote by $N(v)$ its set of neighbors. The \textit{degree} of $v$ is defined by $\deg(v)=|N(v)|$. A graph $H$ is called a \textit{path} if it has at most two nodes of degree at most 1 and all the other nodes have degree 2. For any two distinct nodes $u,v\in V$, we say a path $P\subseteq G$ is \textit{from $u$ to $v$ (or $v$ to $u$)} if $u,v\in V(P)$ and $\deg(u)=1=\deg(v)$. The \textit{diameter} of $G$ is defined by $\text{diam}(G) = \max_{P\subseteq G} |E(P)|$, where $P$ runs over all paths in $G$.

A \textit{phase configuration} at time $t$ is a map $\phi_{\bullet}(t):V\rightarrow S^{1}:=\mathbb{R}/\mathbb{Z}$. A \textit{coupling} is a deterministic time evolution rule for phase configurations, which determines for each $G$ and initial phase configuration $\phi_{\bullet}(0)$ on $G$ a \textit{trajectory} (or \textit{orbit}) $(\phi_{\bullet}(t))_{t\ge 0}$. We say the trajectory \textit{synchronizes} by time $\tau\ge 0$ if for all $u,v\in V$ and $t\ge \tau_{0}$, we have $\phi_{v}(t)\equiv \phi_{u}(t)\mod 1$. We call the unit of time a ``second''. Let $\Omega$ be a metric space. For any function $f:\mathbb{R}\rightarrow \Omega$ and $t\in \mathbb{R}$, we denote $f(t^{\pm})=\lim_{s\rightarrow t\pm 0} f(s)$ in case the corresponding limit exists.

\vspace{0.2cm}
\subsection{The 4-coupling.} We first introduce a prototypical pulse-coupling, which generalizes the $4$-color firefly cellular automaton introduced earlier by the author in \cite{lyu2015synchronization}. For each $v\in V$ and $t\ge 0$, define three events $B_{v}(t)$, $R_{v}(t)$, $P_{v}(t)\subseteq S^{1}\times S^{1}$ by 
\vspace{0.1cm}
\begin{eqnarray*}
	B_{v}(t) &=& \{ \phi_{v}(t^{-})=1 \}\cap \{\phi_{v}(t)=0\}   \\
	R_{v}(t) &=& \{ \text{$\exists u\in N(v)$ s.t. $B_{u}(t)$ occurs} \}\\
	P_{v}(t) &=& R_{v}(t) \cap \{ 0< \phi_{v}(t) \le 1/2 \}.
\end{eqnarray*}
We say $v$ \textit{blinks}, \textit{receives a pulse}, and \textit{gets pulled} at time $t$ if $B_{v}(t)$, $R_{v}(t)$, and $P_{v}(t)$ occurs, respectively. For two adjacent nodes $u,v\in V$, we say that $v$ \textit{receives a pulse from $u$} (resp., \textit{gets pulled} by $u$) at time $t$ if $B_{u}(t)\cap R_{v}(t)$ (resp., $B_{u}(t)\cap P_{v}(t)$) occurs. A \textit{pulse-coupling} is a coupling determined by the following equations
\begin{eqnarray*}
	\begin{cases}
		\dot{\phi}_{v}(t) \equiv 1 & \text{if $R_{v}(t)$ does not occur}\\
		\phi_{v}(t^{+}) = f(\phi_{v}(t))  & \text{if $R_{v}(t)$ occurs},
	\end{cases}
\end{eqnarray*}
where $f:[0,1]\rightarrow [0,1]$ is a prescribed \textit{phase response curve} (PRC). In words, each PCO evolves on the unit circle (say, clockwise) with unit speed, resetting its phase to $f(\phi_{v}(t))$ upon receiving a pulse at time $t$. Note that for a trajectory $(\phi_{\bullet})(t))_{\ge 0}$ evolving through a pulse-coupling, for each node $v\in V$, $\phi_{v}(t)-t$ is a left-continuous step function in $t$. We say the pulse-coupling is \textit{inhibitory} (resp., \textit{excitatory}) if $f(x)\le x$ (resp., $f(x)\ge  x$) for all $x\in [0,1]$. The \textit{4-coupling} is an inhibitory pulse-coupling with the following PRC $f_{0}:[0,1]\rightarrow [0,1]$ (see Figure \ref{PRC_4C}),
\begin{equation*}
f_{0}(x) = 0\cdot  \mathbf{1}[0,1/4](x)+ (x-1/4)\cdot  \,\mathbf{1}[1/4, 1/2](x) + x\cdot \mathbf{1}(1/2,1](x),
\end{equation*}
where $\mathbf{1}(A)$ denotes the indicator function of event $A$, taking value 1 on $A$ and 0 otherwise. 

\begin{figure*}[h]
	\centering
	\includegraphics[width = 0.25 \linewidth]{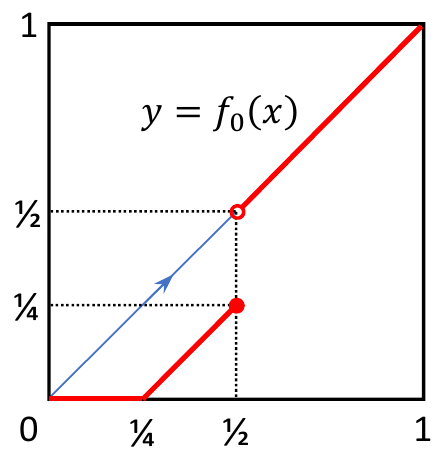}
	\caption{ PRC of the 4-coupling represented as the red graph.  
	}
	\label{PRC_4C}
\end{figure*}   

For an illustration, we give a non-synchronizing example of the 4-coupling on a star with four leaves in Figure \ref{fig:ex_4C}.

\begin{figure*}[h]
	\centering
	\includegraphics[width = 0.8 \linewidth]{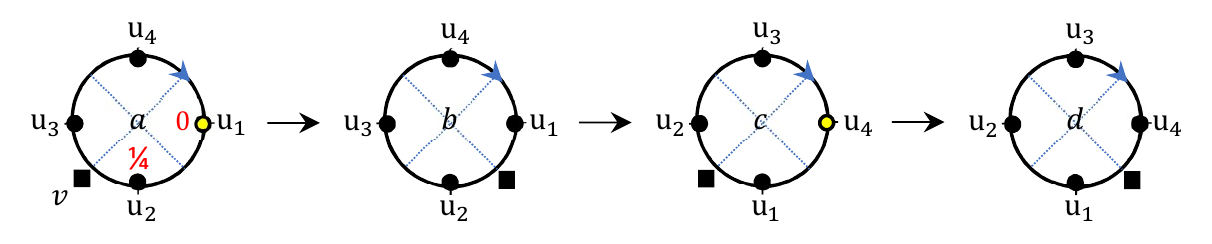}
	\caption{ An example of 4-coupled phase dynamics on a star with center $v=\blacksquare$ and leaves $=\bullet$. The circle represents the phase space $S^{1}=\mathbb{R}/\mathbb{Z}$ with two points $0$ and $1/4$ are indicated in red. Nodes revolve on the circle in clockwise and blinking nodes are indicated as yellow. In every 1/4 second, one of the leaves blink and pulls the center by 1/4 in phase, resulting in a non-synchronizing orbit. 
	}
	\label{fig:ex_4C}
\end{figure*}

Our first main result is stated in the following theorem:

\begin{theorem}\label{4treethm}
	Let $T=(V,E)$ be a finite tree with diameter $d$. 
	\vspace{0.1cm}
	\begin{description}[noitemsep]
		\item[(i)] If $T$ has maximum degree $\le 3$, arbitrary phase configuration on $T$ synchronizes by time $51d$.
		\vspace{0.1cm}
		\item[(ii)] If $T$ has maximum degree $\ge 4$, then there exists a non-synchronizing phase configuration on $T$. 
	\end{description}
\end{theorem}

\vspace{0.5cm}
\subsection{The adaptive 4-coupling.}

Next, we extend the 4-coupling in order to surpass its degree constraint as stated in Theorem \ref{4treethm} (ii). Roughly speaking, the idea is to append each oscillator with an auxiliary state variable, according to which the input is throttled. Here is a brief sketch. 
\vspace{0.1cm}
\begin{description}
	\item{$\bullet$} When `rested', each oscillator follows the 4-coupling.  
	\vspace{0.1cm}
	\item{$\bullet$} If a rested oscillator gets pulled too many times in a moving time window of length 1 second, then it becomes `excited' and enters `refractory period' right on. 
	\vspace{0.1cm}
	\item{$\bullet$} During refractory period, each oscillator ignores all pulses and becomes rested again upon blinking twice.  
\end{description}
\vspace{0.1cm}

To rigorously define this adaptive pulse-coupling, we may extend the state space of each node to $\Omega:=S^{1}\times S^{1} \times (\{ 1,3 \}\times \mathbb{Z}_{2}\times \mathbb{Z}_{4})\times \mathbb{Z}_{3}$. A \textit{joint configuration} at time $t$ is a map $\Sigma_{\bullet}(t):V\rightarrow \Omega$, $\Sigma_{\bullet}(t):=(\phi_{\bullet}(t),\beta_{\bullet}(t),\mu_{\bullet}(t),\sigma_{\bullet}(t))$.  An \textit{event} is a subset $A\subseteq \Omega^{2}$. For each $v\in V$ and $t\ge 0$, denote $\mu_{v}(t)=(\mu_{v}^{1}(t),\mu_{v}^{2}(t),\mu_{v}^{3}(t))$ and define the following events 
\begin{eqnarray*}
	E_{v}(t) &=& P_{v}(t)  \cap  \{ \mu_{v}^{1}(t) = \mu_{v}^{3}(t) \}\\
	J_{v}(t) &=& B_{v}(t) \cup  P_{v}(t) \cup \{ \beta_{v}(t)\in \{1/4,1 \}  \}.
\end{eqnarray*}
The \textit{adaptive 4-coupling} is defined according to the following equations 
\begin{equation}\label{eq:A4C_evolution}
\begin{cases}
\dot{\Sigma}_{v}(t)\equiv (1,1,0,0,0,0) & \text{if $J_{v}(t)$ does not occur}\\ 
\Sigma_{v}(t^{+}) = \mathcal{F}(\Sigma_{v}(t)) & \text{if $J_{v}(t)$ occurs},
\end{cases}
\end{equation}
where the `joint response curve' $\mathcal{F}:\Omega\rightarrow \Omega$ will be described below. Note that \eqref{eq:A4C_evolution} determines a unique \textit{joint trajectory} $(\Sigma_{\bullet}(t))_{t\ge 0}$ for each initial joint configuration $\Sigma_{\bullet}(0)$, where $\Sigma_{v}(t)-(t,t,0,0,0,0)$ is a left-continuous step function on $\Omega$.

The last component $\sigma_{v}(t)$ describes the \textit{state} of $v$ at time $t$, according to which $v$ adjusts its PRC. Namely, we say a node $v\in V$ is \textit{rested} at time $t$ if $\sigma_{v}(t)=0$, \textit{refractory} at time $t$ if $\sigma_{v}(t)\in \{1,2\}$, and \textit{excited} at time $t$ if $\sigma_{v}(t)=0$ and $\sigma_{v}(t^{+})=1$. This state variable jumps (in $\mathbb{Z}_{3}$) as 
\begin{equation*}
\sigma_{v}(t^{+}) - \sigma_{v}(t) = \mathbf{1}\big[ \left( \{ \sigma_{v}(t)=0\} \cap E_{v}(t)  \right) \cup \left( \{\sigma_{v}(t)=1\} \cap B_{v}(t)  \right) \cup \left( \{\sigma_{v}(t)=2\} \cap B_{v}(t) \right) \big].
\end{equation*}
On the other hand, the phase variable $\phi_{v}(t)$ jumps according to  
\begin{equation}
\phi_{v}(t^{+}) = \begin{cases}
\phi_{v}(t) & \text{if $R_{v}(t)$ does not occur} \\
F(\phi_{v}(t),\sigma_{v}(t)) & \text{if $R_{v}(t)$ occurs},
\end{cases}
\end{equation}
where $F:[0,1]\times \mathbb{Z}_{3} \rightarrow [0,1]$ is the following interpolation between the 4-coupling PRC $f_{0}$ and the refractory PRC $x\mapsto x$: 
\vspace{-0.1cm}
\begin{equation*}
F(x,s)= f_{0}(x)\cdot \mathbf{1}\{ s=0 \}   +   x\cdot \mathbf{1}\{s\ne 0 \}.
\end{equation*}

The other variables $\beta_{v}(t)$ and $\mu_{v}(t)$ record some local information during a moving window since the last blinking time, which is needed to determine the `excitation' event $E_{v}(t)$. Their verbal interpretations are as follows. 
\vspace{0.1cm}
\begin{description}
	\item{$\bullet$} $\beta_{v}(t)=$ time lapse modulo 1 since the last time that $v$ blinks
	\vspace{0.1cm}
	\item{$\bullet$} $\mu_{v}^{1}(t)=3-2\cdot \mathbf{1}(\text{$\sigma_{v}=2$ at the latest blinking time of $v$})$ 
	\vspace{0.1cm}
	\item{$\bullet$} $\mu_{v}^{2}(t) = \mathbf{1}(\text{$B_{v}(s)$ did not occur during $ [t-1/4,t)$})$
	\vspace{0.1cm}
	\item{$\bullet$} $\mu_{v}^{3}(t) = \#$ of occurrences of $P_{v}(s)\cap \{ \mu_{v}^{2}(s)=1 \}$ or $P_{v}(s)\cap \{ \mu_{v}^{2}(s)=0 \}\cap \{\beta_{v}(t)=1/4 \}$  (with upper cap of 3) since the last time that $\beta_{v}(t')=0$.
\end{description}
\vspace{0.1cm}
Hence, in words, the excitation event $E_{v}(t)$ is the event of being pulled when the accumulated pull count $\mu_{v}^{3}(t)$ is already matching the `dynamic threshold' $\mu_{v}^{1}(t)$. The pseudocode given in Algorithm 1 below describes the precise definition of the jumping rule $(\beta_{v}(t), \mu_{v}(t))\mapsto (\beta_{v}(t^{+}), \mu_{v}(t^{+}))$ upon $J_{v}(t)$. See Figure \ref{fig:ex_A4C} for an example of the adaptive 4-coupling dynamics.

\begin{algorithm}\label{algorithm:A4C}
	\caption{The adaptive 4-coupling}
	\begin{algorithmic}[1]
		\State \textbf{Variables:} 
		\State  \qquad $\phi_{v}(t)\in S^{1}:$ phase of node $v$
		\State \qquad $\beta_{v}(t)\in S^{1}:$ time lapse since the last blinking time of $v$ modulo $1$
		\State  \qquad $\mu_{v}(t)=(\mu_{v}^{1}(t),\mu_{v}^{2}(t),\mu^{3}_{v}(t))\in \{ 1,3 \}\times \mathbb{Z}_{2} \times \mathbb{Z}_{4}$: memory variable of node $v$
		\State \qquad $\sigma_{v}(t)\in \mathbb{Z}_{3}$: state variable of node $v$  
		\State \qquad $\Sigma_{v}(t)=(\phi_{v}(t),\beta_{v}(t),\mu_{v}^{1}(t),\mu_{v}^{2}(t),\mu_{v}^{3}(t),\sigma_{v}(t)):$ joint state of node $v$ 
		\State \textbf{Events:}
		\State \qquad $B_{v}(t)=\{ \phi_{v}(t^{-})=1 \}\cap \{\phi_{v}(t)=0\}$ 
		\State \qquad $P_{v}(t) =  \{ \text{$\exists u\in N(v)$ s.t. $B_{u}(t)$ occurs} \} \cap \{ 0< \phi_{v}(t) \le 1/2 \}$
		\State \qquad $J_{v}(t)=B_{v}(t)\cup P_{v}(t)\cup \{ \beta_{v}(t) \in \{ 1/4, 1 \} \}$
		\State \textbf{While $J_{v}(t)$ does not occur:}
		\State \qquad \textbf{Do} $\dot{\Sigma}_{v}(t) \equiv (1,1,0,0,0,0)$		
		\State \textbf{Upon $J_{v}(t)$:}
		\State \qquad \textbf{If $B_{v}(t)$ occurs:}
		\State \qquad \quad \textbf{Do} $\Sigma_{v}(t^{+})=(0,\, 0,\, 3-2\cdot \mathbf{1}(\sigma_{v}(t)=2),\, 0,\, 0,\, \sigma_{v}(t)+\mathbf{1}(\sigma_{v}(t)\ne 0))$
		\State \qquad \textbf{Else if $\{ \beta_{v}(t)=1 \}\setminus (B_{v}(t)\cup P_{v}(t))$ occurs:} 
		\State \qquad\quad \textbf{Do} $\Sigma_{v}(t^{+})=(\phi_{v}(t),\, 0,\, \mu_{v}^{1}(t),\, 1,\, 0,\, \sigma_{v}(t))$
		\State \qquad \textbf{Else if $\{ \beta_{v}(t)=1/4 \}\setminus (B_{v}(t)\cup P_{v}(t))$ occurs:} 
		\State \qquad \quad \textbf{Do} $\Sigma_{v}(t^{+})=(\phi_{v}(t),\, 1/4,\, \mu_{v}^{1}(t),\, 1,\, \mu_{v}^{3}(t),\, \sigma_{v}(t))$
		\State \qquad \textbf{Else ($P_{v}(t)$ occurs):}
		\State \qquad \quad \textbf{Do} $\phi_{v}(t^{+})=\phi_{v}(t)\cdot \mathbf{1}(\sigma_{v}(t)\ne 0) +   (\phi_{v}(t)-1/4)\cdot \mathbf{1}(1/4 < \phi_{v}(t) \le 1/2)\mathbf{1}(\sigma_{v}(t)= 0)$ 
		\State \qquad\quad \textbf{Do} $(\beta_{v}(t^{+}), \mu_{v}^{1}(t^{+}), \sigma_{v}(t^{+}))=(\beta_{v}(t),\, \mu_{v}^{1}(t),\,\sigma_{v}(t)+\mathbf{1}(\sigma_{v}(t)=0)\mathbf{1}(\mu_{v}^{1}(t)=\mu_{v}^{3}(t))  )$
		\State \qquad\quad \textbf{If} $\mu_{v}^{2}(t)=0$
		\State \qquad\quad \qquad \textbf{Do}  $\mu_{v}^{2}(t^{+})=\mu_{v}^{3}(t^{+})=\mathbf{1}(\beta_{v}(t)=1/4)$
		\State \qquad\quad \textbf{If} $\mu_{v}^{2}(t)=1$
		\State \qquad\quad \qquad \textbf{Do} $\mu_{v}^{2}(t^{+})=1$ and $\mu_{v}^{3}(t^{+})=[\mu_{v}^{3}(t)+\mathbf{1}(\mu_{v}^{3}(t)\ne 3)] \mathbf{1}(\beta_{v}(t)=1)$
	\end{algorithmic}
\end{algorithm}

\begin{figure*}[h]
	\centering
	\includegraphics[width = 1 \linewidth]{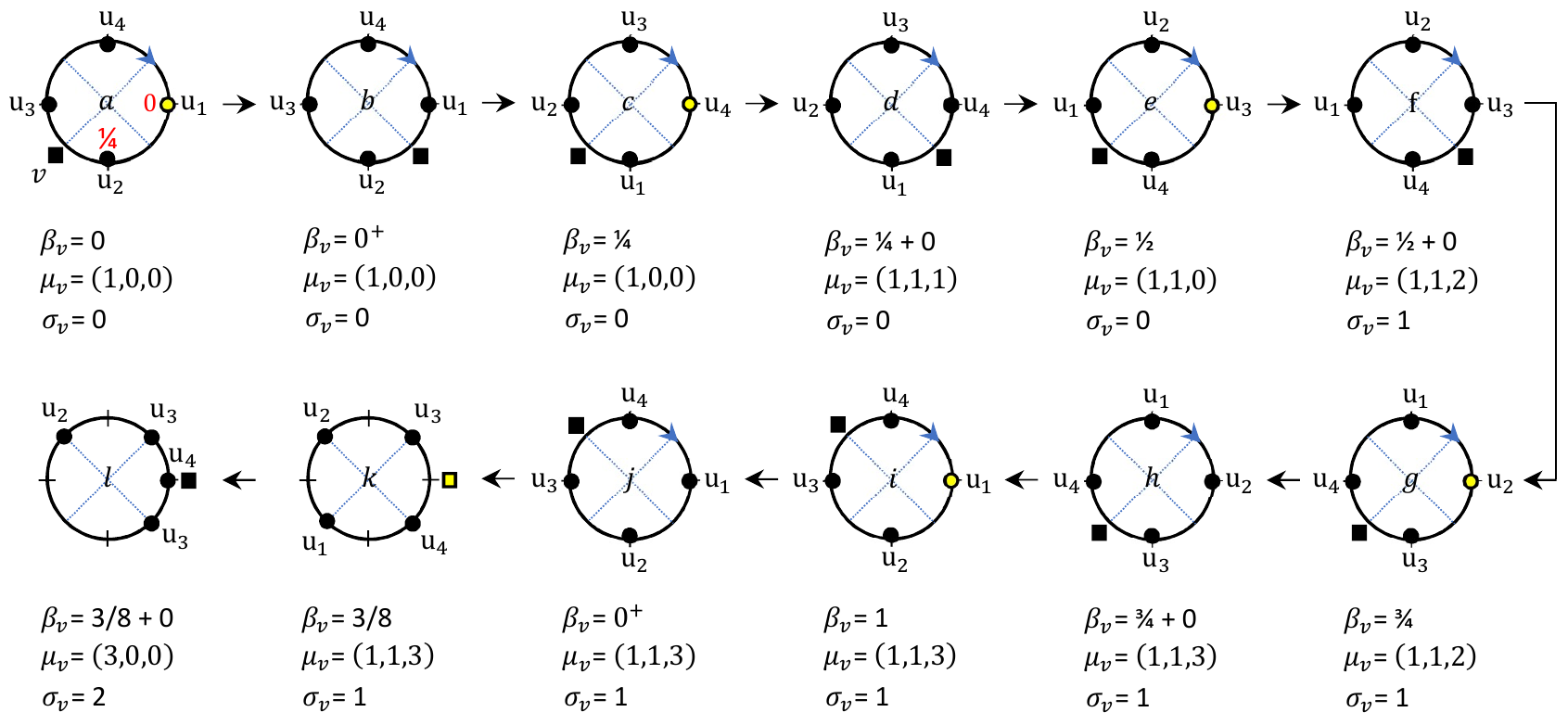}
	\vspace{-0.3cm}
	\caption{ An example of adaptive 4-coupling dynamics on a star with center $v=\blacksquare$ and leaves $u_{i}=\bullet$ for $i\in \{1,2,3,4\}$. The circle represents the phase space $S^{1}=\mathbb{R}/\mathbb{Z}$ with two points $0$ and $1/4$ are indicated in red. Nodes revolve on the circle in clockwise and blinking nodes are indicated as yellow. Initially $\sigma_{\bullet}(0)\equiv 0$ so $\sigma_{u_{i}}(t)\equiv 0$ for all $t\ge 0$; other three variables for $v$ are indicated below $a$. After following the same phase dynamics as in Figure \ref{fig:ex_4C} through $a\rightarrow e$, $v$ gets excited in $f$. Hence it ignores the pulse from $u_{2}$ as in $g\rightarrow h$. Then $v$ blinks in $k$, updates to state $\sigma_{v}=2$, refreshes $\mu_{v}$ to $(3,0,0)$, and pulls the two leaves $u_{1}$ and $u_{4}$. 
	}
	\label{fig:ex_A4C}
\end{figure*}   

We remark that the auxiliary dynamics of the adaptive 4-coupling is chosen carefully to satisfy some desired properties. First, the adaptive 4-coupling reduces to the 4-coupling if the underlying graph $G$ has maximum degree $\le 3$ and the auxiliary variables are initialized properly (Proposition \ref{prop:statebasic} (iii)). For that, we need to make sure that $E_{v}(t)$ never occurs under such condition. Second, $E_{v}(t)$ is invoked often enough so that large degree cases reduce to degree $\le 3$ case in a constant time (Propositions \ref{prop:3branch_gen_to_right_half}). Third, once $E_{v}(t)$ is invoked when rested $\sigma_{v}(t)=0$, the dynamic threshold $\mu_{v}^{1}(t)\in \{ 1,3\}$ is lowered to $1$ so that it is easier for $E_{v}(t)$ to occur once more, which breaks all local symmetries in constant time (Proposition \ref{prop:branch_excitation}). Fourth, $E_{v}(t)$ it is invoked not so often that a node of large degree behaves similarly as in the 4-coupling dynamics every once in a while (Proposition \ref{prop:E_v(t)_needs_enough_pull}). This property is used crucially in the proof of Lemma \ref{lemma:branchorbit_b}. Lastly, with all of the above mentioned properties, the event $E_{v}(t)$ is defined by only using local pulse communication in order to maximize the advantage of pulse-couplings such as scalability, robustness, and efficiency.

Our main result in this paper is stated in the following theorem:

\begin{theorem}\label{A4Ctreethm}
	Let $T=(V,E)$ be a finite tree with diameter $d$ and maximum degree $\Delta$. Then arbitrary initial  joint configuration on $T$ synchronizes under the adaptive $4$-coupling by time $C_{\Delta}d$, where $C_{\Delta}=51+32\cdot \mathbf{1}(\Delta\ge 4)$.
\end{theorem}

\vspace{0.3cm}
\section{Related works}
\label{Section:Related works}

Three types of questions are usually considered in the study of PCOs as well as clock synchronization algorithms: 
\vspace{0.1cm}
\begin{description}[noitemsep]
	\item{Q1.} Given a coupling, on what network topologies is synchrony guaranteed to emerge from all (or almost all) initial configurations?
	\vspace{0.1cm} 
	\item{Q2.} Given a coupling, what conditions on initial configurations can guarantee the emergence of synchrony on arbitrary network topology? 
	\vspace{0.1cm}
	\item{Q3.} What couplings can guarantee the emergence of synchrony on arbitrary network topology and arbitrary initial configuration?
\end{description}
\vspace{0.1cm}
Traditionally PCO literature focuses on Q1, whereas in distributed algorithms literature Q3 is usually asked with respect to some desired properties and performance. Q2 is often considered in the control theory literature, where PCO studies are applied to handle more realistic issues such as signal loss or propagation delay. 

For Q2, there is a fundamental observation which has been made under different settings in the literature. Namely, for arbitrary connected underlying graph, synchronization is guaranteed if all initial phases reside in a half of the oscillation cycle $S^{1}$. The key idea is that in spite of the cyclic hierarchy in the phase space, such concentration gives a linear ordering on the oscillator phases (e.g., from the most lagging to the most advanced). Pulse-couplings under mild conditions respect such monotonicity and contracts the phase configuration toward synchrony. (See Lemma \ref{widthlemma} for our version.) It is known that variants of this lemma hold against propagation delays, faulty nodes, and temporally changing topologies \cite{nishimura2011robust, klinglmayr2012guaranteeing, proskurnikov2015synchronization, nunez2015synchronization}, and are commonly used in multi-agent consensus problems \cite{moreau2005stability, papachristodoulou2010effects, chazelle2011total}.

While the methods based on concentration condition could be applied once the system is nearly synchronized or to maintain synchrony against weak fluctuation, a fundamental question must be addressed: \textit{how can we drive the system close to synchrony from an arbitrary initial configuration in the first place?} This is what Q1 focuses on, which has been answered for some classes of pulse-couplings mainly on complete (all-to-all) graphs or cycles. In their seminal work, Mirollo and Strogatz \cite{mirollo1990synchronization} showed that an excitatory pulse-coupling on complete graphs synchronizes almost all initial configurations. A similar result was derived for an inhibitory pulse-couplings by Klingmayr and Bettstetter \cite{klinglmayr2012self}. For PCOs on cycle graphs, Wang, and Doyle \cite{nunez2015global} addressed Q1. More recently, these authors and Teel \cite{nunez2016synchronization} studied Q1 for PCOs on general topology assuming a global pacemaker. In Theorem \ref{4treethm} and \ref{A4Ctreethm}, we give an answer to Q1 in the case of the 4-coupling and adaptive 4-coupling on tree networks. 

The question Q3 is closely related to the concept of self-stabilization in theoretical computer science. A distributed algorithm is said to be \textit{self-stabilizing} if it recovers desired system configurations from arbitrary system configuration. This notion was first proposed by Dijkstra \cite{dijkstra1982self} as a paradigm for designing distributed algorithms which are robust under arbitrary transient faults. For the convenience in following discussions, we denote by $d$ and $\Delta$ the diameter and maximum degree of the underlying network, respectively. 

A popular approach in designing a clock synchronization algorithm solving Q3 is to use un unbounded memory for each node to `unravel' the cyclic phase space $S^{1}$ and construct an ever-increasing clock counter. As in the above mentioned technique assuming concentration condition, this effectively gives a global total ordering between local times. Then all nodes can tune toward the locally maximal time, for instance, so that the global maximum propagates and subsumes all the other nodes in $O(d)$ time. This idea dates back to Lamport \cite{lamport1978time}, and similar technique has been used in different contexts: for synchronous systems \cite{gouda1990stabilizing} and for asynchronous systems \cite{awerbuch2007time}. The biggest advantage of such approach includes independence of network topology and optimal time complexity of $O(d)$. However, they suffer when it comes to memory efficiency, and assuming unbounded memory on each node is far from practical, especially with the presence of faulty nodes \cite{ghosh2014distributed}.

Considerable amount of works were devoted to solve Q3 with bounded memory per node. Algorithms using $O(\log d)$ and $O(\log |V|)$ memory per node for synchronous \cite{arora1992maintaining} and asynchronous \cite{couvreur1992asynchronous} systems, respectively, are known. Solutions with bounded memory per node assuming some a priori knowledge about global quantities have been studied as well (e.g., \cite{herman2000phase, boulinier2004graph}). However, when \textit{scalability} becomes an important issue (e.g., for wireless sensor networks), it is highly desirable that an algorithm uses only $O(1)$ memory per node without any assumption on the network. Due to the problem of symmetry breaking, however, no such algorithm solving Q3 exists if nodes are indistinguishable (e.g., on rings \cite{dolev2000self}). One way out is to incorporate internal randomization until a desirable global configuration (e.g., concentrated ones) is reached eventually, from which synchrony is guaranteed (e.g., Theorem 5.1 of Lyu \cite{lyu2015synchronization} or Klinglmayr et al. \cite{klinglmayr2012guaranteeing}). However, such algorithms inevitably yield exponential running time, which makes them not ideal for applications. Our Corollary \ref{thm:clocksync} gives a randomized universal solution to Q3 with asymptotically linear worst case expected running time, using only $O(\log \Delta)$ memory per node.   

For synchronous systems with discrete phase clocks taking values from $\mathbb{Z}_{\kappa}$, a number of algorithms which are self-stabilizing on trees with constant memory per node are known: e.g., for $\kappa=3$ by Herman and Ghosh \cite{herman1995stabilizing}, for all odd $\kappa \ge 3$ by Boulinier et al. \cite{boulinier2006toward}. Upper bounds of $O(d)$ for convergence time is known for such algorithms. More recently, the author proposed a class of $\kappa$-state  inhibitory pulse-couplings which are called the \textit{firefly cellular automata} (FCAs) \cite{lyu2015synchronization}. In the reference and in \cite{lyu2016phase}, we showed that the $\kappa$-color FCA is self-stabilizing on finite paths for arbitrary $\kappa\ge 3$, and on finite trees if and only if $\Delta<\kappa\le 6$. The 4-coupling we introduced in this work is a continuous-state and asynchronous-update generalization of the 4-color FCA. That is, if the initial phases for the 4-coupling are discretized on a 1/4 grid points on $S^{1}$, i.e., $\phi_{v}(0)\equiv \phi_{u}(0) \mod 1/4$ for all $u,v\in V$, then the trajectory $(4\phi_{\bullet}(t)+1)_{t\ge 0}$ follows the 4-color FCA dynamics.

\vspace{0.3cm}

\section{Relative circular representation}
\label{Section:Relative circular representation}

In this section, we introduce an alternative representation of phase dynamics of PCOs. Let $G=(V,E)$ be a finite simple graph and consider a phase dynamics $(\phi_{\bullet}(t))_{t\ge 0}$ evolving via a given pulse-coupling.  Since the oscillators evolve on the unit circle $S^{1}$, it is natural to superpose all phases on a single unit circle, so that $|V|$ dots on $S^{1}$ evolve with unit speed adjusting their phase upon receiving pulses from their neighbors. We may choose clockwise orientation for oscillation, since we think of each oscillator as a local clock in network $G$. This way of representing phase dynamics is called the \textit{circular representation}, which we have already used in the examples given in Figure \ref{fig:ex_4C} and \ref{fig:ex_A4C}.

An alternative perspective, which is particularly useful in describing PCO dynamics, is to look at evolution of `relative phases' with respect to an imaginary reference oscillator which revolves the unit circle $S^{1}$ clockwise in unit speed independently. More precisely, fix $\alpha(0)\in S^{1}$ and for each $v\in V$ and $t\ge 0$, define the \textit{relative phase} $\Lambda_{v}(t)$ of $v$ at time $t$ by 
\begin{equation}
\Lambda_{v}(t) = \phi_{v}(t) + \alpha(0)- t \mod 1 \quad \text{$\forall v\in V$ and $t\ge 0$}.
\end{equation} 
Each map $\Lambda_{\bullet}(t):V\rightarrow S^{1}$ is called \textit{relative phase configuration} at time $t$. Then the given phase dynamics $(\phi_{\bullet}(t))_{t\ge 0}$ and the choice of $\alpha(0)\in S^{1}$ induces a unique trajectory $(\Lambda_{\bullet}(t))_{t\ge 0}$ of relative configurations. 

To visualize the relative phase dynamics, we introduce an imaginary node $\alpha$ called the \textit{activator}, which has constant phase $0$ for all times. Then its relative phase at time $t$ is simply 
\begin{equation*}
\alpha(t):=\Lambda_{\alpha}(t)=\alpha(0)-t \mod 1 \quad \text{$\forall t\ge 0$}.
\end{equation*} 
Note that a node $v$ blinks at time $t$ iff $\Lambda_{v}(t^{-})\equiv \alpha(t)\mod 1$. In other words, in the relative phase dynamics, the activator revolves the unit circle counterclockwise with unit speed and each oscillator blinks whenever it is caught up by the activator. Hence oscillators keep their relative phase constant, adjusting their relative phases upon receiving pulses. We call this representation of PCO dynamics the \textit{relative circular representation}. (See Figure \ref{fig:ex0} for example, in the case of the 4-coupling.)

\begin{figure*}[h]
	\centering
	\includegraphics[width = 0.9 \linewidth]{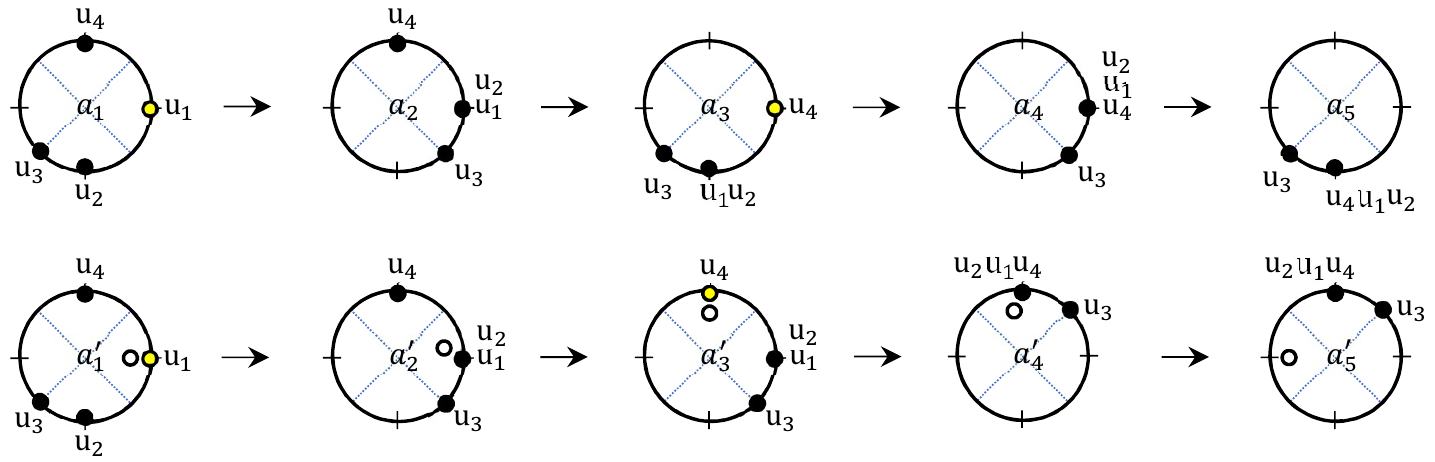}
	\caption{ An example of 4-coupled oscillator dynamics on complete graph $K_{4}$. In  circular representation ($a_{1}\rightarrow a_{5}$), oscillators evolve clockwise with unit speed. During $a_{1}\rightarrow a_{2}$, oscillator $u_{1}$ blinks pulls $u_{2}$ and $u_{3}$. Then $u_{4}$ blinks in $a_{3}$ and pulls all the other nodes by 1/4 in phase. During the following 1/4 second, all oscillators evolve with unit speed and end in $a_{5}$. In relative circular representation ($a_{1}\rightarrow a_{5}$), the activator (shown in empty circle) revolves counterclockwise with constant unit speed. Oscillators blink whenever the activator is on them. For instance, $u_{1}$ blinks in $a_{1}'$ since the activator is on it, and the activator is slightly counterclockwise to $u_{1}$ and $u_{2}$ in $a_{2}$ to indicate an infinitesimal time lapse. 
	}
	\label{fig:ex0}
\end{figure*}

Next, we describe the 4-coupling and adaptive 4-coupling dynamics in terms of relative circular representation. The \textit{counterclockwise displacement} is a function $\delta:S^{1}\times S^{1}\rightarrow [0,1)$ such that 
\begin{equation*}
\delta(x,y) = x-y \mod 1,
\end{equation*}
which is the length of the counterclockwise arc from $x$ to $y$ on $S^{1}$. Given a time evolution of relative phase configurations $(\Lambda_{\bullet}(t))_{t\ge 0}$, we abbreviate $\delta_{t}(u,v):=\delta(\Lambda_{u}(t),\Lambda_{v}(t))$. It is straightforward to verify that the 4-coupling evolves relative phase configurations as follows: 

\begin{equation}\label{eq:relative_evolution_4C}
\begin{cases}
\quad\, \alpha(t) = \alpha(0) - t \\
\Lambda_{v}(t^{+}) = \Lambda_{v}(t)-1/4  
& \text{if $\exists u\in N(v)$ s.t. $\Lambda_{u}(t)=\alpha(t)$ and $1/4\le \delta_{t}(v,u)\le 1/2$} \\   
\Lambda_{v}(t^{+}) = \Lambda_{u}(t)& \text{if $\exists u\in N(v)$ s.t. $\Lambda_{u}(t)=\alpha(t)$ and $0<\delta_{t}(v,u)< 1/4$} \\ 
\Lambda_{v}(t^{+}) = \Lambda_{v}(t) & \text{otherwise},
\end{cases}
\end{equation}	
where all subtractions are taken modulo 1.

For a verbal description, we introduce some terminologies. We say $u$ is \textit{left} to $v$ at time $t$ if $\delta_{t}(u,v)\in (0,1/2)$ and \textit{opposite} to $v$ at time $t$ if $\delta_{t}(u,v)=\delta_{t}(v,u)=1/2$. We say $u$ is \textit{right} to $v$ at time $t$ if $v$ is left to $v$ at time $t$. Note that $u$ is left (resp., right) to $v$ at time $t$ iff it is ahead (resp., behind) of $v$ in the oscillation cycle at time $t$. For instance, in Figure \ref{fig:ex0} $a_{1}'$, $u_{2}$ and $u_{3}$ are left to $u_{1}$, and $u_{2}$ is opposite to $u_{4}$.

Now, (\ref{eq:relative_evolution_4C}) says that each vertex $v$ gets pulled whenever there is a right or opposite blinking neighbor $u$; it shifts its relative phase by $-1/4$ (counterclockwise by $1/4$) if $\delta_{t}(v,u)\in (1/4,1/2]$, and merges with $\Lambda_{u}(t)$ if $\delta_{t}(v,u)\in (0,1/4]$ (see Figure \ref{fig:ex0} $a_{1}'\rightarrow a_{2}'$). The latter `merging' feature is one of the distinct characteristics of the (adaptive) 4-coupling; it quickly reduces continuum number of possibilities into a few highly symmetric ones, enabling a combinatorial analysis of this continuous model. Also note that we have a freedom of choosing $\alpha(0)\in S^{1}$, indicating that the relative circular representation is symmetric under rotation. We will use this rotational symmetry to make a fixed oscillator to have a particular relative phase at a given time for our convenience.

\vspace{0.3cm}

\section{Preliminaries, the branch width lemma, and reduction by excitation}
\label{Section:The width lemma and the branch width lemma}

In this section, we establish some basic properties of the adaptive 4-coupling. Convergence to synchrony from concentrated phase configurations is shown in Lemma \ref{widthlemma}. We also establish its localized variant, Lemma \ref{lemma:branchwidth}, which will play a central role throughout this paper. Moreover, we also show that the occurrence of $E_{v}(t)$ effectively reduces local structures to consider in inductive argument (Proposition \ref{prop:branch_excitation}).

We first observe that in the adaptive 4-coupling dynamics, every node blinks at most once every second and at least once in every five seconds. For each $v\in V$ and any finite interval $I\subset [0,\infty)$, we introduce the \textit{total phase inhibition of $v$ during $I$}, which is defined by
\begin{equation}\label{eq:total_phase_inhibition}
f_{v}(I)=\sum_{\substack{ s\in I\\ \text{$P_{v}(s)$ occurs}   }} \phi_{v}(s^{+})-\phi_{v}(s),
\end{equation} 
where the summation is taken in $\mathbb{R}$ and not modulo 1. It is well-defined by the first part of Proposition \ref{prop:blinking} and $\deg(v)<\infty$. Also note that the summand takes value from $[-1/4,0]$ for all $s\ge 0$. 

\begin{proposition}\label{prop:blinking}
	Let $G=(V,E)$, $(\Sigma_{\bullet}(t))_{t\ge 0}$, and $(\Lambda_{\bullet}(t))_{t\ge 0}$ be as before. For any $v\in V$ and $t_{0}\ge 0$, $B_{v}(t)$ occurs at most once during $(t_{0},t_{0}+1]$ and at least once during $(t_{0},t_{0}+5]$.   
\end{proposition}

\begin{proof}
	The first part follows since the adaptive 4-coupling is inhibitory. For the second part, suppose for contrary that $B_{v}(t)$ never occurs during some interval $(t_{0},t_{0}+5]$. Then necessarily $\sigma_{v}(t)\equiv 0$ and $E_{v}(t)$ never occurs during $(t_{0},t_{0}+4]$. Let $t_{1}=\inf\{ t_{0}< s \le t_{0}+1 \,:\, \beta_{v}(s^{-})=1 \}$. Then $\mu_{v}^{2}(t)\equiv 1$ during $(t_{1},t_{0}+4]\subseteq (t_{0}+1,t_{0}+4)$. Now we claim that $f_{v}((t_{1}+k,t_{1}+k+1])\ge -3/4$ for $k\in \{0,1,2\}$. Suppose for contrary that $f_{v}((t_{1}+k,t_{1}+k+1])< -3/4$ for some $k\in \{0,1,2\}$. Then $P_{v}(t)$ must occur at least four times during $(t_{1}+k, t_{1}+k+1]$. But then when the fourth $P_{v}(t)$ occurs we have $\mu_{v}^{3}(t)=3\ge \mu_{v}^{1}(t)$; hence $E_{v}(t)$ must occur during $(t_{1}+k, t_{1}+k+1]$, contrary to our assumption. This shows the claim. 
	
	Now by the claim $\phi_{v}(t_{1}^{+}+1)=\phi_{v}(t_{1}^{+})+1+f_{v}( (t_{1},t_{1}+1] )\ge 1/4$ and similarly $\phi_{v}(t_{1}^{+}+2)\ge 1/2$. So $\phi_{v}(t)$ increases to 1 with rate $1$ after time $t_{1}+2$, so $v$ blinks at time some time $\le t_{1}+2+1/2 \le t_{0}+4$, a contradiction.  
\end{proof}

Next, we infer some information about $f_{v}(I)$ from the state evolution of $v$. It will be used crucially in the proof of Lemma \ref{lemma:branchorbit_b} in Section \ref{section:locallemmas}.

\begin{proposition}\label{prop:E_v(t)_needs_enough_pull}
	Let $G=(V,E)$, $(\Sigma_{\bullet}(t))_{t\ge 0}$, and $(\Lambda_{\bullet}(t))_{t\ge 0}$ be as before. Let $v\in V$ and suppose $B_{v}(t_{0})$ occurs and $\sigma_{v}(t_{0}^{+})=0$ for some $t_{0}\ge 0$. Let $t_{i}:=\inf\{ t\ge t_{0}\,:\, \sigma_{v}(t)=i \}$ for each $i\in \{1,2\}$. 
	\vspace{0.1cm}
	\begin{description}
		\item[(i)]  If $B_{v}(t)$ does not occur during $(t_{0},t_{1}]$, then $f_{v}((t_{0},t_{1}])<-1/4$. 
		\vspace{0.1cm}
		\item[(ii)]  If $t_{2}\le t_{0}+3/2$, then $f_{v}((t_{0},t_{1}])\in [-1/2,-1/4)$. 
		\vspace{0.1cm} 		
	\end{description} 
\end{proposition}

\begin{proof}
	To show (i), note that since $B_{v}(t_{0})$ occurs, $\mu_{v}(t_{0}^{+})=(1+2\cdot \mathbf{1}[\sigma_{v}(t_{0})=0], 0, 0)$ and $\mu_{v}^{2}(t)\equiv 0$ during $(t_{0},t_{0}+1/4]$. Hence in order for $E_{v}(t_{1})$ to occur, $P_{v}(t)$ must occur at least twice during $(t_{0}+1/4,t_{1}]$. This yields $f_{v}((t_{0},t_{1}])< -1/4$ as desired.  
	
	The hypothesis of (ii) implies $t_{1}\le t_{0}+1/2$. This yields 
	$f_{v}((t_{0},t_{1}]) \ge -1/2$. On the other hand, $B_{v}(t)$ never occurs during $(t_{0},t_{0}+1/2]$ since it did at time $t_{0}$. Hence by (i), we have $f_{v}((t_{0},t_{1}]) < -1/4$. This shows (ii).
\end{proof}

The next proposition shows that the adaptive 4-coupling is essentially the 4-coupling when a node has small degree. 

\begin{proposition}\label{prop:statebasic}
	Let $G=(V,E)$, $(\Sigma_{\bullet}(t))_{t\ge 0}$, and $(\Lambda_{\bullet}(t))_{t\ge 0}$ be as before. Let $v\in V$. 
	\vspace{0.2cm}
	\begin{description}[noitemsep]	
		\item[(i)] If $\deg(v)=1$, then $\sigma_{v}(t)\equiv 0$ for all $t> 2$. 	
		\vspace{0.1cm}
		\item[(ii)] Suppose $\deg(v)\le 3$. If either $(\mu_{v}^{1}(t_{0}), \mu_{v}^{3}(t_{0}), \sigma_{v}(t_{0}))$ or $(\mu_{v}^{1}(t_{0}^{+}), \mu_{v}^{3}(t_{0}^{+}), \sigma_{v}(t_{0}^{+}))$ is $(3,0,0)$ for some $t_{0}\ge 0$, then $\sigma_{v}(t)\equiv 0$ for all $t> t_{0}$.
		
		\item[(iii)] If $G$ has maximum degree $\le 3$ and $(\mu_{\bullet}^{1}(0), \mu_{\bullet}^{3}(0), \sigma_{\bullet}(0))\equiv (3,0,0)$, then $(\phi_{\bullet}(t))_{t\ge 0}$ follows the 4-coupling. 
	\end{description}
\end{proposition}

\begin{proof}
	\begin{description}
		\item{(i)} Let $t_{i}$ be the $i^{\text{th}}$ time that  $B_{v}(t)\cup \{ \beta_{v}(t^{-})=1 \}$ occurs after time 0. Note that for $i\ge 1$, we have $t_{i+1}\le t_{i}+1$ and $\mu_{v}(t_{i}^{+})=(y_{i},0,0)$ where $y_{i}\in \{1,3\}$. Since the unique neighbor of $v$ blinks at most once in every second, $P_{v}(t)$ occurs at most once during every interval $(t_{i},t_{i+1}]$. Hence $E_{v}(t)$ does not occur during each $(t_{i},t_{i+1}]$. Note that since $B_{v}(t)$ occurs once in every 5 seconds (Proposition \ref{prop:blinking}) and $t_{i+1}=t_{i}+1$ if $B_{v}(t_{i})$ occurs, we know that $t_{n}\rightarrow \infty$ as $n\rightarrow \infty$. Thus $E_{v}(t)$ never occurs for all $t\ge t_{1}\ge 1$. Now $\sigma_{v}(s^{+})=0$ for some $s\le 2$, so the assertion follows.   
		\vspace{0.1cm}
		\item{(ii)} First note that $E_{v}(t_{0})$ does not occur and $\sigma_{v}(t_{0})=0$; It is clear if $(\mu_{v}^{1}(t_{0}), \mu_{v}^{3}(t_{0}), \sigma_{v}(t_{0}))$ equals to $(3,0,0)$. Otherwise, note that $\sigma_{v}(t_{0})\in \{0,2\}$ since $\sigma_{v}(t_{0}^{+})=0$. Since $B_{v}(t_{0})$ occurs and $\mu_{v}^{1}(t_{0}^{+})=2$ if  $(\sigma_{v}(t_{0}),\sigma_{v}(t_{0}^{+}))=(2,0)$, we should have $\sigma_{v}(t_{0})=0$. Then $E_{v}(t_{0})$ cannot occur since otherwise we would have $\sigma_{v}(t_{0}^{+})=1$. 
		
		\quad Now it suffices to show that $E_{v}(t)$ never occurs for all $t>t_{0}$. To this end, for each $i\ge 1$, let $t_{i}$ be the $i^{\text{th}}$ time that $B_{v}(t)\cup \{ \beta_{v}(t^{-})=1 \}$ occurs after time $t_{0}$. As before $t_{i+1}\le t_{i}+1$ for all $i\ge 0$. Suppose for contrary that $E_{v}(s_{1})$ occurs for some $s_{1}\in (t_{0},t_{1}]$. In both cases $\mu_{v}^{1}(t_{0}^{+})=3$ and $B_{v}(t)$ never occurs during $(t_{0},t_{1}]$ by definition, so $\mu_{v}^{1}(t)\equiv 3$ during $(t_{0},t_{1}]$. Hence in order for $E_{v}(s_{1})$ to occur, $P_{v}(t)$ must occur at least four times during $(t_{0},t_{1}]$. But since $\deg(v)\le 3$, $P_{v}(t)$ can occur at most three times during any interval $(s,s+1]$, a contradiction. Thus $E_{v}(t)$ never occurs during $(t_{0},t_{1}]$. This yields $\sigma_{v}(t)\equiv 0$ during the same interval and $\sigma_{v}(t_{1}^{+})=0$, so $\mu_{v}^{1}(t_{1}^{+})=3$. By definition of $t_{1}$, $\mu_{v}^{3}(t_{1}^{+})=0$. Thus the hypothesis is satisfied at time $t_{1}^{+}$. By induction, the same conclusion holds for all successive intervals $(t_{j},t_{j+1}]$. Note that $t_{n}\rightarrow \infty$ as $n\rightarrow \infty$ as in the proof of (i). Thus the assertion follows. 	
		\vspace{0.1cm} 
		\item{(iii)}  Follows immediately from (ii).
	\end{description} 
\end{proof}

Fix a finite simple graph $G=(V,E)$. A node $v$ in $G$ is a \textit{leaf} if $\deg(v)=1$. A connected induced subgraph $B\subseteq G$ is called a \textit{branch} if there exists a single node $v\in V(B)$ such that $B-v$ is a disjoint union of leaves in $G$ and $v$ has a single neighbor, say $w$, in $G-B$. We call $v$ and $w$ the \textit{center} and \textit{root} of branch $B$, respectively. We call $B$ a \textit{$k$-branch} in $G$ if it is a branch with $k\ge 1$ leaves. See Figure \ref{fig:branch_fig} for illustration.

\begin{figure*}[h]
	\centering
	\includegraphics[width = 0.3 \linewidth]{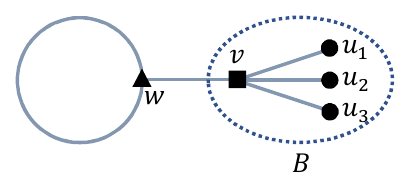}
	\vspace{-0.2cm}
	\caption{ A graph $G$ with a 3-branch $B$ with root $w$, center $v$, and leaves $u_{1},u_{2},$ and $u_{3}$.
	}
	\label{fig:branch_fig}
\end{figure*}

Below is a basic observation about the behavior of leaves in branches under the adaptive 4-coupling. 

\begin{proposition}\label{prop:branch_basic}
	Let $G=(V,E)$, $(\Sigma_{\bullet}(t))_{t\ge 0}$, and $(\Lambda_{\bullet}(t))_{t\ge 0}$ be as before. Let $B\subseteq G$ be a $k$-branch with center $v$ and root $w$.
	\vspace{0.2cm}
	\begin{description}[noitemsep]
		
		\item[(i)] Suppose $\sigma_{v}(t_{0})=0$, $\sigma_{v}(t_{0}^{+})=1$ for some $t_{0}\ge 6$, and let $t_{1}=\inf\{ t\ge t_{0}+1\,:\, \sigma_{v}(t)=0 \}$ and $\Lambda_{v}(t_{0}^{+})=1/2$. Then $t_{1}\le t_{0}+2$ and $\Lambda_{u}(t_{1}^{+})\in (0,1/2]$ for each leaf $u$ in $B$.
		\vspace{0.1cm}
		
		\item[(ii)] Suppose $v$ blinks at time $t_{0}\ge 8$ and $\sigma_{v}(t)\equiv 0$ for all $t\in [6,t_{0}]$. If we let $\Lambda_{v}(t_{0})=1/2$, then $\Lambda_{u}(t_{0}^{+})\in [1/4,3/4]$ for each leaf $u$ in $B$.   
	\end{description}
\end{proposition}

\begin{proof}
	To begin, we note that by Proposition \ref{prop:statebasic} (i), all leaves in $B$ have state $0$ for all times $t\ge 6$. 
	\begin{description}		
		\item{(i)} Let $\tau_{v;i}$ be the $i^{\text{th}}$ time that $v$ blinks after time $t_{0}$. Then $\tau_{v;1}\le t_{0}+1$, $\tau_{v;2}=\tau_{v;1}+1$, and $\Lambda_{v}(t)\equiv 1/2$ for all $t\in (t_{0},\tau_{v;2}]$. Hence if $u\in V$ is any leaf in $B$ with $\Lambda_{u}(t_{0})\in [1/2, 1]$, then it is pulled by $v$ twice at times $\tau_{v;1}$ and $\tau_{v;2}$, and since $\Lambda_{v}(\tau_{v;1})=\Lambda_{v}(\tau_{v;2})=1/2$ and since $\sigma_{u}(t)\equiv 0$ for all $t\ge t_{0}$, we have $\Lambda_{u}(\tau_{v;2}^{+})=1/2$. Moreover, if $u'$ is another leaf in $B$ such that $\Lambda_{u'}(t_{0})\in [0,1/2)$, then $\Lambda_{u'}(t)$ is constant during $[t_{0},\tau_{v;2}]$. Hence all leaves in $B$ have relative phase $\in (0,1/2]$ at time $\tau_{v;2}^{+}$. So the assertion holds for $t_{1}=\tau_{v;2}\le t_{0}+2$. 
		
		\vspace{0.1cm}
		\item{(ii)} Since $v$ blinks at time $t_{0}$, all leaves in $B$ have relative phase $\in (0,3/4]$ at time $t_{0}^{+}$. On the other hand, since $t_{0}\ge 8$, $P_{v}(t)$ must occur during $[6,t_{0}]$. Let $\tau=\sup\{ 6\le s\le t_{0}\,:\, \text{$P_{v}(s)$ occurs} \}$ be the last time that $v$ is pulled before time $t_{0}$. Then $\Lambda_{v}(t)$ is constant during $(\tau,t_{0}]$ so $\Lambda_{v}(\tau^{+})=\Lambda_{v}(t_{0})=1/2$. This yields $\alpha(\tau^{+})\in [1/4, 1/2]$. Since $P_{v}(t)$ does not occur during $[\tau,t_{0}]$, all neighbors of $v$ have relative phase $\in [1/4,1)$ at time $\tau^{+}$. Since the relative phases of all leaves in $B$ are constant during $(\tau,t_{0}]$, all leaves in $B$ have relative phase $\in [1/4,1)$ at time $t_{0}^{+}$. Combining these to observations then gives (ii).
	\end{description}
\end{proof}

Next, let $G=(V,E)$ be a finite simple graph and let $\phi_{\bullet}(t):V\rightarrow S^{1}$ be any phase configuration at time $t$. We define its \textit{width}, which we denote by $\omega(\phi_{\bullet}(t))\in [0,1)$, by 
\begin{equation}
\omega(\phi_{\bullet}(t))= \max_{u,v\in V} \min \{ \delta_{t}(u,v), \delta_{t}(v,u)\}.
\end{equation}  
In words, it is the length of the shortest arc in $S^{1}$ which contains all phases at time $t$. For a branch $B\subseteq G$, we denote $\omega_{B}(t):=\omega(\phi_{B}(t))$, where $\phi_{B}(t) : V(B)\rightarrow S^{1}$ is the restriction of the phase configuration $\phi_{\bullet}(t)$ on $B$. We call $\omega_{B}(t)$ the \textit{branch width} on $B$ at time $t$.

Note that $\phi_{\bullet}(t)$ synchronizes iff the width of $\phi_{\bullet}(t)$ converges to 0. Hence the width can be thought as a global order parameter which measures how close the current phase configuration is to synchrony. The following lemma, Lemma \ref{widthlemma}, implies that $w(\phi_{\bullet}(t))\in [0,1/2)$ for some $t\ge 0$ implies synchronization by time $t+7d$. As mentioned in Section \ref{Section:Related works}, that this half-width concentration implies convergence (without bound on the convergence time) has been observed for broader classes of non-adaptive pulse-couplings, which includes the 4-coupling (e.g., \cite{nishimura2011robust}). An implication of this lemma is that synchrony is stable under small perturbations, as long as the width is still $<1/2$.

\begin{lemma}[width lemma]\label{widthlemma}
	Let $G=(V,E)$ be a finite simple connected graph with diameter $d$ and let $(\Sigma_{\bullet}(t))_{t\ge 0}$ be a joint trajectory on $G$ evolving through the adaptive 4-coupling.  
	\vspace{0.2cm}
	\begin{description}[noitemsep]
		\item[(i)] $\phi_{\bullet}(t)$ synchronizes if and only if $\omega(\phi_{\bullet}(\tau))\in [0,1/2)$ for some $\tau\ge 0$.
		\vspace{0.1cm}
		\item[(ii)] If $\omega(\phi_{\bullet}(\tau))\in [0,1/2)$, then $\omega(\phi_{\bullet}(\tau+7d))=0$. 
	\end{description}
\end{lemma}

\begin{proof} 
	The ``only if'' part of (i) is obvious. (ii) implies the ``if'' part of (i). In order to show (ii), let $v_{0}\in V$ be a vertex such that $\delta_{\tau}(u,v_{0})=\omega(\phi_{\bullet}(\tau))\in [0,1/2)$ for some $u\in V$. In other words, $v_{0}$ is one of the ``most lagging'' oscillator at time $\tau$. Since the adaptive 4-coupling is inhibitory, this implies that $v_{0}$ never get pulled after time $\tau$ and all the other nodes are inhibited toward $v_{0}$. 
	
	\begin{figure*}[h]
		\centering
		\includegraphics[width = 1 \linewidth]{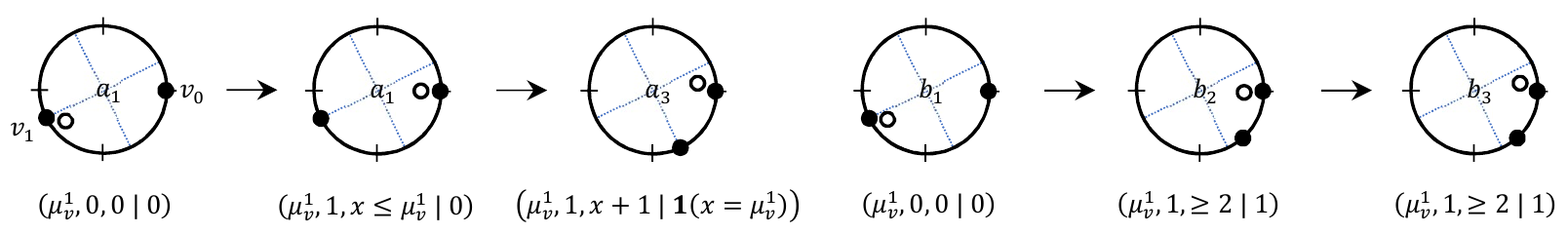}
		\vspace{-0.2cm}
		\caption{ (In relative circular representation) Dynamics on $G$ assuming $\omega(\phi_{\bullet}(\tau))\in [0,1/2)$. $\bullet=$relative phases of $v_{0}$ and $v_{1}$, and $\circ= $ activator. Below each $b_{i}$'s, $(a,b,c \,|\, d)$ denotes $\mu_{v}=(a,b,c)$ and $\sigma_{v}=d$. 
		}
		\label{fig:fig_widthlemma}
	\end{figure*}

	More precisely, let $(\Lambda_{\bullet}(t))_{t\ge 0}$ be the induced relative phase trajectory, where we take $\Lambda_{v_{0}}(\tau)=0$. Decompose $V=V_{0}\sqcup V_{1}\sqcup \cdots \sqcup V_{k}$ where $V_{0}=\{u\in V\,:\, \Lambda_{u}(\tau)=0 \}$ and $	V_{i+1} = \left( \bigcup_{u\in V_{i}} N(u) \right) \setminus V_{i}$ for $1\le i \le k$. Note that $0\le k \le d$ by construction. Observe that since $\omega(\phi_{\bullet}(\tau))\in [0,1/2)$, all nodes have relative phase $\in [0,1/2)$ for all times $t\ge \tau$ and $\Lambda_{v}(t)\equiv 0$ for all $v\in V_{0}$ and $t\ge \tau$. For each $x\in V$ and $i\ge 1$, let $t_{x;i}$ be the $i^{\text{th}}$ time that $B_{x}(t)\cap \{ \sigma_{x}(t^{+})=0 \}$ occurs after time $\tau$. By the hypothesis, $t_{x;i+1}\le t_{x;i}+3+1/2$ for each $i\ge 0$, where we take $t_{x;0}=\tau$.

	Fix $v_{1} \in V_{1}$. By the construction of $V_{i}$'s and the hypothesis, $v_{1}$ is pulled by some neighbor in $V_{0}$ at some time $t_{2}\in (t_{v_{1};1},t_{v_{1};1}+1/2]$. Then either $\sigma_{v_{1}}(t_{2})=0$ and $\Lambda_{v_{1}}(t_{2}^{+})\in [0,1/4)$ (see Figure \ref{fig:fig_widthlemma} $a_{1}\rightarrow a_{3}$) or $\sigma_{v_{1}}(t_{2})=1$. In the latter case, $\Lambda_{v_{1}}(t_{2})\in [0,1/4)$ by Proposition \ref{prop:E_v(t)_needs_enough_pull} (i) (see Figure \ref{fig:fig_widthlemma} $b_{1}\rightarrow b_{3}$). Now applying the same argument for $t_{v_{1};2}$, we get $\Lambda_{v_{1}}(t)=0$ for some $t\le t_{v_{1};2}+1/2 \le \tau + 7$. This holds for all $v_{1}\in V_{1}$, so all nodes in $V_{0}\cup V_{1}$ have relative phase $0$ by time $\tau+7$. Then the assertion follows by repeating the same argument to synchronize all $V_{i}$'s. 
\end{proof}

Next, we state and prove a localized version of Lemma \ref{widthlemma} in the case of adaptive 4-coupling. For any interval $I\subseteq [0,\infty)$ and a subgraph $H\subseteq G$, we say the joint trajectory $(\Sigma_{\bullet}(t))_{t\ge 0}$ \textit{restricts on $H$ during $I$} if the coupling commutes with restriction on $H$ during $I$,  that is, we have 
\begin{equation*}
\Sigma_{v}(t) = \Psi_{v}(t) \quad \text{$\forall v\in V(H)$ and $t\in I$},
\end{equation*}   
where $(\Psi_{\bullet}(t))_{t\ge 0}$ is the joint trajectory on $H$ (independent of $G-H$) evolving from the restricted initial joint configuration $\Psi_{\bullet}(0)=\Sigma_{\bullet}(0)|_{V(H)}:V(H)\rightarrow \Omega$. If $I=[\tau,\infty)$ for some $\tau\ge 0$, then we say the joint trajectory $(\Sigma_{\bullet}(t))_{t\ge 0}$ restricts \textit{after} time $\tau$. 

Now we state our localized version of Lemma \ref{widthlemma}. 

\begin{lemma}[branch width lemma]\label{lemma:branchwidth}
	Let $G=(V,E)$, $(\Sigma_{\bullet}(t))_{t\ge 0}$, $(\Lambda_{\bullet}(t))_{t\ge 0}$ as before. Let $B\subseteq G$ be a branch with center $v$, root $w$, and its set of leaves $L$. Suppose $\Lambda_{v}(t_{0}^{+})=1/2$ for some $t_{0}\ge 6$. 
	\vspace{0.1cm}
	\begin{description}
		\item[(i)] Suppose $\Lambda_{u}(t_{0}^{+})\in [1/2,3/4)$ for all $u\in L$. Then no leaves in $B$ pull $v$ during $[t_{0},\infty)$, and $\omega_{B}(s^{+})<1/4$ whenever $s\ge t_{0}$ and $B_{v}(s)$ occurs. Furthermore, the joint trajectory restricts on $G-L$ during $[t_{0},\infty)$;
		\vspace{0.1cm}
		\item[(ii)] If $\omega_{B}(s_{0}^{+})<1/4$ for some $s_{0}\ge 3$, then (i) holds for some $t_{0}\in [s_{0},s_{0}+3]$. 
	\end{description}
\end{lemma}

The intuition is that the dynamics on $B$ can only be perturbed by a single node $w$, which can only do so often (at most once in every 1 seconds in our case) under an inhibitory pulse-coupling. Once $\omega_{B}(t)$ is small enough, the branch width remains bounded. This then allows us to restrict the global dynamics on a proper subgraph $G-L$.

\begin{proof}
	Let $L=\{ u_{1},\cdots,u_{\ell} \}$ and write $\lambda_{i}=\Lambda_{u_{i}}(t_{0}^{+})$ for each $1\le i \le \ell$. We may assume without loss of generality that all $\lambda_{i}$'s are distinct, so $1/2 \le \lambda_{1}<\lambda_{2}<\cdots<\lambda_{\ell}<3/4$ (e.g., see Figure \ref{fig:branchwidth1} $a_{2}$). For each $x\in V$ and $i\ge 1$, let $\tau_{x;i}$ be the $i^{\text{th}}$ time that $v$ blinks after time $t_{0}$. These times are well-defined by Proposition \ref{prop:blinking}.

	\begin{figure*}[h]
		\centering
		\includegraphics[width = 0.8 \linewidth]{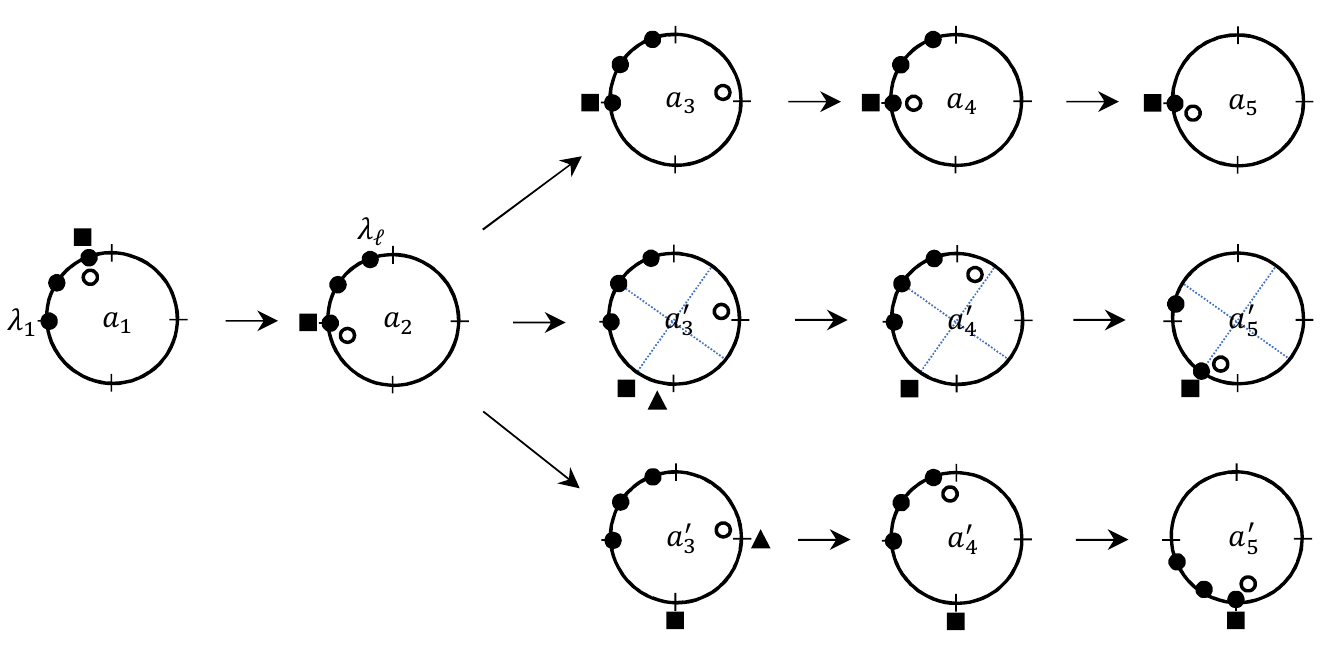}
		\caption{ (In relative circular representation) Local dynamics chasing on a branch $B$ with branch width $<1/4$. $\blacksquare =v$ center, $\bullet = \text{leaves}$, $\blacktriangle=$root, and $\circ= $ activator. External fluctuation on the local dynamics of the branch by its root is weaker than its own restoring force. 
		}
		\label{fig:branchwidth1}
	\end{figure*}
	
	To show (i), since $\tau_{v;n}\rightarrow \infty$ as $n\rightarrow \infty$, it is enough to show that at time $\tau_{v;i}$ for each $i\ge 1$, we are back to the similar local configuration on $B$ and during each interval $(\tau_{v;i},\tau_{v;i+1}]$ none of the leaves in $B$ pulls $v$. Note that since $\Lambda_{v}(t_{0}^{+})=1/2$ and $\lambda_{i}\in [1/2,3/4)$ for all $1\le i \le \ell$, $P_{v}(t)$ never occurs during $(\tau_{w;1},\tau_{v;1}]$. Hence $\Lambda_{v}(\tau_{v;1})=\Lambda_{v}(\tau_{w;1}^{+})\in [1/4,1/2]$ and none of $u_{i}$'s pull $v$ during $[t_{0},\tau_{v;1}]$. Moreover, since all leaves in $B$ have state 0 for all times $t\ge t_{0}$ by Proposition \ref{prop:statebasic} (i), we also get $\Lambda_{u_{i}}(\tau_{v;1}^{+})\in [\Lambda_{v}(\tau_{v;1}),\Lambda_{v}(\tau_{v;1})+1/4)$ for all $1\le i \le \ell$. (See., e.g., Figure \ref{fig:branchwidth1} starting from $a_{2}$). Thus after re-centering so that $\Lambda_{v}(\tau_{v;1}^{+})=1/2$, we recover the same situation with $\omega_{B}(\tau_{v;1}^{+})<1/4$. The same argument applies on successive intervals $(\tau_{v;i},\tau_{v;i+1}]$. This shows (i).
	
	(ii) follows easily from (i). Indeed, if $\sigma_{v}(s_{0}^{+})=0$, then the hypothesis of (i) is satisfied at time $\tau_{u_{1};1}^{+}$ and $\tau_{u_{1};1}\le s_{0}+1$ (see Figure  \ref{fig:branchwidth1} $a_{1}\rightarrow a_{2}$). Otherwise, let $s_{1}=\inf\{ s\ge s_{0}\,:\, \sigma_{v}(s^{+})=0 \}$. Then during $[s_{0},s_{1}]$ all nodes in $B$ keeps the same relative phase, so by the previous case we can apply (i) for some time $t_{0}\le s_{1}+1\le s_{0}+3$. This shows (ii).  
\end{proof}

Our first application of lemma \ref{lemma:branchwidth} for 1-branches is given below.  

\begin{proposition}\label{prop:1-branch}
	Let $G=(V,E)$, $(\Sigma_{\bullet}(t))_{t\ge 0}$, $(\Lambda_{\bullet}(t))_{t\ge 0}$ as before. Suppose $G$ has a $1$-branch $B$ with leaf $u$. Then $\omega_{B}(s_{1}^{+})<1/4$ for some $s_{1}\le 13$ and the joint trajectory restricts on $G-L$ after time $t=16$.
\end{proposition} 

\begin{proof}
	Let $v$ and $w$ be the center and root $B$. By Proposition \ref{prop:statebasic} (i), $\sigma_{u}(t)\equiv 0$ for all $t\ge 3$. By Proposition \ref{prop:blinking}, there exists $t_{0}\in [3,10]$ such that both $B_{v}(t_{0})$ and $\sigma_{v}(t_{0}^{+})=0$ occur. Suppose without loss of generality that $\Lambda_{v}(t_{0})=1/2$. For each $x\in V$ and $i\ge 1$, denote by $\tau_{x;i}$ the $i^{\text{th}}$ blinking time of $x$ after time $t_{0}$. By Lemma \ref{lemma:branchwidth}, it suffices to show the first part of the assertion. We assume $\Lambda_{u}(t_{0}^{+})\notin (1/4,3/4)$ since otherwise $\omega_{B}(t_{0}^{+})<1/4$.

	\begin{figure*}[h]
		\centering
		\includegraphics[width = 0.85 \linewidth]{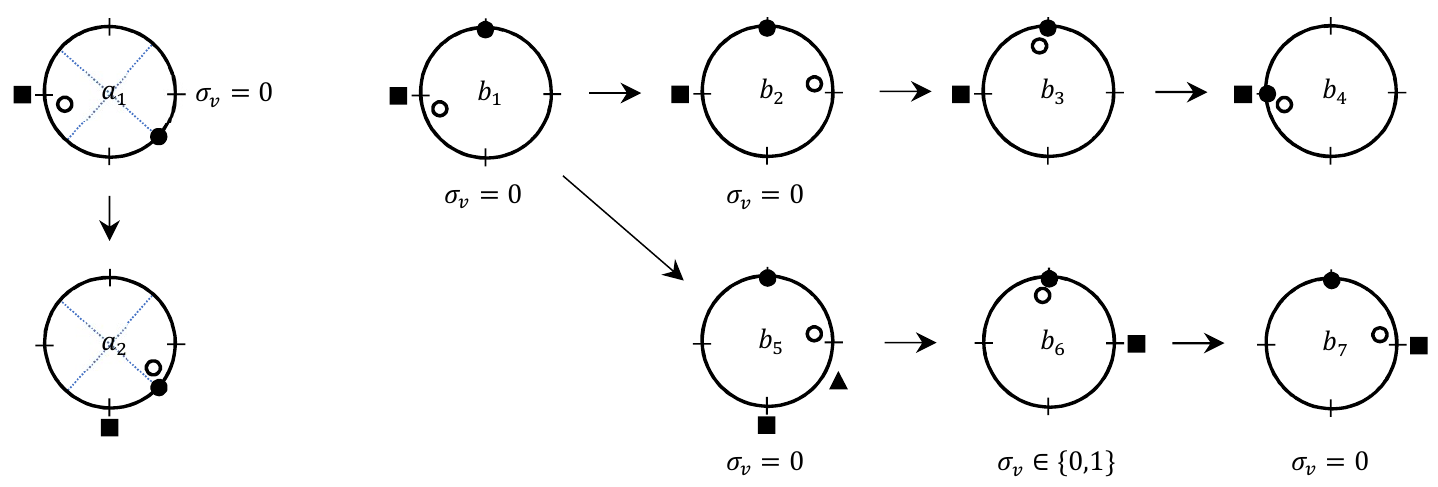}
		\caption{ 
			(In relative circular representation) Local dynamics chasing on 1-branch $B$ from time $t_{0}^{+}$. Its branch width becomes $<1/4$ in a constant time.  $\blacksquare =v$ (center), $\bullet = u$ (left), $\blacktriangle=w$ (root), and $\circ= $ activator. 
		}
		\label{fig:1branchpruning}
	\end{figure*} 	
	
	If $\Lambda_{u}(t_{0}^{+})\in (0,1/2]$, then $\tau_{u;1}\le t_{0}+1/2$ and $\omega_{B}(\tau_{u_{1};1}^{+})<1/4$ (see, e.g., Figure \ref{fig:1branchpruning} $a_{1}\rightarrow a_{2}$). Otherwise $\Lambda_{u}(t_{0}^{+})\in [1/2,3/4]$ and we may assume $\Lambda_{u}(t_{0}^{+})=3/4$ (as in Figure \ref{fig:1branchpruning} $b_{1}$). 
	If $\Lambda_{v}(t_{0}^{+}+1/2)\in (1/4,1/2]$, then $P_{v}(\tau_{u_{1};1})$ does not occur, $\tau_{v;1}\le t_{0}+5/4$, and $\omega_{B}(\tau_{v;1}^{+})<1/4$ (as in Figure \ref{fig:1branchpruning} $b_{1}\rightarrow b_{4}$). Otherwise $\Lambda_{v}(t_{0}^{+}+1/2)=1/4$, which requires $\tau_{w;1}\in [t_{0}+1/4,t_{0}+1/2]$. Then $P_{v}(\tau_{u;1})$ occurs, $\Lambda_{v}(\tau_{u;1}^{+})=0$, $\sigma_{v}(\tau_{u;1}^{+})\in \{0,1\}$, and $\tau_{v;1}=t_{0}+3/2$. Hence if we denote $t_{1}:=t_{0}+3/2+\mathbf{1}[\sigma_{v}(\tau_{v;1})=1]\le t_{0}+5/2$, at time $t_{1}^{+}$ we are back to the previous case in Figure \ref{fig:1branchpruning} $a_{1}$. Thus we get $\omega_{B}(t_{2}^{+})<1/4$ for some $t_{2}\le t_{1}+1/2 \le t_{0}+3\le 13$.
\end{proof}

The following proposition, combined with Lemma \ref{lemma:branchwidth}, shows that if the center of a branch ever jumps into state $1$ (gets excited), then its leaves become irrelevant of the global dynamics after some constant burn-in period. 

\begin{proposition}\label{prop:branch_excitation}
	Let $G=(V,E)$, $(\Sigma_{\bullet}(t))_{t\ge 0}$, and $(\Lambda_{\bullet}(t))_{t\ge 0}$ be as before. Let $B\subseteq G$ be a branch with center $v$ and set $L$ of leaves. Suppose $\sigma_{v}(t_{0})=0$ and $E_{v}(t_{0})$ occurs for some $t_{0}\ge 3$. Then $\omega_{B}(t_{1}^{+})<1/4$ for some $t_{1}\le t_{0}+8$. 
\end{proposition}

\begin{proof}
	By Proposition \ref{prop:1-branch} we may assume $k\ge 2$. By Proposition \ref{prop:branch_basic} (i), all leaves in $B$ have state 0 for all times after $t\ge 3$. Let $s_{0}=\inf\{ s\ge t_{0}+1\,:\, \sigma_{v}(s)=0 \}\le t_{0}+2$. Then $\mu_{v}(s_{0}^{+})=(1,0,0)$. Let $\{u_{1},\cdots, u_{k}\}$ be the set of leaves in $B$. Denote $\lambda_{i}=\Lambda_{u_{i}}(t_{1}^{+})$ for $1\le i \le k$ and for each $x\in V$ and $j\ge 1$, denote by $\tau_{x;j}$ the $j^{\text{th}}$ time that $x$ blinks after time $s_{0}$. We may assume without loss of generality that $\Lambda_{v}(s_{0}^{+})=1/2$. Then by Proposition \ref{prop:branch_basic} (i), we may assume $1/2\ge \lambda_{0}>\lambda_{1}>\cdots>\lambda_{k}>0$. 
	
	\begin{figure*}[h]
		\centering
		\includegraphics[width = 0.98 \linewidth]{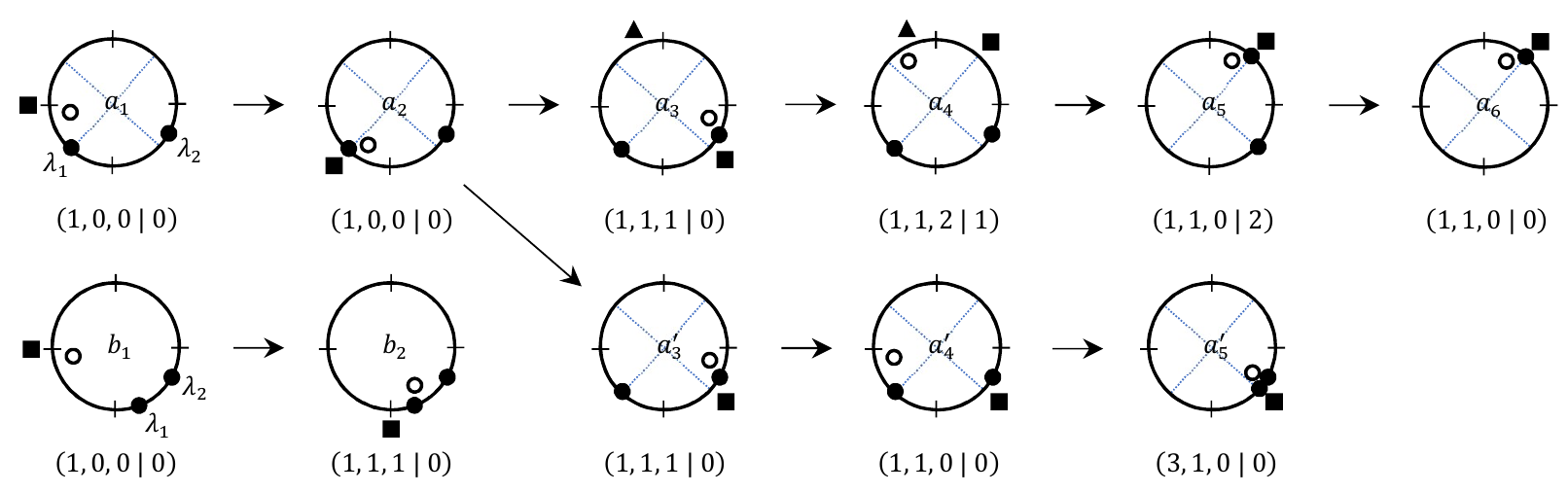}
		\vspace{-0.2cm}
		\caption{ (In relative circular representation) Example of local dynamics on $B+w$ during from time $s_{0}^{+}$ to $\tau_{v;2}^{+}$. $\blacksquare =v$ center, $\bullet=$leaves, $\blacktriangle=w$ root, and $\circ= $ activator. Below each local configuration, $(a,b,c \,|\, d)$ denotes $\mu_{v}=(a,b,c)$ and $\sigma_{v}=d$. 
		}
		\label{fig:2branch_rested}
	\end{figure*}
	
	First we show that if $k=2$, then $\omega_{B}(t_{1}^{+})<1/4$ for some $t_{1}\le s_{0}+3$.  Note that $\lambda_{2}\le (1/4,1/2]$ implies $\omega_{B}(s_{0}^{+})<1/4$, and $\lambda_{1}\in (0,1/4]$ gives $\omega_{B}(\tau_{u_{1};1}^{+})<1/4$ (as in Figure \ref{fig:2branch_rested} $b_{1}\rightarrow b_{2}$). So we may assume $\lambda_{1}\in (1/4,1/2]$ and $\lambda_{2}\in (0,1/4]$ as in Figure \ref{fig:2branch_rested} $a_{1}$. Note that $P_{v}(\tau_{u_{2};1})$ occurs and $P_{v}(t)$ occurs at most twice during $(s_{0}+1/4,s_{0}+1]$. Thus $\Lambda_{v}(s_{0}+1)\in [\lambda_{1}-1/2,1/4]$. If $\Lambda_{v}(s_{0}+1)\in (\lambda_{1}-1/2,1/4]$, then $P_{v}(t)$ does not occur during $[s_{0}+1,\tau_{v;1}]$ so $\tau_{v;1}\le s_{0}+3/4+1$ and $\Lambda_{v}(\tau_{v;1})\in (\lambda_{1}-1/2,1/4]$. This gives $\omega_{B}(\tau_{v;1}^{+})<1/4$ (see, e.g., Figure \ref{fig:2branch_rested} $a_{3}\rightarrow a_{5}$). Otherwise $\Lambda_{v}(s_{0}+1)=\lambda_{1}-1/2$, which requires $P_{v}(\tau_{w;1})$ to occur between $\tau_{u_{2};1}$ and $s_{0}+1$. Since $\mu_{v}^{1}(t)\equiv 1$ during $(s_{0},s_{0}+1]$, $E_{v}(\tau_{w;1})$ occurs. Then $\tau_{v;1}\le s_{0}+3/4+1$,  $\tau_{v;2}=\tau_{v;1}+1$, and $\omega_{B}(\tau_{v;2}^{+})=0$ (see Figure \ref{fig:2branch_rested} $a_{3}\rightarrow a_{6}$). This shows the assertion for $k=2$. 
	
	Now assume $k\ge 3$ and we show $\omega_{B}(t_{1}^{+})<1/4$ for some $t_{1}\le t_{0}+8$. Suppose $\lambda_{k}\in (0,1/4)$ and let $r=\min\{ 1\le i \le k\,:\, \lambda_{i}\in (0,1/2) \}$. If $r\le k-1$ as in Figure \ref{fig:3branch_excitation} $a_{1}$, then $\Lambda_{v}(\tau_{u_{r+1};1})\in (0,1/4)$ and $E_{v}(\tau_{u_{r+1};1})$ occurs, so we have $\omega_{B}(\tau_{u_{r+1};1}^{+}+2)<1/4$ and $\tau_{u_{r+1};1}^{+}+2\le s_{0}+1/2 \le t_{0}+5/2$. (See transitions $a_{1}\rightarrow a_{8}$ in Figure \ref{fig:3branch_excitation}).

	\begin{figure*}[h]
		\centering
		\includegraphics[width = 1 \linewidth]{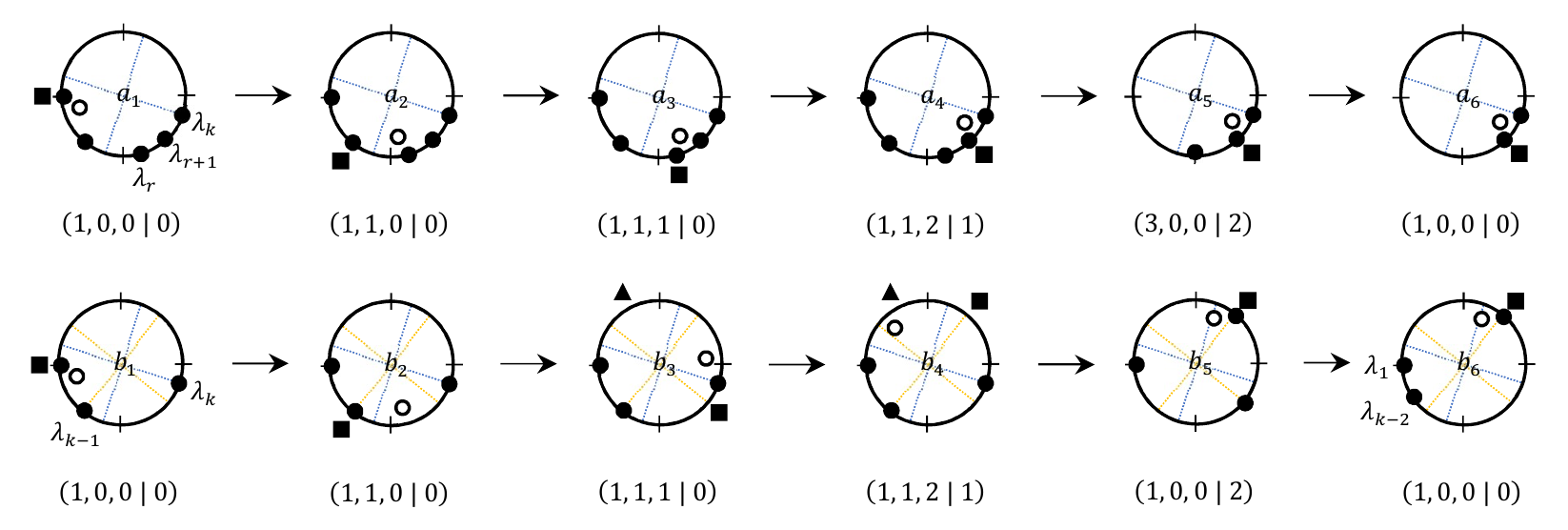}
		\vspace{-0.5cm}
		\caption{ 
			(In relative circular representation) A local configuration chasing on branch $B$ from time $s_{0}^{+}$. $\blacksquare =v$ center, $\bullet = \text{leaves}$, $\circ= $ activator, and $\blacktriangle = w$ root. Below each local configuration, $(a,b,c\,|\, d)$ abbreviates $\mu_{v}(t)=(a,b,c)$ and $\sigma_{v}(t)=d$.  
		}
		\label{fig:3branch_excitation}
	\end{figure*}  
	
	Suppose $r=k$ as in Figure \ref{fig:3branch_excitation} $b_{1}$. If $\Lambda_{v}(s_{0}+1)\in (0,1/2]$, then $\tau_{v;1}\le s_{0}+3/2$ and $\omega_{B}(\tau_{v;1}^{+})<1/4$, as similarly as in the previous case in Figure \ref{fig:3branch_excitation} $a_{1}\rightarrow a_{2}\rightarrow a_{3}'\rightarrow a_{5}'$. Hence we may assume $\Lambda_{v}(s_{0}+1)\in [1/2,1]$, which requires $\tau_{w;1}\in [s_{0}+1/2,s_{0}+1)$ and $P_{v}(\tau_{w;1})$ occurs. Then $E_{v}(\tau_{w;1})$ occurs and $\Lambda_{v}(\tau_{w;1}^{+})\in (3/4,1]$. Hence $\tau_{v;2}\le \tau_{w;1}+2\le s_{0}+ 3/4+2$ and at time $\tau_{v;2}^{+}$, either the leaves in $B$ have two distinct phases or we are back to the previous case of $r\le k-1$ (see the transition $b_{1}\rightarrow b_{8}$ in Figure \ref{fig:3branch_excitation}). Hence by the previous cases, we have $\omega_{B}(t_{1}^{+})<1/4$ for some $t_{1}\le \tau_{v;2}+3\le s_{0}+3/4+5 \le t_{0}+3/4+7$.  This shows the assertion.
\end{proof}

Lastly in this section, we summarize some of our observations in the next proposition for later use. 

\begin{proposition}\label{prop:2branch_rested}
	Let $G=(V,E)$, $(\Sigma_{\bullet}(t))_{t\ge 0}$, and $(\Lambda_{\bullet}(t))_{t\ge 0}$ be as before. Let $B\subseteq G$ be a $2$-branch. If $\omega_{B}(t)\ge 1/4$ for all $t\in [0,16]$, then $\sigma_{x}(t)\equiv 0$ for all $t\ge 7$ and $x\in V(B)$.
\end{proposition}

\begin{proof}
	Combining Propositions \ref{prop:statebasic} (i), the assertion holds if $x$ is a leaf in $B$. Let $v$ be the center of $B$. By Proposition \ref{prop:blinking}, $v$ blinks at some time $t_{0}\in [2,7]$. If $\sigma_{v}(t_{0})\in \{1,2\}$, then there exists $s_{0}\in [0,t_{0})$ such that $\sigma_{v}(s_{0})=0$ and $E_{v}(s_{0})$ occurs. So by Proposition \ref{prop:branch_excitation} we get $\omega_{B}(t_{1})<1/4$ for some $t_{1}\le s_{0}+8 \le 15$, a contradiction. Hence $\sigma_{v}(t_{0})=0$, so $\mu_{v}(t_{0}^{+})=(3,0,0)$ and $\sigma_{v}(t_{0}^{+})=0$. Then $\sigma_{v}(t)\equiv 0$ for all $t\ge t_{0}\ge 7$ by Proposition \ref{prop:statebasic} (ii). 
\end{proof}

\vspace{0.3cm}
\section{Proof of main results assuming three lemmas about local limit cycles}
\label{Section:Proof of main results}

In this section, we prove our main results, Theorems \ref{4treethm} and \ref{A4Ctreethm}, assuming a key lemma (Lemma \ref{key}). We derive Lemma \ref{key} from three lemmas of about local limit cycles on branches (Lemmas \ref{lemma:branch_attraction}, \ref{lemma:branchorbit_a}, and \ref{lemma:branchorbit_b}), which we will prove in Section \ref{section:locallemmas}.

Let $T=(V,E)$ be a finite tree. A branch $B\subseteq T$ is called \textit{terminal} if all but at most one neighbors of its root are either leaves or centers of other branches. For instance, suppose $V(T)\ge 3$ and let $P\subseteq T$ be a longest path with vertices $v_{0},v_{1},v_{2},\cdots,v_{l}$ in order. Then $v_{1}$ is the center of a terminal branch rooted at $v_{2}$. Moreover, if there is anther branch $B'\ne B$ rooted at $v_{2}$, then it is also a terminal branch. 

In  Lemma \ref{lemma:branchwidth}, we showed that if any branch ever has branch width $<1/4$, then we can delete its leaves without affecting the global dynamics. In Lemma \ref{key}, we show that the adaptive 4-coupling is capable of breaking possible symmetries of local phase configurations on $B$, and shrinks its branch width below the threshold $1/4$ in a constant time: 

\begin{lemma}[key lemma]\label{key}
	Let $T=(V,E)$ be a finite tree with maximum degree $\Delta$ and let $(\Sigma_{\bullet}(t))_{t\ge 0}$ be a joint trajectory on $T$. Let $B$ be any terminal branch in $T$. Then $\omega_{B}(t_{1}^{+})<1/4$ for some $t_{1}\in [4,D_{\Delta}]$, where $D_{\Delta}=94 + 63 \cdot \mathbf{1}[ \Delta\ge 4 ]$.
\end{lemma}

Assuming this lemma, we give a proof of Theorem \ref{4treethm} and \ref{A4Ctreethm} here. The idea is simple: we can then restrict the joint trajectory on $T$ to the subtree obtained by deleting the leaves in all terminal branches in $T$, and use induction on the diameter.

\begin{proof}[Proof of Theorem \ref{4treethm} and \ref{A4Ctreethm}] 
	
	For Theorem \ref{4treethm} (ii), let $T=(V,E)$ be a finite tree with a node $v$ of degree $m\ge 4$. Then we can basically reproduce the example given in Figure \ref{fig:ex_4C}, Namely, let $T_{1},\cdots,T_{m}$ be the connected components of $T-v$. Consider a phase configuration $\phi_{\bullet}(0):V\rightarrow S^{1}$ such that $\phi_{T_{i}}(0)\equiv i/4$ mod $1$ for $i=1,2,\cdots,m$ and $\phi_{v}(0)=1/4$. Then all nodes in each component $T_{i}$ blinks simultaneously once in every 1 second, and the `center' $v$ is pulled by 1/4 in phase each time one of its neighbor blink. As in the example in Figure \ref{fig:ex_4C}, $v$ gets pulled constantly and never blinks. Hence $\phi_{\bullet}(t)$ does not synchronize under the 4-coupling.   
	
	Now we show Theorem \ref{A4Ctreethm}. By Proposition \ref{prop:statebasic} (iii), this also yields Theorem \ref{4treethm} (i). Let $T=(V,E)$ be a tree with diameter $d$ and maximum degree $\Delta$.  Fix an arbitrary trajectory $(\Sigma_{\bullet}(t))_{t\ge 0}$ of joint configurations on $T$ evolving through the adaptive 4-coupling. We use induction on $d$ to show that arbitrary dynamics synchronizes by time $C_{\Delta} d$. Note that  $C_{\Delta}\le (D_{\Delta}+8)/2$. For $d\le 1$, $T$ is isomorphic to a singleton or $K_{2}$ and the assertion is easy to check. We may assume $d\ge 2$. Let $T'$ be the subtree of $T$ obtained by deleting all leaves in all terminal branches, which has diameter $d-2$. 
	
	By Lemmas \ref{key} and \ref{lemma:branchwidth}, the dynamic on $T$ restricts on $T'$ after time $D_{\Delta}+3$.  By the induction hypothesis, all nodes in $T'$ synchronize by time $t_{1}\in [0,D_{\Delta}+3+C_{\Delta}(d-2)]$. Also note that for each node $u\in V(T')$, $\sigma_{u}(t)\equiv 0$ and $P_{u}(t)$ never occurs during $[t_{1}+2,\infty)$. Now we may contract $T'$ into a single node, say, $v$. The resulting tree $T/T'$ is a isomorphic to a star with center $v$ and maximum degree $|V(T-T')|$. Note that $v$ blinks once in every 1 second after time $t\ge t_{1}$ in the induced dynamics. When it blinks at some time $t_{2}\in (t_{1}+2,t_{1}+3]$ all leaves in $T/T'$ must be left to $v$, since otherwise $P_{v}(t)$ occurs during $(t_{2},t_{2}+1/2]$, contrary to our assumption. This implies $\omega(\phi_{\bullet}(t_{2}^{+}))\le 1/4$ and $\omega(\phi_{\bullet}(t_{2}^{+}+1))=0$. Thus we obtain synchrony on $T$ by time $t=t_{2}+2 \le t_{1}+5 \le D_{\Delta}+C_{\Delta}(d-2)+8 \le C_{\Delta}d$. This completes the induction. This shows the assertion.
\end{proof}

Note that Proposition \ref{prop:1-branch} yields a special case of Lemma \ref{key} when $B$ is a 1-branch. Our strategy for the general case is the following. We first show that if $\omega_{B}(t)\ge 1/4$ for all $t\le \mathtt{b}$, then $B+w\subseteq G$ undergoes a particular `conditional' local limit cycle during some window inside $[0,b]$. Now if we assume $G$ is a tree and $B$ is a terminal branch, then each of `non-synchronizing' terminal branches rooted at $w$ induces a constraint on the local dynamics of the common root $w$. These constraints are not compatible, so we can assume that there is at most one non-synchronizing terminal branch rooted at $w$, whose local limit cycle must be `supported' by the input on $w$ from its leaves plus at most one extra neighbor. This will be shown to be impossible for large $\mathtt{b}$, thereby showing that $\omega_{B}(t_{0})<1/4$ for some $t_{0}\le \mathtt{b}$.  

In order to give a more precise description of the two local limit cycles, consider the following two particular local configurations on a branch. Suppose for some $t_{0}\ge 0$, we have $\phi_{v}(t_{0}^{+})=0$ and $\{ \phi_{u}(t_{0}^{+})\,:\, u\in N(v)\setminus \{w \}  \} = \{ 0,\lambda \}$ for some $\lambda\in (1/4, 1/2)$. In relative circular representation, such a local instance is represented by Figure \ref{fig:2branch_limiting_config} $a$, and we say that ``the branch $B$ has local configuration in Figure \ref{fig:2branch_limiting_config} $a$ at time $t_{0}^{+}$". Similarly, the local configuration in Figure \ref{fig:2branch_limiting_config} $b$ represents the instance at $t_{0}^{+}$ where $\phi_{v}(t_{0}^{+})=0$ and $\{ \phi_{u}(t_{0}^{+})\,:\, u\in N(v)\setminus \{w \}  \} = \{ 0,1/4 \}$. 

\begin{figure*}[h]
	\centering
	\includegraphics[width = 0.35 \linewidth]{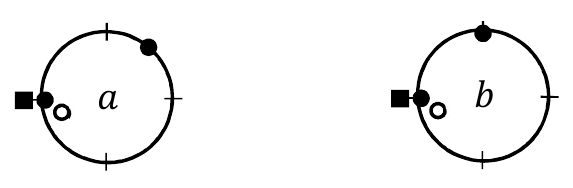}
	\caption{ (In relative circular representation) Two `absorbing' local configurations on a branch, where $\blacksquare$, $\bullet$, and $\circ$ denote the relative phases of center, leaves, and the activator, respectively. 
	}
	\label{fig:2branch_limiting_config}
\end{figure*} 

In the following lemma, we show that any $k$-branch $B$ will either have a small branch width $<1/4$ or have one of the two local configurations in Figure \ref{fig:2branch_limiting_config} in at most first 32 seconds. 

\begin{lemma}\label{lemma:branch_attraction}
	Let $G=(V,E)$ be a finite graph and let $(\Sigma_{\bullet}(t))_{t\ge 0}$ and $(\Lambda_{\bullet}(t))_{t\ge 0}$ be as before. Let $B\subseteq G$ be a $k$-branch rooted at $w\in V$. Define $E_{k}=16+21\cdot \mathbf{1}[ k\ge 3 ]$. Then there exists $t_{1}\in [8,E_{k}]$ such that we have either of the following. 
	\vspace{0.2cm}
	\begin{description}
		\item[(i)] $\omega_{B}(t_{1}^{+})<1/4$.
		\vspace{0.1cm}
		\item[(ii)] At time $t_{1}^{+}$, all leaves in $B$ have at most two distinct phases and $B$ has one of the two local configurations in Figure \ref{fig:2branch_limiting_config}. 
	\end{description}
\end{lemma}

Lemma \ref{lemma:branch_attraction} reduces it down to analyzing possible local orbits on $B$ involving those two local configurations. We will observe that starting from either of the two local configurations in Figure \ref{fig:2branch_limiting_config}, it is indeed possible to have large branch width $\omega_{B}\ge 1/4$ for arbitrarily long periods of time by going through one of the two local limit cycles on $B+w\subseteq T$ in Figure \ref{fig:2branch_limit}.

We next address that there are essentially no other ways to avoid having small branch width $<1/4$ than going through the local limit cycles in Figure \ref{fig:2branch_limit}. More precisely, for the local configuration \ref{fig:2branch_limiting_config} $a$, we will show the following lemma.

\begin{lemma}\label{lemma:branchorbit_a}
	Let $T=(V,E)$ be a finite tree, and let $(\Sigma_{\bullet}(t))_{t\ge 0}$ and $(\Lambda_{\bullet}(t))_{t\ge 0}$ be as before. Let $B\subseteq T$ be a 2-branch rooted at $w\in V$. Suppose for some $t_{0}\ge 8$, $B$ has the local configuration in Figure \ref{fig:2branch_limiting_config} $a$ at time $t_{0}^{+}$. Then for any $\mathtt{a}\ge 17$, we have either of the following. 
	\vspace{0.1cm}
	\begin{description}
		\item[(i)] $w_{B}(t^{+})< 1/4$ for some $t\in [0,t_{0}+\mathtt{a}]$.
		\vspace{0.1cm}
		\item[(ii)] The local dynamics on $B+w\subseteq T$ during $(t_{0},t_{0}+\mathtt{a}-2]$ undergo the limit cycle given by the transitions in Figure \ref{fig:2branch_limit} (a). 
		\vspace{0.1cm}
		\item[(iii)] $B$ has the local configuration in Figure \ref{fig:2branch_limiting_config} $b$ at time $t_{1}^{+}$ for some $t_{1}\le t_{0}+\mathtt{a}-13$.
	\end{description}
	\vspace{0.1cm}
	Furthermore, if $t_{0}\ge 37$ and (ii) holds, then 
	\vspace{0.1cm}
	\begin{description}
		\item[(iv)] There exists a subtree $T'\subseteq T$ such that $w$ is the center of a branch in $T'$ and the dynamics on $T$ restrict on $T'$ during $(t_{0},t_{0}+\mathtt{a}-6]$.
	\end{description}
	
\end{lemma}

\begin{figure*}[h]
	\centering
	\includegraphics[width = 0.95 \linewidth]{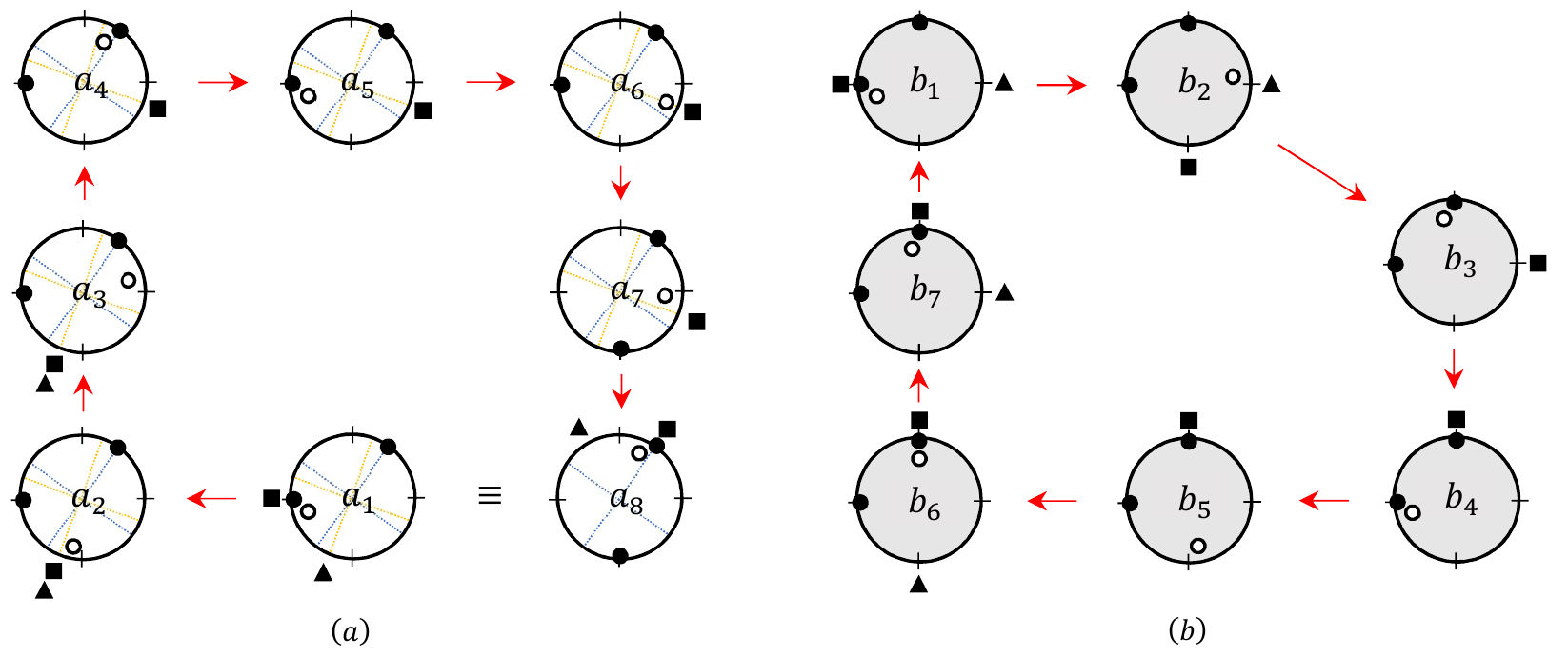}
	\vspace{-0.2cm}
	\caption{ (In relative circular representation) A conditional local limit cycle on $B+w$, where each transition arrow takes at most 1/2 second.	$\blacksquare$, $\bullet$, and $\circ$ denote the relative phases of center, leaves, and the activator, respectively. Relative phases of $w$ are indicated by $\blacktriangle$ except $c_{2},c_{3},c_{4}$, and $c_{5}$, where they are undetermined (hence not shown).
	}
	\label{fig:2branch_limit}
\end{figure*}

For the local configuration \ref{fig:2branch_limiting_config} $b$, we have the following lemma.

\begin{lemma}\label{lemma:branchorbit_b}
	Let $T=(V,E)$ be a finite tree, and let $(\Sigma_{\bullet}(t))_{t\ge 0}$ and $(\Lambda_{\bullet}(t))_{t\ge 0}$ be as before. Let $B\subseteq T$ be a 2-branch rooted at $w\in V$. Suppose for some $t_{0}\ge 8$, $B$ has the local configuration in Figure \ref{fig:2branch_limiting_config} $b$ at time $t_{0}^{+}$. Then for any $\mathtt{b}\ge 26$, either of the following holds:
	\vspace{0.1cm}
	\begin{description}
		\item[(i)] $\omega_{B}(t^{+})< 1/4$ for some $t\in [0,t_{0}+\mathtt{b}]$.
		\vspace{0.1cm}
		\item[(ii)]  $B+w$ undergoes the transitions in Figure \ref{fig:2branch_limit} $(b)$ during $(t_{0}+2, t_{0}+\mathtt{b}-8]$.
	\end{description}
\end{lemma}

In the rest of this section, we give a proof of Lemma \ref{key} assuming Lemmas \ref{lemma:branch_attraction}, \ref{lemma:branchorbit_a}, and \ref{lemma:branchorbit_b}.

\begin{proof}[Proof of Lemma \ref{key}] 
	
	Let $T=(V,E)$ be a finite tree with maximum degree $\Delta$. Fix a terminal branch $B$ with root $w$ and center $v$. Let $\mathtt{a}\ge 17$ and $D_{\Delta}\ge E_{\Delta-1}$ be some constants to be determined later. Suppose $\omega_{B}(t)\ge 1/4$ for all $[0,D_{\Delta}]$. Then by Lemma \ref{lemma:branch_attraction}, there exists $t_{0}\in [3,E_{\Delta-1}]$ such that $B$ has one of the two local configurations in Figure \ref{fig:2branch_limit_a_pf} at time $t_{0}^{+}$. Then by Lemma \ref{lemma:branchorbit_a} and \ref{lemma:branchorbit_b}, either of the following must hold:
	
	\begin{description}
		\vspace{0.1cm}
		\item{(i)} $B+w$ undergoes the local limit cycle in Figure \ref{fig:2branch_limit} (a) during $(t_{0}, t_{0}+\mathtt{a}-2]$.
		\vspace{0.1cm}
		\item{(ii)} $B+w$ undergoes the local limit cycle in Figure \ref{fig:2branch_limit} (b) during $(E_{\Delta-1}+\mathtt{a}-13,D_{\Delta}-8]$,
	\end{description}
	\vspace{0.1cm}
	We will show that both statements are impossible if $\mathtt{a}\ge 2E_{\Delta-1}+24$ and $D_{\Delta}\ge \mathtt{a}+22+E_{\Delta-1}$. This will show that the assertion holds for $D_{\Delta}=3E_{\Delta-1}+46$.  
	
	Suppose (i) holds. Then by Lemma \ref{lemma:branchorbit_a} (iv), there exists a subtree $T'\subseteq T$ such that $w$ is the center of some branch $B_{w}\subseteq T'$ and the dynamics on $T$ restrict on $T'$ during $(t_{0},t_{0}+\mathtt{a}-4]$. Let $w'\in V(T')$ be the root of $B_{w}$. By applying Lemma \ref{lemma:branch_attraction} to the branch $B_{w}$ restarting the dynamics at $t=t_{0}$, $B_{w}$ must have one of the local configurations in Figure \ref{fig:2branch_limiting_config} at time $t_{1}^{+}$ for some $t_{1}\in [t_{0},t_{0}+E_{\Delta-1}]$. Hence at least one of the conclusions of Lemma \ref{lemma:branchorbit_a} or \ref{lemma:branchorbit_b} must hold. We will show that none of them holds. 
	
	First we show that $\omega_{B_{w}}(t)\ge 1/4$ for all $t\in [t_{0}, t_{0}+\mathtt{a}-13]$. Suppose for contrary that $\omega_{B_{w}}(t)<1/4$ for some $t\in [t_{0},t_{0}+\mathtt{a}-13 ]$. Then by Lemma \ref{lemma:branchwidth}, the leaves in $B_{w}$ do not affect the dynamics on $w$ during $[t_{0}+\mathtt{a}-10, t_{0}+\mathtt{a}-4 ]$. Recall the definition of $f_{v}(I)$ in \eqref{eq:total_phase_inhibition} in Section \ref{Section:The width lemma and the branch width lemma}.  
	Hence $f_{w}((s,s+1])\ge -1/4$ for any $s\in [t_{0}+\mathtt{a}-10, t_{0}+\mathtt{a}-5 ]$. But note that the transition $a_{1}\rightarrow a_{8}\equiv a_{1}\rightarrow a_{8}$ in Figure \ref{fig:2branch_limit} occurs on $B+w$ at least once during $[t_{0}+\mathtt{a}-10, t_{0}+\mathtt{a}-4 ]$, which requires $f_{w}((s_{1},s_{2}])<-1/2$, where $(s_{1},s_{2}]$ is the interval during which the first cycle $a_{1}\rightarrow a_{8}$ occurs. Since $|s_{2}-s_{1}|<2$, this yields a contradiction. Thus we may assume $\omega_{B_{w}}(t)\ge 1/4$ for all $t\in [t_{0}, t_{0}+\mathtt{a}-13]$.
	
	\begin{figure*}[h]
		\centering
		\includegraphics[width = 0.7 \linewidth]{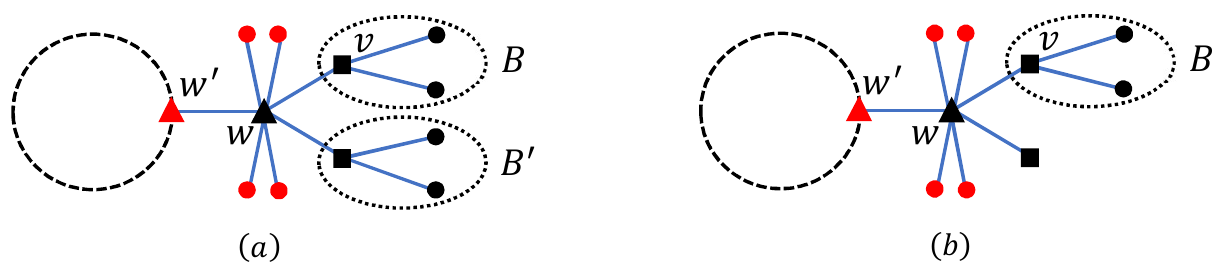}
		\vspace{-0.2cm}
		\caption{ (a) Terminal branch $B\subseteq T$ is rooted at $w$, which might have another terminal branch $B'$ also rooted at itself and other neighbors of degree 1 shown in red dots. At most one neighbor of $w$, say $w'$, may be neither of degree 1 or center of another terminal branch. (b) Assuming $\omega_{B}(t)\ge 1/4$ for long times, we can omit all leaves of $B'$ without affecting the global dynamics.
		}
		\label{fig:tree_substructure}
	\end{figure*}    
	
	Then restarting the dynamics on $T'$ at time $t=E_{\Delta-1}$, Lemmas \ref{lemma:branchorbit_a} and \ref{lemma:branchorbit_b} yield that either of the two transitions $a_{1}\rightarrow a_{8}\equiv a_{1}\rightarrow a_{8}$ or $b_{1}\rightarrow b_{1}\rightarrow b_{1}$ in Figure \ref{fig:2branch_limit} must occur on $B_{w'}+w'$ during $(t_{1}, t_{0}+\mathtt{a}-21]$. But none of them are compatible with the local limit cycle in Figure \ref{fig:2branch_limit} on $B+w$ during this period. Indeed, in the latter case, the center $w$ of the branch $B_{w}\subseteq T'$ must blink twice in exactly 2 seconds at the same phase during some interval $I\subset (t_{1}, t_{0}+\mathtt{a}-21]$ of length 4. But from the local limit cycle on $B+w$, we deduce that $w$, as the root of $B$, should blink once in every $<2$ seconds in different phases. In the former case, as the center of $B_{w}$, $w$ is pulled twice between each consecutive blinks (by its root and one of its leaves); but as the root of $B$, the total phase inhibition during each consecutive blink is $<-1/2$, which requires at least three pulls. Therefore given that $t_{0}+\mathtt{a}-21\ge t_{1}+6$, which is satisfied if $\mathtt{a}\ge 2E_{\Delta-1}+24$, we rule out (i). In fact, notice that this holds for any branch rooted at $w$, which is also a terminal branch. Hence, for any such $\mathtt{a}$, we assume that (ii) holds for all branches in $T$ rooted at $w$.

	Next, suppose (ii) holds. We claim that there exists a subtree $T''\subseteq T$ such that $B$ is the unique terminal branch in $T''$ rooted at $w$ and the joint trajectory restricts on $T''$ during $(t_{0}+\mathtt{a}+2,D_{\Delta}-8]$ (see Figure \ref{fig:tree_substructure} (b)). Suppose this is true Then provided  $D_{\Delta}\ge \mathtt{a}+22+E_{\Delta-1}$, the transition $b_{1}\rightarrow b_{1}$ in Figure \ref{fig:2branch_limit} occurs at least five times during $(t_{0}+\mathtt{a}+2,D_{\Delta}-8]$. Note that $w$ blinks exactly once in every 2 seconds at relative phase 0 (right before Figure \ref{fig:2branch_limit} $b_{2}$) during the five cycles. Hence in $b_{1}$ of the third cycle, there is no leaf neighbor of $w$ with relative phase $\in (0,1/2]$. Also, notice that $w$ must keep state 0 during the first four cycles.  Since $w$ keeps the same relative phase during $b_{7}\rightarrow b_{1}$ in the second cycle, it follows that no leaf neighbor of $w$ has relative phase $\in [1/2,3/4)$ at $b_{1}$ of the third cycle. But then $w$ can only be pulled by its external neighbor $w'\notin V(B)$ during $b_{3}\rightarrow b_{6}$ in the third cycle, yet $f_{w}(I)<-1/2$ during the transition, a contradiction.

	It remains to show the claim. If $B$ is the unique branch in $T$ rooted at $w$, then we are done. Suppose $T$ has another branch $B'\ne B$ rooted at $w$. Let $L'$ be the set of leaves in $B'$. It is enough to show that the dynamics on $T$ restrict on either $T-L'$ or $T-B'$ during $(t_{0}+\mathtt{a}+2,D_{\Delta}-8]$. According to Lemma \ref{lemma:branchwidth}, we may assume that $\omega_{B'}(t)\ge 1/4$ for all $t \in [0,t_{0}+\mathtt{a}-1]$. Then by a similar argument for $B$, either of the transition $a_{1}\rightarrow a_{8}\equiv a_{1}\rightarrow a_{8}$ or $b_{1}\rightarrow b_{1}\rightarrow b_{1}$ in Figure \ref{fig:2branch_limit} must occur on $B'+w$ during $[E_{\Delta-1},t_{0}+\mathtt{a}-9]$. However, since the local orbits on $B+w$ and $B'+w$ assumes identical local dynamics on the common root $w$, $B'+w$ must undergo the local limit cycle in Figure \ref{fig:2branch_limit} (b) during $(t_{0}+\mathtt{a}-13,t_{0}+\mathtt{a}-9]$. Moreover, such a local limit cycle is uniquely determined given the blinking times of $w$. So both $B$ and $B'$ must have identical local dynamics during this period. But then since they get affected by their common root $w$, their local dynamics are identical for all times $t\ge t_{0}+\mathtt{a}-13$. In this case the joint trajectory restricts on $T-B'$ after time $t_{0}+\mathtt{a}-13$, as desired. This shows the claim, and hence the assertion. 
\end{proof}

\vspace{0.3cm}
\section{Proof of Lemmas \ref{lemma:branch_attraction}, \ref{lemma:branchorbit_a}, and \ref{lemma:branchorbit_b}}
\label{section:locallemmas}

In this section, we prove Lemmas \ref{lemma:branch_attraction}, \ref{lemma:branchorbit_a}, and \ref{lemma:branchorbit_b}. This will complete the proof of Lemma \ref{key} and hence that of our main results, Theorems \ref{4treethm} and \ref{A4Ctreethm}.

\vspace{0.2cm}
\subsection{Attraction to the local limit cycles.}
\label{subsection:attraction}

In this subsection we prove Lemma \ref{lemma:branch_attraction}. The  following proposition combined with Proposition \ref{prop:branch_excitation} shows that it is enough to consider 2-branches only. 

\newpage
\begin{proposition}\label{prop:3branch_gen_to_right_half}
	Let $G=(V,E)$, $(\Sigma_{\bullet}(t))_{t\ge 0}$, and $(\Lambda_{\bullet}(t))_{t\ge 0}$ be as before. Let $B\subseteq G$ be a branch with center $v$. Suppose $v$ blinks at some time $t_{0}\ge 10$. Then there exists $t_{1}\in [8, t_{0}+6]$ such that one of the followings hold:
	\vspace{0.2cm}
	\begin{description}
		\item[(i)] $\omega_{B}(t_{1}^{+})<1/4$;
		\vspace{0.1cm}
		\item[(ii)] All leaves of $B$ have at most two distinct states at time $t_{1}^{+}$;
		\vspace{0.1cm}
		\item[(iii)] $\sigma_{v}(t_{1})=0$ and $E_{v}(t_{1})$ occurs.
	\end{description} 
\end{proposition}

\begin{proof}
	Let $w$ be the root of $B$ and $L=\{u_{1},\cdots,u_{k}\}$ be the set of leaves in $B$. We may assume $k\ge 3$ since otherwise (ii) holds for $t_{1}=t_{0}$. By Proposition \ref{prop:statebasic} (i), all $u_{i}$'s have state $0$ for all times after $t=3$. If $\sigma_{v}(t_{0})\in \{1,2\}$, then (iii) holds for some $t_{1}\in [8,t_{0}]$. Hence we may assume $\sigma_{v}(t_{0})=0$ and so $\mu_{v}(t_{0}^{+})=(3,0,0)$. For any $x\in V$ and $j\ge 1$, denote by $\tau_{x;j}$ the time of the $j^{\text{th}}$ blink of $x$ after time $t_{0}$. For $x=u_{i}$, we write $\tau_{u_{i};j}=\tau_{i;j}$. We may assume without loss of generality that $\Lambda_{v}(t_{0})=1/2$. Denote $\lambda_{i}=\Lambda_{u_{i}}(t_{0}^{+})$.  
	Proposition \ref{prop:branch_basic} (ii) then yields $\lambda_{i}\in [1/4,3/4]$ for all $1\le i \le k$. We may assume that $\lambda_{i}$'s are distinct. We also assume $\lambda_{1}\in [1/2,3/4]$ since otherwise (i) holds for $t_{1}=\tau_{1;1}< t_{0}+1/4$.

	\begin{figure*}[h]
		\centering
		\includegraphics[width = 0.9 \linewidth]{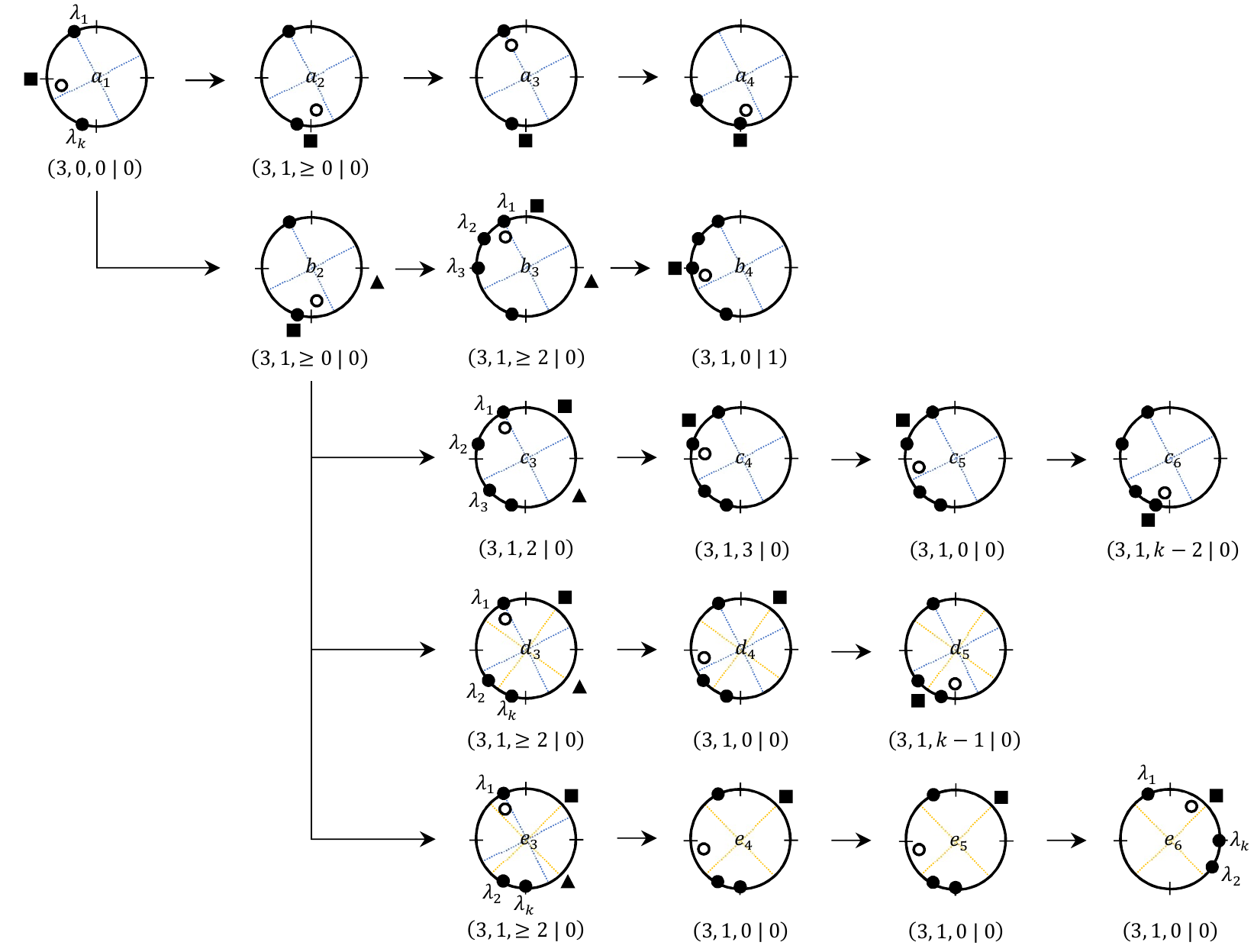}
		\caption{ (In relative circular representation) Forward tracking local dynamics on a $k$-branch $B$ ($k\ge 3$) starting from time $t_{0}^{+}$ when $\lambda_{1}\in [1/2,3/4)$ and $\lambda_{k} \in [1/4,1/2]$. Below each local configuration,  $(a,b,c\,|\, d)$ abbreviates $\mu_{v}=(a,b,c)$ and $\sigma_{v}=d$. 
		}
		\label{fig:3branch_gen_1}
	\end{figure*}
	
	We first show the assertion when $\lambda_{1}\in [1/2,3/4)$. If $P_{v}(\tau_{1;1})$ does not occur, then $P_{v}(t)$ does not occur during $[\tau_{1;1},\tau_{v;1}]$ so $\Lambda_{v}(\tau_{v;1})=\Lambda_{v}(\tau_{1;1})\in (\lambda_{1}-1/2,1/2]$. So $\omega_{B}(\tau_{v;1}^{+})< (\lambda_{1}-1/4)-(\lambda_{1}-1/2)=1/4$, so (i) holds for some $t_{1}=\tau_{v;1}\le t_{0}+3/2$, as desired. (See, e.g., transition $a_{1}\rightarrow a_{4}$ in Figure \ref{fig:3branch_gen_1}). Hence we may assume $P_{v}(\tau_{1;1})$ occurs, which requires $\tau_{w;1}\in (\tau_{k;1},\tau_{1;1})$ and $P_{v}(\tau_{w;1})$ to occur. Furthermore, if $\lambda_{3}\in [1/2,3/4]$, then $E_{v}(\tau_{2;1})$ or $E_{v}(\tau_{3;1})$ occurs (see Figure \ref{fig:3branch_gen_1} $b_{2}\rightarrow b_{4}$). Thus we assume $\lambda_{3}\in [1/4,1/2)$. 
	
	\begin{figure*}[h]
		\centering
		\includegraphics[width = 0.9 \linewidth]{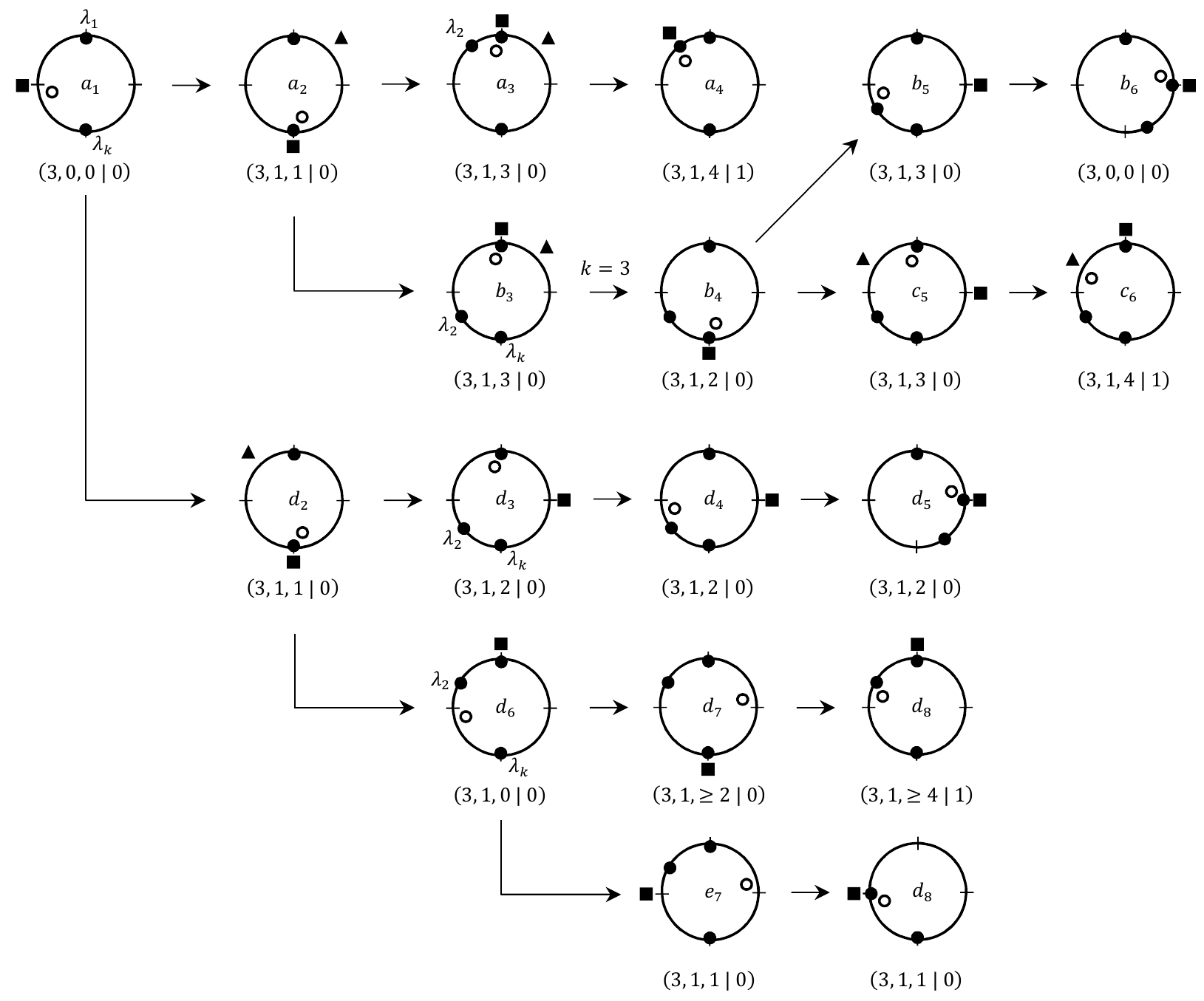}
		\caption{ (In relative circular representation) Forward tracking local dynamics on a $k$-branch $B$ ($k\ge 3$) starting from time $t_{0}^{+}$ when $\lambda_{1}=3/4 $ and $\lambda_{k}=1/4$.  Below each local configuration, $(a,b,c\,|\, d)$ abbreviates $\mu_{v}=(a,b,c)$ and $\sigma_{v}=d$. 
		}
		\label{fig:3branch_gen_2}
	\end{figure*}

	If $\lambda_{2}\in [1/2,3/4]$ (as in Figure \ref{fig:3branch_gen_1} $c_{3}$), then $P_{v}(\tau_{j;2})$ occurs for all $3\le j \le k$ so $\Lambda_{v}(\tau_{k;2}^{+})\in [1/4,1/2]$ and $\mu_{v}(\tau_{k;2}^{+})=(3,1,k-2)$ (as in Figure \ref{fig:3branch_gen_1} $c_{6}$). This is similar to the previous case in Figure \ref{fig:3branch_gen_1} $a_{2}$ so $P_{v}(\tau_{1;2})$ (and hence $P_{v}(\tau_{w;2})$) must occur. But since we know $\mu_{v}^{3}(\tau_{k;2}^{+})=k-2 \ge 1$, $E_{v}(\tau_{2;2})$ occur, as desired. Hence we assume $\lambda_{2}\in [1/4,1/2)$ as in Figure \ref{fig:3branch_gen_1} $d_{3}$ or $e_{3}$. If $P_{v}(\tau_{2;2})$ occurs then so does $P_{v}(\tau_{j;2})$ for all $3\le j \le k$; hence we have $\Lambda_{v}(\tau_{2;2}^{+})\in [1/4,1/2]$ and $\mu_{v}(\tau_{v;2}^{+})=(3,1,k-1)$. Hence a similar argument for the case in Figure \ref{fig:3branch_gen_1} $c_{6}$ applies. Otherwise $P_{v}(\tau_{j;2})$ does not occur for all $2\le j \le k$; but then $\tau_{v;1}\le t_{0}+7/4$ and at time $\tau_{v;1}^{+}$, we are back to the original configuration at time $t_{0}^{+}$. After recentering $\Lambda_{v}(\tau_{1;1}^{+})=1/2$, we have $\lambda_{j}\in [1/2,3/4)$ for all $2\le j \le k$ and $\lambda_{1}\in [1/4,1/2]$. Hence the local dynamics from time $\tau_{1;1}^{+}$ follows the transitions in Figure \ref{fig:3branch_gen_1} and lead to one of the previous cases $b_{4}$ or $c_{6}$. Hence the assertion holds for some $t_{1}\le t_{0}+4$.

	Next, suppose $\lambda_{1}=3/4$. By a similar argument the assertion holds if $P_{v}(\tau_{1;1})$ does not occur. If both $P_{v}(\tau_{w;1})$ and $P_{v}(\tau_{1;1})$ occur by time $\tau_{1;1}$, then an entirely similar argument we used in the previous case applies (see, e.g., the transitions on $a_{i}$'s, $b_{i}$'s, and $c_{i}$'s in Figure \ref{fig:3branch_gen_2} for $\lambda_{1}=3/4$ and $\lambda_{k}=1/4$). Note that if $\lambda_{k}\ne 1/4$, then $P_{v}(\tau_{w;1})$ must occur for $P_{v}(\tau_{1;1})$ to occur. Hence we may assume that $\lambda_{k}=1/4$ and $\Lambda_{w}(\tau_{k;1}^{+})\in [1/4,3/4]$ (as in Figure \ref{fig:3branch_gen_2} $d_{2}$). Then $\mu_{v}(\tau_{1;1}^{+})=(3,1,2)$, so $\lambda_{3}\in [1/4,3/4]$ implies $E_{v}(\tau_{3;1})$ occurs. So assume $\lambda_{3}\in [1/4,1/2)$.

	If $\lambda_{2}\in [1/4,1/2)$, then $\Lambda_{v}(\tau_{1;1}^{+})=0$ and $P_{v}(t)$ does not occur during $(\tau_{1;1},\tau_{v;1}]$, so $\tau_{v;1}=t_{0}+3/2$ and at time $\tau_{v;1}^{+}$, we are back to the previous case in Figure \ref{fig:3branch_gen_1} $a_{1}$. Hence in this case the assertion holds for $t_{1}\le t_{0}+3/2+4$. Otherwise $\lambda_{2}\in [1/2,3/4]$ and $\Lambda_{v}(t_{0}^{+}+1)=3/4$ as in Figure \ref{fig:3branch_gen_2} $d_{6}$. Then $\mu_{v}(t_{0}^{+}+1)=(3,1,0)$ and $P_{v}(\tau_{k;2})$ occurs, so $c:=\mu_{v}^{3}(t_{0}^{+}+3/2)\ge 1$. If $c\ge 2$, then $\Lambda_{v}(t_{0}^{+}+3/2)=1/4$ and $E_{v}(\tau_{2;2})$ occurs. Otherwise $c=1$ and $\Lambda_{v}(t_{0}^{+}+3/2)=1/2$, so $P_{v}(t)$ does not occur during $(t_{0}+3/2,t_{0}+2]$ and hence $\tau_{v;1}=t_{0}+2$. Then the leaves in $B$ have only two relative phases (namely, $1/4$ and $1/2$) at time $t_{0}^{+}+2$ (see $d_{6}\rightarrow e_{7}\rightarrow e_{8}$ in Figure \ref{fig:3branch_gen_2}), as desired. Thus the assertion holds for some $t_{1}\le t_{0}+6$.
\end{proof}

Next, we show Lemma \ref{lemma:branch_attraction} for 2-branches.

\begin{proposition}\label{2branchtoabsorbing}
	Let $G=(V,E)$, $(\Sigma_{\bullet}(t))_{t\ge 0}$, and $(\Lambda_{\bullet}(t))_{t\ge 0}$ be as before. Suppose $G$ has a 2-branch $B$. Then either $\omega_{B}(t_{1}^{+})<1/4$ for some $t_{1}\le 16$ or $B$ has a local configuration in Figure \ref{fig:2branch_limiting_config} at time $t_{2}^{+}$ for some $t_{2}\in [7,14]$. 
\end{proposition}

\begin{proof}
	Suppose for contrary that $\omega_{B}(t)\ge 1/4$ for all $t\in [0,16]$. Then all nodes in $B$ have state 0 for all times $t\ge 7$ by Proposition \ref{prop:2branch_rested}. Let $v$ and $w$ be the center and root of $B$, respectively. Let $t_{0}=\inf\{ t\ge 7\,:\, \text{$B_{v}(t)$ occurs} \}$. Then $t_{0}\le 12$ by Proposition \ref{prop:blinking}. Let $\Lambda_{v}(t_{0})=1/2$ and let $L=\{u_{1},u_{2}\}$ be the set of leaves in $B$. Denote $\lambda_{i}:= \Lambda_{u_{i}}(\tau_{0}^{+})$ for $i\in \{1,2\}$. Then $\lambda_{i}\in (0,3/4]$. We may assume that $\omega_{B}(t_{0}^{+})\ge 1/4$ and $B$ does not have one of the local configurations in Figure \ref{fig:2branch_limiting_config}.

	\begin{figure*}[h]
		\centering
		\includegraphics[width = 1 \linewidth]{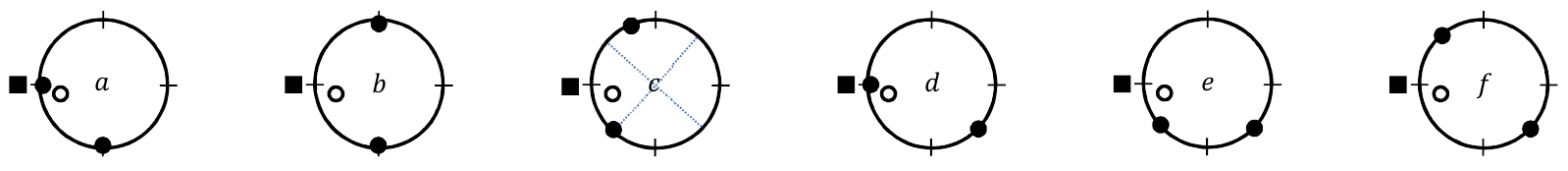}
		\vspace{-0.1cm}
		\caption{ (In relative circular representation) Possible Local configurations on $B$ right after its center blinks. $\blacksquare =v$ center, $\bullet = \text{leaves}$,  and $\circ= $ activator.
		}
		\label{fig:2branch_0}
	\end{figure*}  
	
	By interchanging the roles of $u_{1}$ and $u_{2}$, we may assume that $B$ has one of the following classes of local configurations in Figure \ref{fig:2branch_0}: (a) $(\lambda_{1},\lambda_{2})=(1/2,1/4)$, (b) $(\lambda_{1},\lambda_{2})=(3/4,1/4)$, (c) $\lambda_{1}\in (1/2,3/4]$ and $\lambda_{1}-\lambda_{2}\in [1/4,1/2)$, (d) $\lambda_{1}=1/2$ and $\lambda_{2}\in (0,1/4)$, (e) $\lambda_{1}\in (1/4,1/2)$ and $\lambda_{2}\in (0,1/4)$, and lastly (f) $\lambda_{1}\in (1/2,3/4]$, $\lambda_{2}\in (0,1/4]$, and $(\lambda_{1},\lambda_{2})\ne (3/4,1/4)$. For cases (a) and (c), $B$ gets a desired local configuration in 1/4 second. Similarly, case (e) leads to (d) in 1/4 second. Hence we only need to consider cases (b), (d), and (f). In all cases, we show the second assertion holds for some $t_{2}\le t_{0}+2 \le 14$.

	\begin{figure*}[h]
		\centering
		\includegraphics[width=0.9 \textwidth]{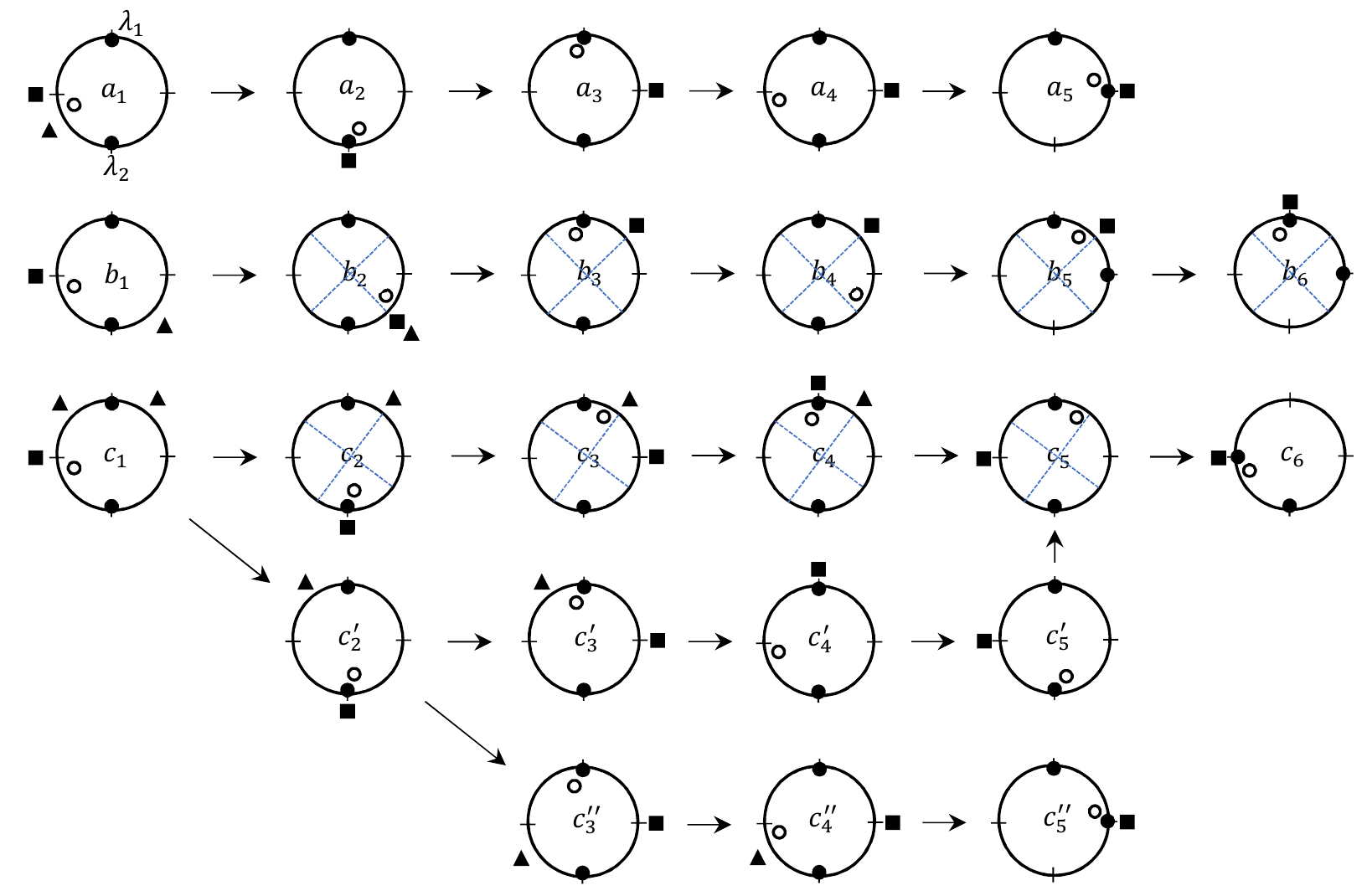}
		\caption{ 
			(In relative circular representation) A local configuration chasing on branch $B$. $\blacksquare =v$ center, $\bullet = \text{leaves}$, $\circ= $ activator, and $\blacktriangle = w$ root.
		}
		\label{fig:4tree_combined_1}
	\end{figure*}

	Now suppose $B$ has the local configuration Figure \ref{fig:2branch_0} $b$ at time $t_{0}^{+}$. If $\Lambda_{w}(t_{0}^{+})\in [1/4,1/2]$ as in Figure \ref{fig:4tree_combined_1} $a_{1}$, then $\Lambda_{v}(t_{0}^{+}+1/4)=1/4$ as in the transition $a_{1}\rightarrow a_{2}$. Then $v$ is not pulled by $w$ during $(t_{0}+1/4,t_{0}+1]$ but it is pulled by the leaf in $B$ at relative phase $-1/4$ at time $t_{0}^{+}+3/4$. So we have $a_{2}\rightarrow a_{3}\rightarrow a_{4}$. Then $v$ pulls one of the leaf in $B$ at time $t_{0}+3/2$, so $a_{4}$ leads to $a_{5}$, which is the local configuration in Figure \ref{fig:2branch_limiting_config} $b$. Hence the assertion holds for $t_{2}=t_{0}+3/2+1/4$. An entirely similar argument shows that when $\Lambda_{w}(t_{0}^{+})\in [0,1/4)$, the assertion hold for $t_{2}=t_{0}+7/4$ (see Figure \ref{fig:4tree_combined_1} $b_{1}\rightarrow b_{6}$). 
	
	Now let $\Lambda_{w}(t_{0}^{+})\in [1/4,1/2]$ as in Figure \ref{fig:4tree_combined_1} $c_{1}$. Then $\Lambda_{v}(t_{0}^{+}+1/4)=1/4$ and either $\Lambda_{w}(t_{0}^{+}+1/4)\in [-1/4, 0)$ ($c_{1}\rightarrow c_{2}$) or $\Lambda_{w}(t_{0}^{+}+1/4)\in [1/4, 3/4)$ ($c_{1}\rightarrow c_{2}'$). In the former case, the transition $c_{2}\rightarrow c_{6}$ shows that the assertion holds for $t_{2}=t_{0}+2$. In the latter case, either $\Lambda_{w}(t_{0}^{+}+3/4)\in [1/2,3/4)$ and $c_{2}'\rightarrow c_{3}'$, or $\Lambda_{w}(t_{0}^{+}+3/4)\in [0,1/2)$ and $c_{2}'\rightarrow c_{3}''$. In the first case, $v$ is pulled by $w$ during $(t_{0}+3/4,t_{0}+1]$ so we have the transition $c_{3}'\rightarrow c_{5}'\rightarrow c_{5}\rightarrow c_{6}$; hence the assertion holds for $t_{2}=t_{0}+2+1/4$. In the second case, $v$ is not pulled by $w$ during $[t_{0}+4/3,t_{0}+3/2]$, so $c_{3}''\rightarrow c_{5}''$ and the assertion holds for $t_{2}=t_{0}+3/2+1/4$.

	\begin{figure*}[h]
		\centering
		\includegraphics[width=0.9 \textwidth]{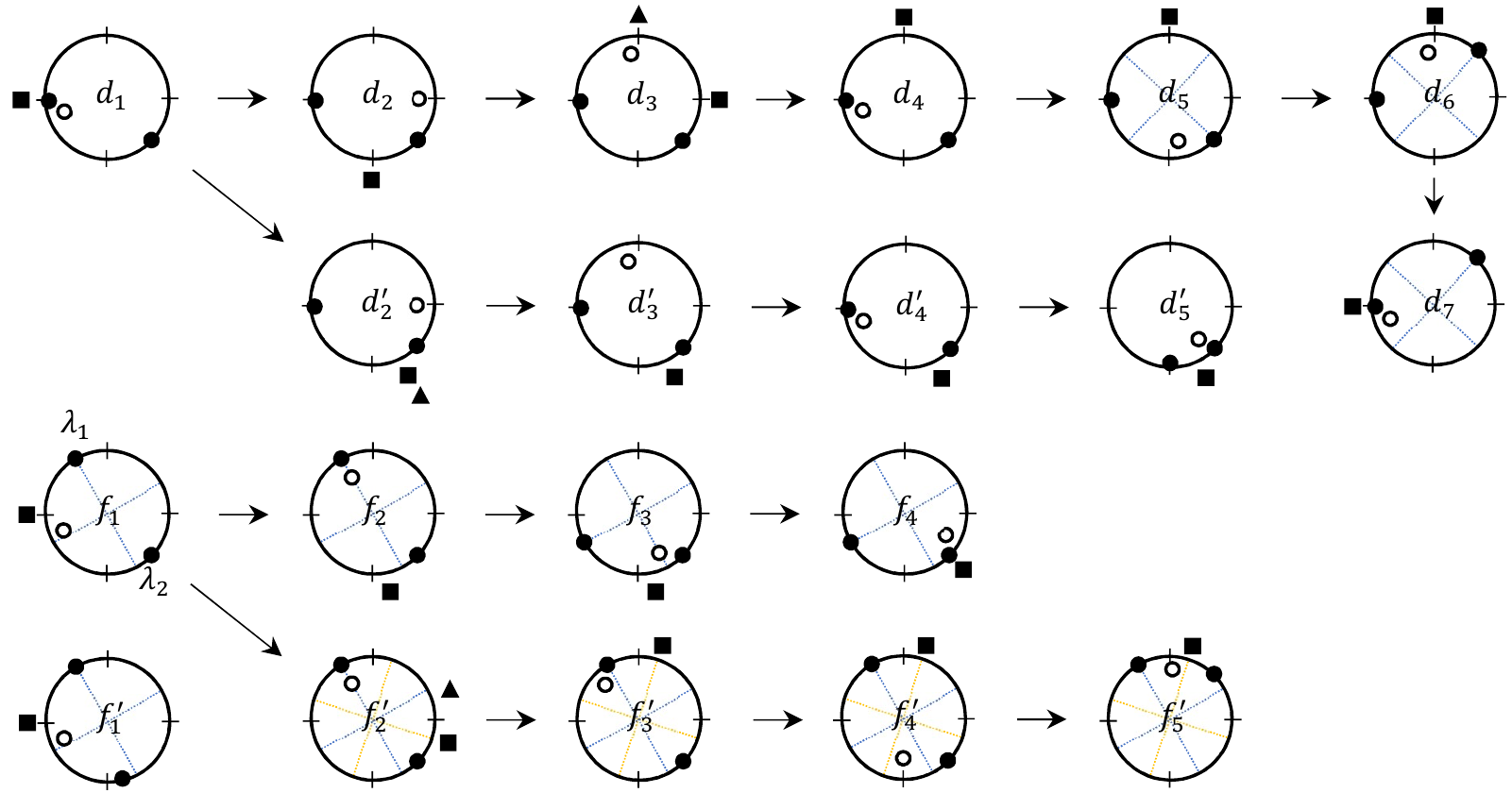}
		\caption{ 
			(In relative circular representation) A local configuration chasing on branch $B$. $\blacksquare =v$ center, $\bullet = \text{leaves}$, $\circ= $ activator, and $\blacktriangle = w$ root.
		}
		\label{fig:4tree_combined_2}
	\end{figure*}  
	
	On the other hand, suppose $B$ has the local configuration in Figure \ref{fig:2branch_0} $d$ at time $t_{0}^{+}$. If $\Lambda_{w}(t_{0}^{+}+3/4)\in (0,1/4]$ (as in  Figure \ref{fig:4tree_combined_2} $d_{3}'$), then $P_{v}(t)$ does not occur during $(t_{0}+3/4,t_{0}+1]$, so $B_{v}(t_{2})$ occur for some $t_{2}\in (t_{0}+5/4,t_{0}+3/2]$ and $\omega_{B}(t_{2}^{+})<1/4$ (see $d_{2}'\rightarrow d_{5}'$ in Figure \ref{fig:4tree_combined_2}). So we may assume $\Lambda_{v}(t_{0}^{+}+3/4)=1/4$ (as in  Figure \ref{fig:4tree_combined_2} $d_{3}$). Then the rest of transition $d_{3}\rightarrow d_{7}$ in Figure \ref{fig:4tree_combined_2} is straightforward. Note that $d_{7}$ is the desired local configuration in Figure \ref{fig:2branch_limiting_config} $a$. Hence the assertion holds for some $t_{2}\le t_{0}+2\le 12$. This also yields the assertion for case (e) holds for some $t_{2}\le t_{0}+3\le 13$. For case (f), a similar argument shows that the assertion holds for some $t_{2}\le t_{0}+2$ (see the transitions between $f^{*}_{i}$'s in Figure \ref{fig:4tree_combined_2}). This shows the assertion.	
\end{proof}

\newpage

At this point Lemma \ref{lemma:branch_attraction} follows easily. 

\begin{proof}[Proof of Lemma \ref{lemma:branch_attraction}]
	For $k\le 2$, by Proposition \ref{2branchtoabsorbing}, either $\omega_{B}(s_{1}^{+})<1/4$ for some $s_{1}\le 16$ or (ii) holds at time $s_{2}^{+}$ for some $s_{2}\in [7,12]$. In the former case, if $s_{1}<7$, we may choose $t_{1}\in [7,12]$ so that $B_{v}(t_{1})$ occurs according to Proposition \ref{prop:blinking}. Then by Lemma \ref{lemma:branchwidth} (i), $\omega_{B}(t_{1}^{+})<1/4$ as desired. Hence the assertion holds with $E_{k}=16$. 
	
	Suppose $k\ge 3$. By Proposition \ref{prop:blinking}, the center $v$ of $B$ blinks at some time $t_{0}\in [10,15]$. If Proposition \ref{prop:3branch_gen_to_right_half} (i) or (ii) holds, then it reduces back to the $k\le 2$ case by time $t=21$ so the assertion holds with $E_{k}\le 37$ by the previous case. Otherwise, Proposition \ref{prop:3branch_gen_to_right_half} (iii) holds for some $t_{1}\in [8,21]$ so by Proposition \ref{prop:branch_excitation}, the assertion holds with $E_{k}\le 21+8$. Hence in all cases, Lemma \ref{lemma:branch_attraction} holds with $E_{k}=16+21\cdot \mathbf{1}[k\ge 3]$. 
\end{proof}

\vspace{0.2cm}
\subsection{Local limit cycles on branches.}
\label{subsection:locallimitcycle}

In this subsection, we analyze local dynamics on a 2-branch, starting from the absorbing local configurations in Figure \ref{fig:2branch_limit}. The goal here is to show Lemmas \ref{lemma:branchorbit_a} and \ref{lemma:branchorbit_b}. We first prove Lemma \ref{lemma:branchorbit_b}. 

\begin{proof}[Proof of Lemma \ref{lemma:branchorbit_b}]
	Suppose that (i) does not hold. Then by Proposition \ref{prop:2branch_rested}, $\sigma_{x}(t)\equiv 0$ for all $t\ge t_{0}$ and $x\in V(B)$.  Forward-tracking of local dynamics on $B+w$ from time $t_{0}^{+}$ is developed in Figure \ref{4tree1}, where each arrow shows a transition between local configuration on $B+w$ of duration at most 1/2 second. The local configuration $b_{1}^{*}$ in Figure \ref{4tree1} represents the one in Figure \ref{fig:2branch_limiting_config} $b$ on $B$ together with $\Lambda_{w}\in [0,1/4]$.

	We claim the following statements in order:
	\vspace{0.2cm}
	\begin{description}[noitemsep]
		\item{(1)} If $\omega_{B}(t^{+})\ge 1/4$ for all $t\in [t_{0},t_{0}+2]$, then $B+w$ has $b_{1}^{*}$ at time $t_{0}^{+}$ and $b_{5}$ at time $t_{0}^{+}+5/4$.    

		\item{(2)} If $\omega_{B}(t^{+})\ge 1/4$ for all $t\in [t_{0},t_{0}+4]$, then $\Lambda_{w}(t_{0}+7/4)\in [0,1/4]$. 
		\vspace{0.1cm} 
		\item {(3)}  Suppose $\omega_{B}(t^{+})\ge 1/4$ for all $t_{1}\in [t_{0},t_{0}+6]$. 
		\vspace{0.1cm}
		\begin{description}[noitemsep]
			\item{(3-1)} If $\Lambda_{w}(t_{0}+7/4)\in [0,1/4)$ ($b_{6}'$ or $b_{6}''$), then $B+w$ has $b_{8}$ or $b_{8}'''$ at time $t_{0}^{+}+2$ and $\sigma_{w}(t_{0}+7/4)=2$.
			\vspace{0.1cm}
			\item{(3-2)} If $\Lambda_{w}(t_{0}+7/4)=1/4$ ($b_{6}$), then either $\sigma_{w}(t_{0}+7/4)=0$ and $B+w$ has $b_{8}$ at time $t_{0}^{+}+2$ or $\sigma_{w}(t_{0}+7/4)=2$ and $B+w$ has $b_{8}'$ at time $t_{0}^{+}+2$. 
		\end{description}
		\vspace{0.1cm}
		\item {(4)} Suppose $\omega_{B}(t^{+})\ge 1/4$ for all $t_{1}\in [t_{0},t_{0}+8]$. If $\sigma_{w}(t_{0}+7/4)\in \{0,2\}$, then  $\sigma_{w}(t_{0}+2+7/4)=0$ and $\Lambda_{w}(t_{0}+7/4)=1/4$.
	\end{description} 
	\vspace{0.2cm}
	
	\begin{figure*}[h]
		\centering
		\vspace{-0.3cm}
		\includegraphics[width = 1 \linewidth]{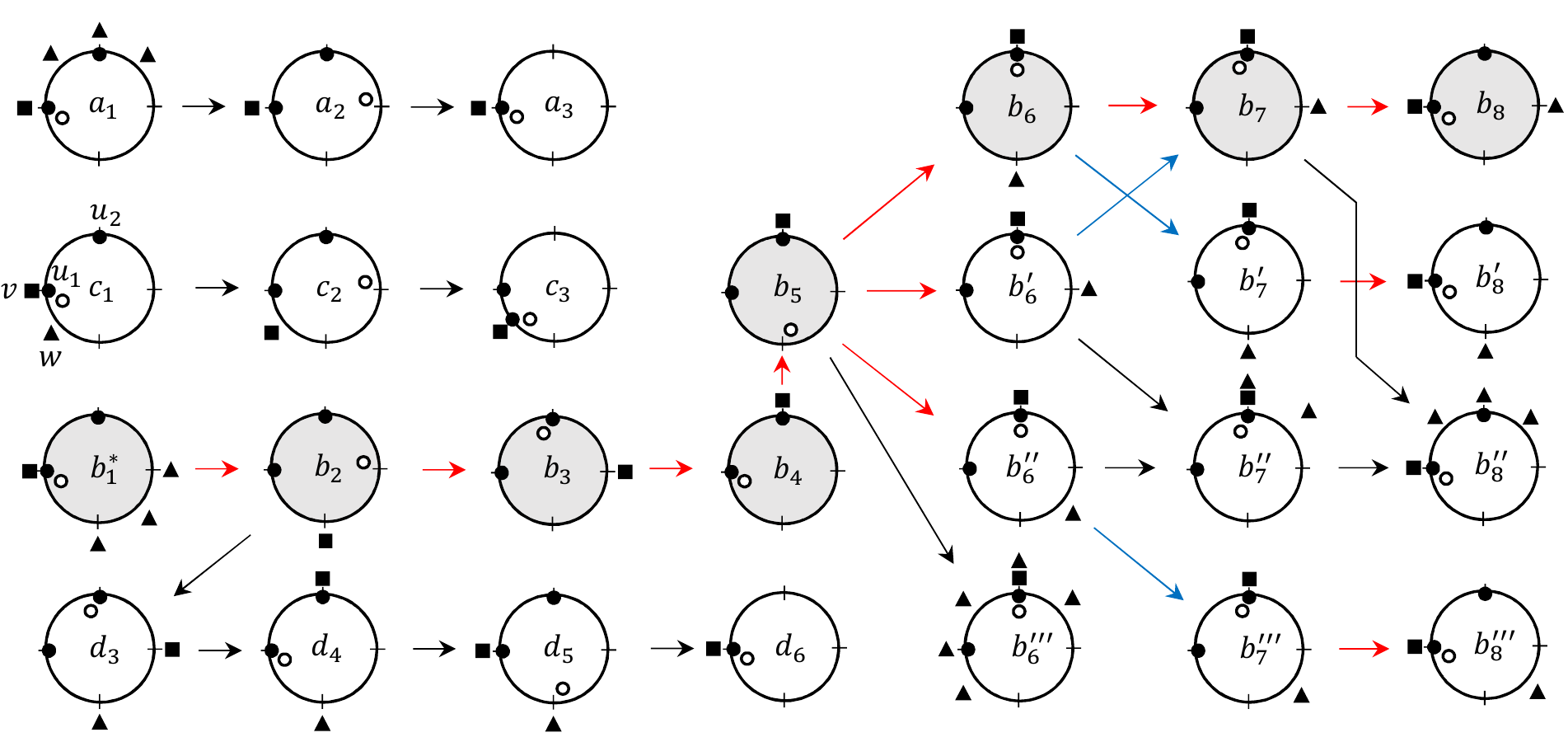}
		\caption{(In relative circular representation) A local dynamics chasing on a 2-branch starting from the local configuration in Figure \ref{fig:2branch_limiting_config} $b$. $\blacksquare =v$ center, $\bullet = \text{leaves}$, $\circ= $ activator, and $\blacktriangle = w$ root. All transitions take at most 1/2 second.
		}
		\label{4tree1}
	\end{figure*} 
	
	We first observe that the assertion follows from the above claim. Suppose (i) does not hold. By (1)-(3), $B+w$ has the local configuration $b_{1}^{*}$ at times $t_{0}+2k$ for all $k\in \mathbb{N}_{0}$ such that $2k\le \mathtt{b}-6$. Then by (2) and (3), $\sigma_{w}(t_{0}+2k+7/4)\in \{0,2\}$ for all such $k$, and by (4) we have $\Lambda_{w}(t_{0}+7/4)=1/4$ so that $\sigma_{w}(t_{0}+2k+7/4)=0$ for all $k\in \mathbb{N}$ with $2k+7/4\le \mathtt{b}-6$. Thus we conclude that $B+w$ repeats the transition $b_{1}^{*}\rightarrow b_{8}$ during $(t_{0}+2, t_{0}+\mathtt{b}-8]$, as desired. 
	
	Now we show the claim. Label the two leaves by $u_{1}$ and $u_{2}$ as in Figure \ref{4tree1} $b_{1}$. The claim is easy if $\Lambda_{w}(t_{0}^{+})\notin [0,1/4]$. Indeed, if $\Lambda_{w}(t_{0}^{+})\in (1/2,1]$, $v$ is not pulled by $w$ in the following 1/2 second so it blinks again after 1 second at the same relative phase, synchronizing the leaves (transition $a_{1}\rightarrow a_{3}$ in Figure \ref{4tree1}). Hence the assertion holds with $\omega_{B}(t_{1}^{+})<1/4$ for some $t_{1} \in (t_{0},t_{0}+1]$. Second, if $\Lambda_{w}(t_{0}^{+})\in (1/4,1/2)$, then $v$ is pulled by $w$ but not enough to be pulled by $u_{2}$, so when $v$ blinks again, it pulls the two leaves and makes the branch width $<1/4$ (transition $c_{1}\rightarrow c_{3}$ in Figure \ref{4tree1}). Hence $\omega_{B}(t_{1})<1/4$ for some $t_{1}\in (t_{0},t_{0}+5/4)$. Thus we may assume $\Lambda_{w}(t_{0}^{+})\in [0,1/4]$ as in Figure \ref{4tree1} $b_{1}^{*}$. Then the transition $b_{1}^{*}\rightarrow b_{2}$ is clear. Suppose $\Lambda_{w}(t_{0}^{+}+3/4)\in [1/4,3/4]$. Since $\Lambda_{w}(t_{0}^{+})\in [0,1/4]$, this implies $\Lambda_{w}(t_{0}^{+}+3/4)=1/4$, so $b_{2}\rightarrow d_{3}$. Then it leads to $d_{6}$, where $v$ pulls the leaf of relative phase $-1/4$ and making $\omega_{B}(t_{0}^{+}+1)=0$. Thus $\Lambda_{w}(t_{0}^{+}+3/4)\in (-1/4,1/4)$ so $b_{2}$ leads to $b_{5}$ by time $t_{0}^{+}+5/4$. This shows (1). 
	
	To show (2), suppose $\omega_{B}(t^{+})\ge 1/4$ for all $t_{1}\in [t_{0},t_{0}+4]$. By (1), $B+w$ has the local configuration $b_{5}$ at time $t_{0}^{+}+5/4$. If $\Lambda_{w}(t_{0}+7/4)\in (1/4, 1)$ as in Figure \ref{4tree1} $b'''_{6}$, then at time $t_{0}^{+}+2$ it leads back to the previous cases of Figure \ref{4tree1} $a_{1}$ or $c_{1}$; hence $\omega_{B}(t_{1})<1/4$ for some $t_{1}\in [t_{0},t_{0}+2+5/4]$, a contradiction. This shows (2). 
	
	Next we show (3). Suppose $\omega_{B}(t^{+})\ge 1/4$ for all $t_{1}\in [t_{0},t_{0}+6]$. For (3-1), first suppose $\Lambda_{w}(t_{0}+7/4)\in [0,1/4)$ ($b_{6}'$ or $b_{6}''$). If $\sigma_{w}(t_{0}+7/4)=0$, then $B+w$ has the local configuration $b_{8}''$ at time $t_{0}^{+}+2$, so by the previous case $a_{1}$ we have $\omega_{B}(t_{0}^{+}+2+1)<1/4$. If $\sigma_{w}(t_{0}+7/4)=1$, then $B+w$ has the local configurations $b_{8}$ or $b_{8}'''$ at time $t_{0}^{+}+2$ and $w$ recovers state 0 at some time $t_{1}\in (t_{0}+3+1/4, t_{0}+3+1/2]$ (at a relative phase $\in [0,1/4)$) so by (2), $B+w$ has the local configurations $b_{6}'$ or $b_{6}''$ at time $t_{0}^{+}+2+7/4$. But this implies $\sigma_{w}(t_{0}^{+}+2+7/4)=0$; in order to get excited during $(t_{1},t_{0}+2+7/4]$, since $B_{w}(t_{1})$ occurs, $\Lambda_{w}$ should decrease by $>1/4$ (by Proposition \ref{prop:E_v(t)_needs_enough_pull} (i)), so $\Lambda_{w}(t_{0}^{+}+2+7/4)\in [-1/4,0)$ for contrary. Hence by the previous case, we get $\omega_{B}(t_{0}^{+}+2+2+1)<1/4$, a contradiction. This shows (3-1).
	
	For (3-2), let $\Lambda_{w}(t_{0}+7/4)=1/4$. Then by (1) we have the transition $b_{5}\rightarrow b_{6}$ during $(t_{0}+5/4,t_{0}+7/4]$, which then leads to either $b_{8}$, $b_{8}'$ or $b_{8}''$ by time $t_{0}^{+}+2$. If the transition $b_{7}\rightarrow b_{8}''$ is used, then $B+w$ starts over the similar local dynamics with local configuration $b_{8}''\equiv a_{1}$ so by the previous case we have $\omega_{B}(t_{0}^{+}+2+1)<1/4$, a contradiction. Hence $B+w$ has $b_{8}$ or $b_{8}'$ at time $t_{0}^{+}+2$. Lastly, if $\sigma_{w}(t_{0}+7/4)=1$, then it is easy to see that $B+w$ has the transition $b_{8}(\subset b_{1}^{*})\rightarrow b_{2}\rightarrow d_{3}$ starting from time $t_{0}^{+}+2$. Thus by the previous case,  $\omega_{B}(t_{0}^{+}+4)=0$, a contradiction. This shows (3-2). 
	
	Lastly, we show (4). Suppose $\omega_{B}(t^{+})\ge 1/4$ for all $t\in [t_{0},t_{0}+8]$ and $\sigma_{w}(t_{0}+7/4)=2$. Then by (1)-(3) $\Lambda_{w}(t_{0}+7/4)\in [0,1/4]$ and $B+w$ has $b_{1}^{*}$ at time $t_{0}^{+}+2$. Restarting dynamics from time $t_{0}^{+}+2$, (3) implies $\sigma_{w}(t_{0}+2+7/4)\in \{0,2\}$. Suppose for contrary that $\sigma_{w}(t_{0}+2+7/4)=2$. Starting from times $t_{0}^{+}+2$ with local configuration $b_{1}^{*}$ on $B+w$, note that $w$ blinks and recovers state 0 at some time $s_{1}\in [t_{0}+2+1/4, t_{0}+2+1/2]$. Hence if we denote $t_{2}=\inf\{ t\ge s_{1}\,:\, \sigma_{w}(t_{2})=2 \}$, then $t_{2}\le t_{0}+2+7/4\le s_{1}+2/3$. Thus by Proposition \ref{prop:E_v(t)_needs_enough_pull} (ii), we have $\Lambda_{w}(t_{0}+2+7/4)\in [3/4,1)$. However, by (2) and (3) applied from time $t_{0}^{+}+2$, $\Lambda_{w}(t_{0}+2+7/4)\in [0,1/4]$ and $\sigma_{w}(t_{0}+2+7/4)\in \{0,2\}$. Hence we must have $\sigma_{w}(t_{0}+2+7/4)=0$, as desired. 
	
	On the other hand, suppose $\sigma_{w}(t_{0}+7/4)=0$. Then by (2) and (3), we have $\Lambda_{w}(t_{0}+7/4)=1/4$. If $\sigma_{w}(t_{0}^{+}+2)=1$, then $B+w$ has the transition $b_{8}(\subset b_{1}^{*})\rightarrow b_{5}$ starting from time $t_{0}^{+}+2$, and $w$ recovers state $0$ at time $t_{0}+2+3/2$ at relative phase $0$. Hence $\Lambda_{w}(t_{0}^{+}+2+7/4)\in [3/4,1)$ by Proposition \ref{prop:E_v(t)_needs_enough_pull} (i), so it leads to $b_{6}'''$ and we get $\omega_{B}(t_{0}+2+2+1)<1/4$ by a previous case, which is a contradiction. Thus we must have $\sigma_{w}(t_{0}^{+}+2)=0$. Then $B_{w}(t_{0}+2+1/2)$ occurs with $\sigma_{w}(t_{0}^{+}+2+1/2)=0$ and $\Lambda_{w}(t_{0}^{+}+2+1/2)=0$ (in $b_{2}$). Then by Proposition \ref{prop:E_v(t)_needs_enough_pull} (ii), it is impossible to have $\Lambda_{w}(t_{0}+2+7/4)\in [0,1/4]$ with $\sigma_{w}(t_{0}+2+7/4)=2$. Hence by (2) and (3) applied from time $t_{0}^{+}+2$, we conclude that $\sigma_{w}(t_{0}+2+7/4)=0$. This shows the assertion. 
\end{proof}

We finish this subsection by showing Lemma \ref{lemma:branchorbit_a}.

\begin{proof}[Proof of Lemma \ref{lemma:branchorbit_a}]
	Suppose that (i) does not hold. As in the beginning of proof of Lemma \ref{lemma:branchorbit_b}, we may assume that all nodes in $B$ have state 0 for all times $t\ge t_{0}$. Denote by $v\in V$ the center of $B$ and let $L$ be the set of leaves in $B$.  We may assume that $\Lambda_{v}(t_{0}^{+})=1/2$ and the two leaves in $B$ have relative phases $\lambda_{1}=1/2$ and $\lambda_{2}\in (-1/4,0)$.  Forward-tracking of the local dynamics starting from the local configuration in Figure \ref{fig:2branch_limiting_config} $a$ at time $t_{0}^{+}$ is developed in Figure \ref{fig:2branch_limit_a_pf}.

	Note that $\Lambda_{v}(t_{0}^{+}+1/2)\in [1/4,1/2]$ since $v$ can only be pulled by $w$ during $(t_{0},t_{0}+1/2]$. We claim the followings, from which Lemma \ref{lemma:branchorbit_a} (i)-(iii) follows immediately: 
	\vspace{0.1cm}
	\begin{description}
		\item{(1)} If $\Lambda_{w}(t_{0}^{+})\notin [0,\lambda_{2}+1/2]$, then $\omega_{B}(t_{1}^{+})<1/4$ for some $t_{1}\in [t_{0}+1,t_{0}+5/4]$ and the joint trajectory restricts on $T-L$ during $(t_{0},\infty)$. (Transitions $b_{1}\rightarrow b_{3}$ and $c_{1}\rightarrow c_{4}$ in Figure \ref{fig:2branch_limit_a_pf}). 
		\vspace{0.1cm}
		\item{(2)} If $\Lambda_{w}(t_{0}^{+})\in [0,1/4]$, then $B$ has the local configuration in Figure \ref{fig:2branch_limiting_config} $b$ at time $t_{0}^{+}+2$. (Transitions $d_{1}\rightarrow d_{7}$ in Figure \ref{fig:2branch_limit_a_pf}). 
		\vspace{0.1cm}
		\item{(3)} If $\Lambda_{w}(t_{0}^{+})\in (1/4,\lambda_{2}+1/2]$, then $B$ has the local configuration in Figure \ref{fig:2branch_limiting_config} $a$ (up to rotation) at time $t_{0}^{+}+3/2-\lambda_{2}$. (Transitions $a_{1}\rightarrow a_{8}$ in Figure \ref{fig:2branch_limit_a_pf}). 
	\end{description}
	\vspace{0.1cm}
	
	\begin{figure*}[h]
		\centering
		\includegraphics[width = 0.9 \linewidth]{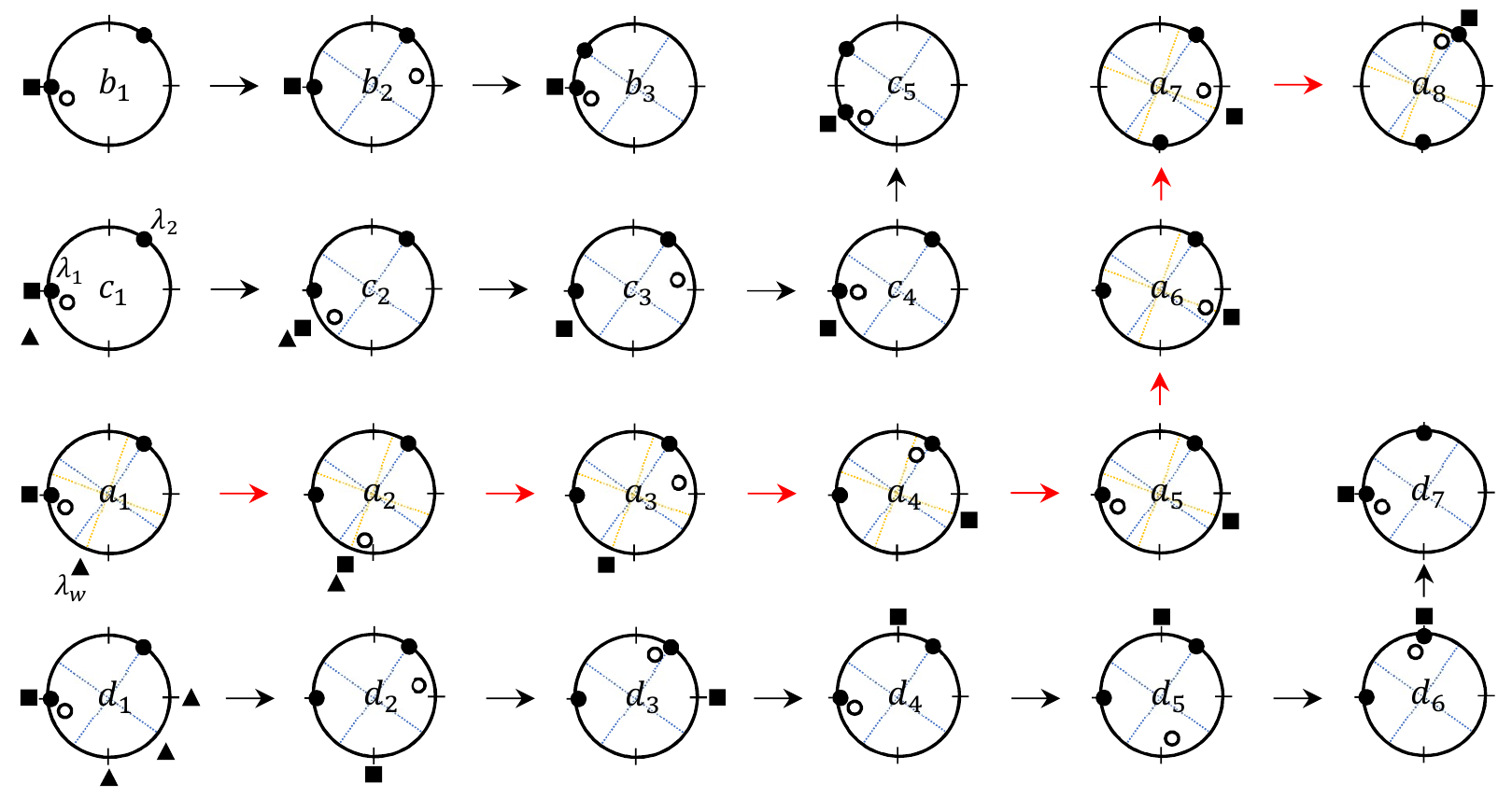}
		\vspace{-0.2cm}
		\caption{ (In relative circular representation) Forward-tracking of local dynamics on a 2-branch starting from the local configuration in Figure \ref{fig:2branch_limiting_config} $a$. $\blacksquare$, $\bullet$, $\blacktriangle$, and $\circ$ denote the relative phases of center, leaves, root, and the activator, respectively. 
		}
		\label{fig:2branch_limit_a_pf}
	\end{figure*}  
	
	For (1), observe that the hypothesis implies  $\Lambda_{v}(t_{0}^{+}+1/2)\in [\lambda_{2}+1/2, 1/2]$. If $\Lambda_{v}(t_{0}^{+}+1/2)=1/2$, then $\omega_{B}(t_{0}^{+}+1)<1/4$ (see the transition $b_{1}\rightarrow b_{3}$ in Figure \ref{fig:2branch_limit_a_pf}). If $\Lambda_{v}(t_{0}^{+}+1/2)\in [\lambda_{2}+1/2,1/2)$ (as in Figure \ref{fig:2branch_limit_a_pf} $c_{2}$), then $\omega_{B}(t_{1}^{+})<1/4$ for some $t_{1}=t_{0}+1-\lambda_{2}$ (see the transition $c_{1}\rightarrow c_{4}$ in Figure \ref{fig:2branch_limit_a_pf}). For (2), the hypothesis yields $\Lambda_{v}(t_{0}^{+}+1/2)=1/4$ (as in Figure \ref{fig:2branch_limit_a_pf} $d_{2}$) and it is easy to see that $B$ has the local configuration in Figure \ref{fig:2branch_limiting_config} $b$ at time $t_{0}^{+}+2$ (see the transition $d_{1}\rightarrow d_{7}$ in Figure \ref{fig:2branch_limit_a_pf}). For (3), the hypothesis gives $\Lambda_{v}(t_{0}^{+}+1/2)\in (1/4, \lambda_{2}+1/2)$ as in Figure \ref{fig:2branch_limit_a_pf} $a_{3}$. Then the rest of transition $a_{3}\rightarrow a_{8}$ is straightforward. Notice that the local configuration in Figure \ref{fig:2branch_limit_a_pf} $a_{8}$ is equivalent to $a_{1}$ up to rotation. This shows the claim.

	To show Lemma \ref{lemma:branchorbit_a} (iv), suppose that Lemma \ref{lemma:branchorbit_a} (ii) holds for $B$ and let $B'$ be any branch in $T$ rooted at $w$. Let $L'$ be the set of leaves in $B'$. It suffices to show that the dynamics on $T$ restrict on $T-L'$ or $T-B'$ during $(t_{0}+21,t_{0}+\mathtt{a}-6]$. First suppose $B'=B$.  Note that $v$ blinks exactly once during the transition $a_{1}\rightarrow a_{8}$ at $a_{6}$. Observe that in order for another transition $a_{1}\rightarrow a_{8}$ to occur from time $t_{1}^{+}$ where $t_{1}:=t_{0}+3/2+\lambda_{2}$ (from $a_{8}$ in Figure \ref{fig:2branch_limit_a_pf}), we need to have $\Lambda_{w}(t_{1}^{+})\in (\lambda_{2}-1/4,-1/4]$. By back-tracking, it then follows that $w$ is right to $v$ in $a_{6}$, so $w$ is not pulled by $v$ during the transition $a_{1}\rightarrow a_{8}$. Hence it follows that $w$ is never pulled by $v$ during $[t_{0},t_{0}+\mathtt{a}-4]$ so that the dynamics on $T$ restrict on $T-B$ during this period. 
	
	Second, suppose $B'\ne B$ is any branch in $T$ rooted at $w$. By Lemma \ref{lemma:branchwidth}, $\omega_{B'}(t)<1/4$ for some $t\in [0,t_{0}+18]$ implies that the joint trajectory restricts on $T-L'$ after time $t=t_{0}+21$. Hence we may assume $\omega_{B'}(t)\ge 1/4$ for all $t\in [0,t_{0}+18]$. Hence by Lemma \ref{lemma:branch_attraction}, $B'$ has one of the two local configurations in Figure \ref{fig:2branch_limiting_config} at least once during before time $E_{\Delta-1}\le t_{0}$. Then by Lemmas \ref{lemma:branchorbit_a} (i)-(iii) and \ref{lemma:branchorbit_b}, either of the transitions $a_{1}\rightarrow a_{8}\equiv a_{1}\rightarrow a_{8}$ or $b_{1}\rightarrow b_{1}\rightarrow b_{1}$ in Figure \ref{fig:2branch_limit} occurs on $B'+w$ during $t\in (t_{0},t_{0}+16]$. However, note that the local dynamics on $B+w$ induces that $w$ blinks once in every $<7/4$ second during $(t_{0},t_{0}+\mathtt{a}-2]$ but the latter transition $b_{1}\rightarrow b_{1}\rightarrow b_{1}$ on $B'+w$ requires $w$ to blink twice with the gap of 2 seconds some times during $(t_{0},t_{0}+16]$, which is impossible. Thus $B'+w$ must repeat the former transition $a_{1}\rightarrow a_{8}\equiv a_{1}$ during $(t_{0},t_{0}+16]$. 
	
	We show that in fact the same transition $a_{1}\rightarrow a_{8}$ repeats during the extended period  $(t_{0},t_{0}+\mathtt{a}-4]$. Then by the previous case, the joint trajectory restricts on $T-B'$ during $(t_{0},t_{0}+\mathtt{a}-6]$, as desired. Suppose not. By (1)-(3) and the observation in the previous paragraph, the local orbit on $B'+w$ must repeat the transition $a_{1}\rightarrow a_{8}$ and then use $b_{1}\rightarrow b_{3}$ or $c_{1}\rightarrow c_{5}$ in Figure \ref{fig:2branch_limit_a_pf} at least once during $(t_{0},t_{0}+\mathtt{a}-4+3/4)$. In fact, we claim that it has to be the latter transition. To see this, let $\tau_{w;0}\le t_{0}+\mathtt{a}-5$ be the time that $w$ blinks during the last transition $a_{1}\rightarrow a_{8}$ on $B'+w$, and let $\tau_{w;i}$ for $i\ge 1$ be the $i^{\text{th}}$ time that $w$ blinks after time $\tau_{w;0}$. Let $v'$ be the center of $B'$. At time $\tau_{w;0}^{+}$, $B+w$ has the local configuration in Figure \ref{fig:2branch_limit_a_pf} $a_{2}$ (after rotation, if necessary). Denote $\lambda_{w}=\Lambda_{w}(\tau_{w;0}^{+})$ and the relative phases of the two leaves in $B'$ at time $\tau_{w;0}^{+}$ by $\lambda_{1}'$ and $\lambda_{2}'$. Then we should have $\lambda_{1}'\in (\lambda_{w},\lambda_{w}+1/4)$ and $\lambda_{2}'\in [\lambda_{w}+1/2,\lambda_{1}'+1/2)\subset (-1/4,\lambda_{1}'+1/2)$. Recall that $w$ must blink at relative phase $\in [\lambda_{2}-1/4,3/4]$ at time $\tau_{w;1}$ in order to make another transition $a_{1}\rightarrow a_{8}$ on $B+w$. After the last transition $a_{1}\rightarrow a_{8}$ on $B'+w$, we have $\Lambda_{v'}=\lambda_{2}'$ as similarly in $a_{8}$. Thus $w$ pulls $v'$ at time $\tau_{w;1}$. Hence $B'+w$ must have the transition $c_{1}\rightarrow c_{5}$ starting at time $t_{1}^{+}$. 
	
	To finish the proof, recall that $\tau_{w;1}$ is the time that $w$ blinks in Figure \ref{fig:2branch_limit_a_pf} $c_{2}$ for $B'+w$, so $\Lambda_{v'}(\tau_{w;1}^{+})=\Lambda_{w}(\tau_{w;1}^{+})=\lambda_{w}$. Observe that the dynamics on $T$ restrict on $T-L'$ after time $\tau_{w;1}$ so $\Lambda_{v'}(t)\equiv \lambda_{w}$ during $(\tau_{w;1},\tau_{w;2}]$. On the other hand, $B+w$ undergoes another transition $a_{2}\rightarrow a_{8}$ from time $\tau_{w;1}^{+}$ to say time $\tau^{+} \le \tau_{w;2}^{+}$. Since another transition $a_{1}\rightarrow a_{8}$ continues from time $\tau^{+}$, we need to have $\Lambda_{w}(\tau^{+})\in [\Lambda_{v}(\tau^{+})-1/4, 3/4 ]$ (in Figure \ref{fig:2branch_limit_a_pf} $a_{8}$). But since $\lambda_{w}\in (1/4, \Lambda_{v}(\tau^{+})-1/2]$ and $v'$ keeps the relative phase $\lambda_{w}$ during $(\tau_{w;1},\tau_{w;2}]$, $w$ had to be pulled by $v'$ during $a_{2}\rightarrow a_{6}$ starting at time $\tau_{w;1}^{+}$, so this leads to a contradiction. This shows Lemma \ref{lemma:branchorbit_a} (iv).   
\end{proof}

\vspace{0.3cm}

\section{Application to distributed clock synchronization algorithms}
\label{section:application}

In this section, we discuss an implementation of the adaptive 4-coupling to an autonomous distributed systems and analyze its performance as a clock synchronization algorithm in various aspects. We borrow some standard terminologies from the distributed algorithms and clock synchronization literature (see e.g., \cite{lynch1996distributed}, \cite{gouda2005state}, and \cite{awerbuch2007time}). We define an \textit{autonomous distributed system}  as a \textit{state model}  augmented with a \textit{local clock } at each node. That is, it is a set of anonymous finite-state machines interacting over a communication network via broadcasting messages to all neighbors, whose execution is triggered upon receiving messages and periodic `beats' of its local clock.

More precisely, let $G=(V,E)$ be a finite simple graph with diameter $d$ and maximum degree $\Delta$. Fix $\epsilon>0$, which we call the \textit{time resolution} of the system. We view each node $v\in V$ as an identical \textit{automaton} with a finite state space $\mathcal{S}$ and a deterministic update rule $\mathcal{F}:\mathcal{S}\rightarrow \mathcal{S}$, which we also call a \textit{distributed algorithm}. We suppose each node $v$ has a \textit{local clock}, which \textit{beats} at times $t_{v}+k\epsilon$ for some $t_{v}\in [0,\epsilon)$ and for all $k\in \mathbb{N}_{0}$, independently of the update rule $\mathcal{F}$. We say the system is \textit{synchronous} if $t_{v}=t_{u}$ for all $u,v\in V$, and \textit{asynchronous} otherwise. The \textit{system configuration} at time $t$ is a map $\Pi_{\bullet}(t):V\rightarrow \mathcal{S}$, $\Pi_{v}(t)\in \mathcal{S}$. Received messages are stored in a first-in-first-out queue. Upon beats, each node updates its state by applying $\mathcal{F}$ to its current state. This defines the \textit{trajectory}$ (\Pi_{\bullet}(t))_{t\ge 0}$ of a given system configuration $\Pi_{\bullet}(0)$. 

In order to implement the adaptive 4-coupling as a distributed algorithm, we discretize the continuum factor $S^{1}\times S^{1}$ of its state space $\Omega=S^{1}\times S^{1}\times (\{1,3\}\times \mathbb{Z}_{2}\times \mathbb{Z}_{4})\times \mathbb{Z}_{3}$ into $\mathbb{Z}_{M}\times \mathbb{Z}_{M}$, where $M\ge 4$ is an integer multiple of $4$. The resulting distributed algorithm, which we call the adaptive 4-coupling modulo $M$ (A4C/$M$), is given below.

\begin{algorithm}\label{algorithm:A4C/M}
	\caption{The adaptive 4-coupling modulo $M$ (A4C/$M$)}\label{euclid}
	\begin{algorithmic}[1]
		\State \textbf{Variables:} 
		\State  \qquad $\phi_{v}\in \mathbb{Z}_{M}:$ phase of node $v$
		\State \qquad $\beta_{v}\in \mathbb{Z}_{M}:$ time lapse from the last blinking time of $v$ modulo $M$
		\State  \qquad $\mu_{v}=(\mu_{v}^{1},\mu_{v}^{2},\mu^{3}_{v})\in \{ 1,3 \}\times \mathbb{Z}_{2} \times \mathbb{Z}_{4}$: memory variable of node $v$
		\State \qquad $\sigma_{v}\in \mathbb{Z}_{3}$: state variable of node $v$  
		\State \qquad $\mathtt{pulse}_{v}\in \{0,1\}$: 1 if $v$ received a $\mathtt{pulse}$ since last beat and 0 otherwise
		\State $\textbf{Upon reciveing a pulse:}$ 	
		\State \qquad \textbf{Do} $\mathtt{pulse}_{v}=1$
		\State $\textbf{Upon a beat:}$ 
		\State  \qquad \textbf{Do} $\mu_{v}^{2}=1-\mathbf{1}(\mu_{v}^{2}=0)\mathbf{1}(0\le \beta_{v} < M/4)$
		\State \qquad \textbf{If} $0<\phi_{v}\le M/2$ and $\mathtt{pulse}_{v}=1$ \textbf{then} 
		\State \qquad \qquad \textbf{Do} $\phi_{v} = \phi_{v}\cdot \mathbf{1}(\sigma_{v}\ne 0) + (\phi_{v}-M/4) \cdot \mathbf{1}(M/4 < \phi_{v} \le M/2) \mathbf{1}(\sigma_{v}=0)$
		\State \qquad\qquad  \textbf{Do} $\sigma_{v} =\sigma_{v}+\mathbf{1}(\sigma_{v}=0)\cdot \mathbf{1}(\mu_{v}^{1}=\mu_{v}^{3})$
		\State \qquad \qquad \textbf{Do} $\mu_{v}^{1}=\mu_{v}^{1}$ and $\mu_{v}^{3}=[\mu_{v}^{3}+\mathbf{1}(\mu_{v}^{3}\ne 3)\cdot \mathbf{1}(\mu_{v}^{2}=1)]\cdot \mathbf{1}(\beta_{v}=M-1)$
		\State \qquad\textbf{If} $\phi_{v}=M-1$ \textbf{then}
		\State \qquad \qquad \textbf{Do} $\mu_{v}=(3-2\cdot \mathbf{1}(\sigma_{v}=2),0,0)$ 
		\State \qquad \qquad \textbf{Do} $\sigma_{v}=\sigma_{v}+\mathbf{1}(\sigma_{v}\ne 0)$
		\State \qquad \qquad  \textbf{Send} $\mathtt{pulse}=1$ to all neighbors
		\State \qquad \textbf{Do} $\beta_{v}=\beta_{v}\cdot \mathbf{1}(\phi_{v}\ne M-1)$ and $\phi_{v}=\phi_{v}+1$ and $\mathtt{pulse}_{v}=0$
	\end{algorithmic}
\end{algorithm}

For the sake of simplicity, we assume that each application of $\mathcal{F}$ at each node takes infinitesimal time. Hence for each $v\in V$, $\Pi_{v}(t)\equiv Const.$ during each interval $(\min (0, t_{v}+\epsilon k) , t_{v}+\epsilon (k+1) ]$, $k\in \{ -1,0,1,2\cdots \}$. Furthermore, we assume that all messages are delivered to neighbors without delay or loss.

Remark that the adaptive 4-coupling can be considered as the continuum limit  of these clock synchronization algorithms as $\epsilon\rightarrow 0$ with $M=4\lfloor 1/4\epsilon \rfloor$. Moreover, for large enough $M$, our main result for the continuum version (Theorem \ref{A4Ctreethm}) carries over to the discrete versions, as stated in the following theorem.

\begin{theorem}\label{thm:implementation}
	Consider an autonomous distributed system on a finite simple graph $G=(V,E)$. Then the distributed algorithm \textup{A4C/$M$} on $G$ has the following properties. 
	\vspace{0.1cm}
	\begin{description}
		\item[(i)] If there exists $t_{0}\ge 0$ such that 
		\begin{equation}\label{eq:A4C/M_sync_cond}
		\phi_{\bullet}(t+\epsilon)=\phi_{\bullet}(t)+1\mod M \quad \forall t\ge t_{0},
		\end{equation}
		then 
		\begin{equation}\label{eq:offset}
		\mathtt{offset}(t):=\max_{\substack{v\in V \\ u\in N(v)}} |\phi_{u}(t)-\phi_{v}(t) \,\, \textup{mod}\,\, M | \le 1 \quad \forall t\ge t_{0}+\epsilon (3M+1). 
		\end{equation}
		Moreover, if the system is synchronous, then $\mathtt{offset}(t)\equiv 0$ for all $t\ge t_{0}+\epsilon (3M+1)$. 
		\vspace{0.1cm}
		\item[(i)] $O(\log M)$ of memory per node is sufficient for implementation. Each node sends at most $\frac{\Delta}{\epsilon M}$ bits per unit time only using binary messages. 
		\vspace{0.1cm}
		\item[(iii)] If $G$ is a finite tree and $M\ge 64$, then there exists an absolute constant $C>0$ for which (\ref{eq:A4C/M_sync_cond}) holds for $t_{0}=C\epsilon M d $ regardless of initial configuration $\Pi_{\bullet}(0)$. 
	\end{description}
\end{theorem}

\begin{proof}[Sketch of proof]
	Say a node $v$ \textit{blinks} at time $t$ if $\phi_{v}(t)=M-1$ and $\phi_{v}(t^{+})=0$. 
	\vspace{0.1cm}
	\begin{description}
		\item{(i)} Note that the hypothesis implies $\sigma_{\bullet}(t)\equiv 0$ for all $t\ge t_{0}+\epsilon (2M+1)$ and also each node takes phase $M-1$ exactly once in every $M$ beats. For each $x\in V$ and $i\in \mathbb{N}$, denote by $\tau_{x;i}$ the $i^{\text{th}}$ time that $x$ blinks after time $t_{0}+\epsilon(2M+1)$. The hypothesis is then equivalent to  $\tau_{x;i+1}-\tau_{x;i}\equiv \epsilon$ for all $x\in V$ and $i\in \mathbb{N}$. Fix two neighboring nodes $u,v\in V$. By the hypothesis, $\phi_{u}(\tau_{v;i}^{+})\equiv r\mod M$ for all $i\in \mathbb{N}$ where $r\in \{M/2, M/2 + 1, \cdots  M-1\} \cup \{0\}$. By changing the role of $u$ and $v$, we conclude that $r\in \{ M-1,0 \}$. Furthermore, we must have $r=0$ if the system is synchronous. This holds for all adjacent pairs $(u,v)$. Since $\tau_{v;1}\le t_{0}+\epsilon(3M+1)$ for all $v\in V$, the assertion then follows easily. 
		
		\vspace{0.1cm}
		\item{(ii)} Each node needs to store five variables, whose size depends at most linearly on $M$ and no other parameters. Hence the A4C/$M$ can be implemented by using $O(\log M)$ memory per node. The second part follows from noting that each node only sends a binary pulse signal to each neighbor upon blinks, which occurs at most once in every $\epsilon M$ since phase updates can only be inhibited. 
		\vspace{0.1cm}
		
		\item{(iii)} Now suppose $G$ is a tree. We wish to show that analogs of Lemmas \ref{lemma:branchwidth} and \ref{key} hold for the discretized version. An entirely similar argument shows that a version of Lemma \ref{lemma:branchwidth} holds for the discretized version, which requires $\omega_{B}(t_{0})<1/4-2/M$ instead of $\omega_{B}(t_{0})<1/4$. Hence it suffices to show that Lemma \ref{key} can be strengthened to guarantee $\omega_{B}(t_{0})<1/4-2/M$ for some $t_{0}=O(\epsilon M)$. We claim that this could be done for all $M=4k$ for $k\ge 14$, whose verification is left to the reader. 
	\end{description}
\end{proof}

\begin{customremark}{8.2}
	To emphasize possible advantages of our algorithm implied by the above theorem, we describe the following application scenario. The modulo $M$ phase variable $\phi_{v}$ can be used to extend the local time frame from period $\epsilon$ to $\epsilon M$, enabling one to program the behavior of each node over a longer period than what is provided by the local clock. Given that each phase $\phi_{v}$ keep incrementing by 1 mod $M$ with a small and bounded offset, this provides a good global time frame for collective computations over the system. 
	
	For example, imagine a wireless sensor network trying to transmit a binary string, say $a_{1}a_{2}\cdots a_{n}$ for some $n<M-1$, to an observer at distance. If the local times are perfectly synchronized, then we can simply let each node transmit bit $a_{i}$ at phase $i\in \mathbb{Z}_{M}$, possibly a period at phase $M-1$ to mark the end of each string. However, if the local times are not well synchronized, then the nodes could transmit different bits at each time so the observer could find it hard to determine what string is being transmitted. 
	
	Here is a possible strategy assuming the offset is bounded by 1. Then the time difference along any edge can be at most $\epsilon$, so the global time difference is at most $\epsilon d$. If $M>4(n+1)d$, then we may partition $\mathbb{Z}_{M}$ into $M/4d$ intervals of equal length $4d$. Then we can let each node to transmit bit $a_{i}$ at phase $4di$. Now all nodes transmit bit $a_{i}$ during an interval of length $\le 2\epsilon d$, and any consecutive such intervals are separated by an interval of length $\ge 2\epsilon d$ where all nodes are idle. Thus the desired binary string could be effectively transmitted to the observer. Hence, in the scaling $M=O(1/\epsilon)$, the maximal length of binary strings which can be collectively transmitted during a time window of length $O(1)$ increases linearly in $1/\epsilon$.

	Theorem \ref{thm:implementation} (i) says the A4C/$M$ is a \textit{clock synchronization algorithm}, that is, if the phases of all nodes increment by 1 modulo $M$, then it is guaranteed to have small bounded offset. Theorem \ref{thm:implementation} (ii) is about efficiency of the A4C/$M$. Its first part says the algorithm is \textit{scalable}, that is, it can be implemented with a constant memory per node regardless of the communication network $G$. Its second part implies that the algorithm uses bounded amount of information exchange to operate as the time resolution $\epsilon$ goes to zero, under the scaling $M=O(1/\epsilon)$. Lastly, under the same scaling, Theorem \ref{thm:implementation} (iii) together with (ii) guarantees the convergence of A4C/$M$ on finite trees from all initial configurations in time $O(d)$. Such a global convergence of a distributed algorithm is called \textit{self-stabilization} \cite{dijkstra1982self}. $\blacktriangle$   
\end{customremark}

\vspace{0.1cm}

Next, we extend the self-stabilization of our clock synchronization algorithm to general graphs. Our approach follows the popular paradigm of first designing distributed algorithms for trees and then combining with a spanning tree algorithm. There is an extensive literature on distributed spanning tree construction (see e.g., the survey by G\"artner \cite{gartner2003survey}). Itkis and Levin \cite{itkis1994fast} proposed a scalable randomized (Las Vegas) distributed algorithm, which self-stabilizes on arbitrary graphs with probability 1 and the worst case running time has expectation of order $O(d^{5}\log|V|)$. However, their hypothesis was that each node has a \textit{pointer} to each of its neighbors so that it can distinguish messages from different neighbors; In the end of algorithm, each node $v$ ends up with a unique pointer $\mathtt{root}_{v}$ towards its root and a set $\mathtt{Children}_{v}$ of pointers toward its children. This is certainly not possible in our anonymous system, so we first need to construct `local identifiers' in a scalable and self-stabilizing way.

The use of local pointers can be justified if we adopt the standard assumption that nodes have distinct ID. That is, we now let each node append $(\mathtt{targetID},\mathtt{sourceID})$ of additional field to each message they broadcast; then receiver may acknowledge messages of $\mathtt{targetID}$ matching its own ID, and also can distinguish messages from different neighbors. However, since there are at least $|V|$ ID values, this requires $O(\log |V|)$ memory per node to process the ID field. In fact, the ID's only need to be distinct among neighbors and such locally unique identifiers can be implemented with $O(\log \Delta)$ memory per node. 

For this matter, it is enough to construct a coloring of given network where no two nodes within graph distance 2 have the same color. Namely, call a map $\iota :V\rightarrow \mathbb{N}$  \textit{distance $\le 2$ coloring} if any restriction on a $2$-ball is injective, i.e., $\iota(u)\ne \iota(v)$ whenever $\text{dist}(u,v)\le 2$, where $\text{dist}$ denotes the shortest path distance in $G$. Notice that $\iota$ is a local identifier iff it is a proper vertex coloring on the \textit{supergraph} $G^{2}$, which is obtained by adding edges between non-adjacent nodes in $G$ with distance $\le 2$. Since $G^{2}$ has maximum degree at most $\Delta^{2}-1$, it follows that $G$ admits a local identifier assuming $\Delta^{2}$ distinct values.

\begin{algorithm}
	\caption{Randomized distance $\le 2$ coloring}\label{euclid}
	\begin{algorithmic}[1]
		\State{$\textbf{Variables:}$} 
		\State \qquad $\mathtt{R}_{v}\in \{0,1\cdots,\Delta^{2}\}:$ color of node $v$
		\State \qquad $\mathtt{NR}_{v}\subseteq \{0,1\cdots,\Delta^{2}\}:$ set of colors of neighbors of $v$ 
		\State{$\textbf{Upon a beat:}$}
		\State{\qquad \textbf{Do} Receive $\mathtt{NR}_{u}$ from all $u\in N(v)$}
		\State{\qquad \textbf{If} $\mathtt{R}_{v} < \max \left( \{0,1,\cdots,\Delta^{2}\} \setminus \bigcup_{u\in N(v)}\mathtt{NR}_{u}\right)$ \\
			\qquad \qquad \textbf{If} $\text{rand}\{0,1\}=1$ \textbf{then}\\
			\qquad \qquad \textbf{Do}  $\mathtt{R}_{v} = \max\left( \{0,1,\cdots,\Delta^{2}\} \setminus \bigcup_{u\in N(v)}\mathtt{NR}_{u}\right)$} 
		\State{\qquad \textbf{Do} Receive $\mathtt{R}_{u}$ from all $u\in N(v)$ and update $\mathtt{NR}_{v}$} 
	\end{algorithmic}
\end{algorithm}

In particular, a distance $\le 2$ coloring can be implemented in any distributed system by paying $O(\log \Delta )$ memory per node. Gradinariu and Tixeuil \cite{gradinariu2000self} obtained a self-stabilizing Las Vegas distributed algorithm for the usual vertex (i.e., distance 1) coloring using at most $\Delta+1$ colors, assuming $O(\log \Delta)$ memory per node with expected worst case running time of $O(\Delta \log |V|)$. A minor modification of their algorithm, which is given above, yields a distributed distance $\le 2$ coloring construction as given above. Clearly the algorithm requires $O(\log \Delta)$ memory per node. Also, a similar analysis for the distance 1 version shows that its expected worst case running time is $O(\Delta^{2} \log |V|)$.

Now let $\mathcal{A}$ denote the algorithm obtained by composing the distance $\le 2$ coloring algorithm, spanning tree algorithm, and the A4C/$M$ for $M\ge 64$.

\begin{corollary}\label{thm:clocksync}
	Consider an autonomous distributed system on an arbitrary finite simple graph $G=(V,E)$ with diameter $d$ and maximum degree $\Delta$. Then the composite algorithm $\mathcal{A}$ has the following properties. 
	\vspace{0.1cm}
	\begin{description}
		\item[(i)] $\mathcal{A}$ can be implemented with $O(\log M\Delta)$ memory per node. 
		\vspace{0.1cm}
		\item[(ii)] Define the worst case running time $\tau_{G}$ of $\mathcal{A}$ by
		\begin{equation}
		\tau_{G} = \max_{\Pi_{\bullet}(0):V\rightarrow \mathcal{S}} \inf \{ t_{0}\ge 0 \,:\, \text{(\ref{eq:A4C/M_sync_cond}) holds}  \},
		\end{equation}
		where the infimum is taken over all sample paths of the two randomized algorithms. Then $\mathbb{E}[\tau_{G}] = O(\epsilon M |V|+(d^{5}+\Delta^{2})\log |V|)$. In particular, $\mathcal{A}$ is self-stabilizing on arbitrary finite simple connected graphs with probability 1.
	\end{description}
\end{corollary}

\begin{proof}
	(i) follows by taking maximum of memory requirements of each of the three algorithms. Now we show (ii). On the first layer, the randomized distance $\le 2$ algorithm presented before constructs local identifiers with worst case expected running time $O(\Delta^{2}\log |V|)$. On the second layer, the randomized distributed spanning tree algorithm of Itkis and Levin \cite{itkis1994fast} constructs a time series of subgraphs $(H_{t})_{t\ge 0}$ of $G$, which converges, almost surely, to some spanning tree $T\subseteq G$ of diameter $d'\le |V|$ in expected $O(d^{5}\log |V|)$ time in the worst case. On the top layer, each node runs the A4C/$M$ on the time series $(H_{t})_{t\ge 0}$, meaning that in every pulse they send, nodes specify their root and children as target field, and they acknowledge messages containing target field matching their current local ID. After the first two layer converges, (\ref{eq:A4C/M_sync_cond}) holds for some $t_{0}=O(\epsilon M d')$ by Theorem \ref{thm:implementation} (iii). Noting that $d'\le |V|$, linearly of expectation gives the assertion. 
\end{proof}

\vspace{0.3cm}
\section{Concluding Remarks}
\label{Section:Concluding remarks}
\vspace{0.2cm}

\textbf{On complete graphs.} While the very first almost sure self-stabilizing property of PCOs were established on complete graphs in \cite{mirollo1990synchronization}, we established self-stabilization of our pulse-coupling on the topology of the other extreme, namely trees. Then it is natural to ask whether our models are also self-stabilizing on complete graphs. Interestingly, the answer is negative. For instance, consider a phase configuration $\Lambda_{0}$ on $K_{n}$ for $n\ge 3$ where only three phases $0,1/4$, and $5/8$ are occupied. Then the orbit under (adaptive) 4-coupling is non-synchronizing with period of 5 seconds. Such a bad configuration does depend highly on symmetry, so one could hope that there are not too many of them so that we can still have almost sure synchronization with respect to some probability measure on the initial configurations. However, unlike other traditional PCO models, our pulse-coupling uses a very strong discretization property, merging arbitrary number of nearby phases into one single phase. Hence it is not  entirely obvious whether the (adaptive) 4-coupling is almost sure self-stabilizing on complete graphs. 
\vspace{0.3cm}

\textbf{Propagation delay.} Hong et al. \cite{hong2005scalable} gave a detailed discussion on applying existing theories of PCO systems to actual distributed control protocol in wireless sensor networks. One of the main concern was to relax the ideal assumption of zero propagation delay of pulsing signal; since pulse-couplings are based on sparsely triggered events compared to processing times in each unit, they are the main source of synchronization error. As many authors pointed out through variants of the width lemma (Lemma \ref{widthlemma}), pulse-couplings are highly robust against propagation delays once the width becomes $<1/2$. Our tree theorems (Theorem \ref{4treethm} and \ref{A4Ctreethm}) basically follows from the branch width lemma (Lemma \ref{lemma:branchwidth}) and the key lemma (Lemma \ref{key}). It is not hard to see that the branch width lemma still holds even if we allow propagation delays of $<1/4$ second on the edges. Furthermore, all other propositions still remain valid if we assume identical propagation delays on the edges; hence our main results would still hold, with order $O(d)$ of maximum time dispersion after local synchronization. We claim that this is still true for non-identical propagation delays of magnitude $<1/4$. Verification of this claim is left for a future work. 
\vspace{0.3cm}

\textbf{Non-identical or time-varying frequencies.} One of the most interesting open question seems to be extending our theory to non-identical or time-varying frequencies for each local clock. That is, an ideal clock has constant speed 1, but in general a clock may have any fixed or time-varying speed $\omega_{v}(t)\in (1-\rho,1+\rho)$ for some small constant $\rho\in (0,1)$, which is called the `skew' of the system. Are some versions of Lemma \ref{lemma:branchwidth} and \ref{key} still valid in this asynchronous setting? Even the validity of the width lemma seems not entirely clear, as the skew can accumulate (unlike propagation delay) over a long period and stretch the small width beyond the threshold. Some (diffusive) couplings in the literature achieves frequency consensus by assuming that the nodes can read frequencies of neighbors \cite{mallada2011distributed}, but such protocol requires unbounded memory per node, unless the frequencies are globally bounded and such bound is known to the nodes a priori. The famous Kuramoto model only uses neighboring phase information and obtains frequency consensus on complete graphs \cite{strogatz2000kuramoto}, but in a non-scalable way. Hence it would be interesting to see if the adaptive 4-coupling (or its variant) could achieve tight phase synchrony in spite of nonzero clock skew.   
\vspace{0.2cm}

\textbf{ Fast minimum diameter spanning tree algorithm and improving the bound on $\mathbb{E}[\tau_{G}]$.} The biggest bottleneck for the running time of $\mathcal{A}$ comes from the unknown diameter $d'$ of the random spanning tree. The spanning tree algorithm by Itkis and Levin \cite{itkis1994fast} constructs a depth-first search tree rooted at a randomly chosen `center'; hence there is no guarantee that the so constructed spanning tree would have diameter close to the optimal diameter $d$ of the underlying graph $G$. However, in some special cases we would expect a depth-first search tree centered at a randomly chosen node should have expected diameter of order $O(\log |V|)$ in some special cases where $G$ has a weak expansion property. In general, known algorithms of finding minimum diameter spanning tree (MDST) runs too slow or requires too much memory per node for our purpose (see, e.g., 
\cite{bui2004distributed}). If there is a sub-linear time sub-linear diameter spanning tree algorithm which uses $O(\log \Delta)$ memory per node, we can improve the bound on $\mathbb{E}[\tau_{G}]$ to $o(|V|)$ while keeping the same efficiency and scalability.  


\section*{Acknowledgement}

The author give special thanks to David Sivakoff, Steven Strogatz, and Osama Khalil for helpful discussions. Valuable comments from the referees helped improving this work substantially. Finally, the author appreciates Joowon Lee for her warm support.

\appendix
\renewcommand*{\appendixname}{}
\numberwithin{equation}{section}
\numberwithin{figure}{section}
\numberwithin{table}{section}

\setcounter{MaxMatrixCols}{20}

\end{document}